\documentclass[11pt]{article}
 
\usepackage[margin=1in]{geometry} 
\usepackage{amsmath,amsthm,amssymb}
\usepackage{mathtools}
\usepackage{tikz-cd}
\usepackage{ upgreek }
\usepackage{color}
\usepackage{fancyhdr}
\usepackage[title]{appendix}
\usepackage{afterpage}
\usepackage{tcolorbox}
\usepackage{hyperref}
\usepackage{titlesec}
\usepackage{enumerate}
\usepackage[shortlabels]{enumitem}
\usepackage{setspace}
\usepackage{minitoc}
\usepackage{ifthen}
\usetikzlibrary{arrows}
\usetikzlibrary{decorations.markings}
\usepackage{tikz}
\usepackage{caption}
\usepackage{subcaption}
\usepackage{wrapfig}
\usepackage{pgfplots}
\usepackage[latin1]{inputenc}
\usepackage{listings}
\usepackage{float}
\usepackage[normalem]{ulem}
\usepackage{cancel}
 
\newcommand{\N}{\mathbb{N}}
\newcommand{\Z}{\mathbb{Z}}

\newcommand{\R}{\mathbb{R}}
\newcommand{\E}{\mathbb{E}}
\newcommand\Tstrut{\rule{0pt}{2.6ex}}  
\newcommand\Bstrut{\rule[-0.9ex]{0pt}{0pt}}

\newcommand{\ADD}[1]{{\color{black}
			#1}}

\newtheorem{thm}{Theorem}[section]
\newtheorem{theorem}[thm]{Theorem}
\newtheorem{corollary}[thm]{Corollary}
\newtheorem{lemma}[thm]{Lemma}  
\newtheorem{proposition}[thm]{Proposition}

\newtheorem{definition}[thm]{Definition}
\newtheorem{hypothesis}[thm]{Hypothesis}

\theoremstyle{remark}
\newtheorem{remark}[thm]{Remark}
\newtheorem{example}[thm]{Example}
\newtheorem*{lemma*}{Lemma}

\numberwithin{equation}{section}

\usepackage{subfiles} % Best loaded last in the preamble

\begin{document}
%\setboolean{printBibInSubfiles}{false}

\title{The {M}aslov index, degenerate crossings and the stability of pulse solutions to the {S}wift-{H}ohenberg equation}

\author{Margaret Beck\footnote{Department of Mathematics and Statistics, Boston University; Boston, MA, USA; mabeck@bu.edu}, Jonathan Jaquette\footnote{Department of Mathematics and Statistics, Boston University; Boston, MA, USA} \footnote{Department of Mathematical Sciences, New Jersey Institute of Technology; Newark, NJ, USA; jcj@njit.edu} 
\footnote{J.J.'s ORCID: 0000-0001-8380-3148}, Hannah Pieper\footnote{Department of Mathematics and Statistics, Boston University; Boston, MA, USA; hpieper@bu.edu} \footnote{Corresponding Author.}}

\date{\today}
\maketitle

\begin{abstract}
In the scalar Swift-Hohenberg equation with quadratic-cubic nonlinearity, it is known that symmetric pulse solutions exist for certain parameter regions. In this paper we develop a method to determine the spectral stability of these solutions   by associating a Maslov index to them. This requires extending the method of computing the Maslov index introduced by Robbin and Salamon [\emph{Topology} 32, no.4 (1993): 827-844] to so-called degenerate crossings. We extend their formulation of the Maslov index to degenerate crossings of general order in the case where the intersection is fully degenerate, meaning that if the dimension of the intersection is $k$, then each of the $k$ crossings is a degenerate one. We then argue that, in this case, this index coincides with the number of unstable eigenvalues for the linearized evolution equation. Furthermore, we develop a numerical method to compute the Maslov index associated to symmetric pulse solutions. Finally, we consider several solutions to the Swift-Hohenberg equation and use our method to characterize their stability.  
\end{abstract}


\noindent {\bf Acknowledgments:} This material is based upon work supported by the National Science Foundation under Award No. DMS-1907923 and DMS-2205434.

\tableofcontents

%%%%%%%%%%%%%%%%%%%%%%%%%%%%%%%%%%%%%%%%%%%%%%%%%%%%%%%%%%%%%%%%%%%%%%%%%%%%%%%%%%%%%%%%%%%%%%%
%%%%%%%%%%%%%%%%%%%%%%%%%%%%%%%%%%%%%%%%%%%%%%%%%%%%%%%%%%%%%%%%%%%%%%%%%%%%%%%%%%%%%%%%%%%%%%%
%%%%%%%%%%%%%%%%%%%%%%%%%%%%%%%%%%%%%%%%%%%%%%%%%%%%%%%%%%%%%%%%%%%%%%%%%%%%%%%%%%%%%%%%%%%%%%%
%%%%%%%%%%%%%%%%%%%%%%%%%%%%%%%%%%%%%%%%%%%%%%%%%%%%%%%%%%%%%%%%%%%%%%%%%%%%%%%%%%%%%%%%%%%%%%%

%\subfile{Introduction}

\section{Introduction}\label{S:intro}
% \subsection{Notes to Authors}
% There have been some notation changes. Hopefully the notation is consistent and I haven't missed anything, but the recent changes made are: 
% \begin{enumerate}[a)]
% \item The symplectic form has been changed from $\tilde J = \begin{pmatrix} 0 & -I_n \\ I_n & 0
% \end{pmatrix}$ to $J = \begin{pmatrix} 0 & I_n \\ -I_n & 0
% \end{pmatrix}$.
% \item When representing a Lagrangian subspace as a graph, I have tried to keep the notation consistent as $A: \ell(t_0) \to \mathcal W$ with $\mathcal W \cap \ell(t_0) = \{0\}$ such that $v + A(t)v \in \ell(t)$ for $v \in \ell(t_0)$. Formerly, the notation was $v + w(t) \in \ell(t)$. I think I have changed all instances of this. 
% \item Formerly, the graph formulation of $\ell(t)$ was written as $(x, A(t)x)$ for $x \in \R^n$ and $A \in \R^{n \times n}$. In order to simplify coordinate changes, I have changed the graph formulation to $v + A(t)v$ for $v \in \ell(t_0)$ and $A(t) \in \R^{2n \times 2n}$.
% \item Changed arbitrary matricies in discussion and technical results to $B$ from $A$ since $A$ is already in use as a graph operator. 
% \item The order of the degeneracy is $j$. The dimension of the intersection is $k$. 
% \item The results for general quadratic forms (outside of the context of the Maslov index) and denoted with $R$. 
% \item Claims have been removed. 
% \end{enumerate}

% \textbf{To Do:}
% \begin{enumerate}[a)]
% \item 
% \end{enumerate}

% -------------------------------------

Coherent structures, such as wave trains, fronts and pulses, often occur in spatiotemporal systems found in nature or laboratory \ADD{settings.} Mathematically, these objects can be studied as the solutions of partial differential equations posed on infinite domains. Oftentimes, one wants to understand the long term behavior of these models to gain insight into observable physical phenomena. To understand the long term behavior of these models, one often begins by seeking to understand the stability of these coherent structures. One can study spectral, linear and nonlinear (in)stabilities of such solutions. In particular, spectral stability is determined by whether the linearization of the PDE about a given solution has spectrum with positive real part. The spectrum can be divided into two disjoint sets: the essential and point spectrum. The essential spectrum is often relatively easy to compute, so the bulk of the work in determining spectral stability lies in looking for unstable eigenvalues. 

This work seeks to understand the (in)stability of pulse solutions to the Swift-Hohenberg equation. The Swift-Hohenberg equation \cite{CrossHohenberg,SwiftHohenberg} for $u: \R \times \R^+ \to \R$ is defined as 
\begin{equation}\label{eq: swift hohenberg} 
u_t = -(1+\partial_x^2)^2u + f(u),
\end{equation}
where the nonlinearity $f$ is typically chosen to be a parameter-dependent polynomial. For real parameters $\mu$ and $\nu$, we will work with the nonlinearity
\begin{equation}\label{eq: f - nonlinearity of SH}
f(u) =\nu u^2 - u^3 - \mu u.
\end{equation}
This equation was derived by Swift and Hohenberg in 1977 to study thermal fluctuations of a fluid near the Rayleigh-Benard convective instability, but has since been used as a general model with which to study pattern-forming behavior. For the Swift-Hohenberg equation, spectral stability is necessary for both linear and nonlinear stability \cite{henry81}. Therefore, if we can show that a pulse solution is spectrally unstable, then it is linearly and nonlinearly unstable as well.

Previous work has addressed the existence of pulse solutions to the 1D Swift-Hohenberg equation using numerical methods \cite{burkeknobloch07-2,burkeknobloch07,burkeknobloch07-3} and spectral methods have been used to compute the eigenvalues of these solutions \cite{burkeknobloch06} for certain parameter regimes. The work of \cite{makridesSanstede19,beck09} rigorously proved that a pulse solution can be formed by concatenating front and back solutions. Furthermore, the work of \cite{makridesSanstede19} argues that these pulse solutions are stable if and only if the front and back solutions are stable. 

We seek to round out this theory by providing a framework that can be used to count the eigenvalues of these pulse solutions. This framework is general in the sense that it can be used to characterize both spectral stability and spectral instability of such solutions.  Additionally, this framework applies to a wide range of parameter values (see Hypothesis \ref{hyp: nonlinearity}). Furthermore, this paper lays the groundwork for our forthcoming paper that uses validated numerics to provide a mathematical proof of both the existence and stability of these pulse solutions. 

Our strategy for proving the (in)stability of pulse solutions to the Swift-Hohenberg equation relies on the Maslov index. We will detect instabilities in the point spectrum by building directly on the results in \cite{BCJ18}, which allow one to count unstable eigenvalues by instead counting objects referred to as conjugate points. This extension is nontrivial due to the presence of nonregular (degenerate) crossings, which we will say more about shortly. In order to address this degeneracy, we will generalize the crossing form and main results of \cite{robbinsalamon} so that they apply to degenerate crossings of any order,  with the caveat that the intersection is fully degenerate, meaning that if the dimension of the intersection is $k$, then each of the $k$ crossings is a degenerate one. We will describe how our contribution to Maslov theory fits within the existing body of work in Section \ref{Ss: mathematical preliminaries}. Lastly, we numerically compute conjugate points and estimate unstable eigenvalues for several example pulse solutions using Fourier methods as a proof of concept. 

For an open parameter region, the Swift-Hohenberg equation admits symmetric pulse solutions. In particular, for certain parameter values, one can observe interesting behavior called homoclinic snaking in which there are an infinite number of symmetric (and asymmetric) pulse solutions \cite{beck09}. \ADD{We will look at two sets of parameter values.} \ADD{For the parameter values $(\nu, \mu) = (1.6, 0.05)$, which lie outside of the snaking region, there are two symmetric pulse solutions.} We also consider the parameter values $(\nu, \mu) = (1.6, 0.20)$ that lie within the snaking region. The branches of symmetric pulse solutions are approximated by \eqref{eq: BK normal form} with $\phi = 0, \pi$. We denote each pulse solution as $\varphi_\phi(x; \nu, \mu)$ so we can specify the corresponding phase condition $\phi$ in \eqref{eq: BK normal form} and the parameter values $\nu$ and $\mu$. Example parameter values and solutions are depicted in Figure \ref{fig:introduction results}. The numbers of unstable eigenvalues for the example solutions in Figure \ref{fig:introduction results} are given in Table \ref{table: eigenvalues}.

\begin{table}[H]
\begin{center}
\begin{tabular}{||c c c||} 
 \hline
 Symmetric Pulse  & \# Unstable Eigenvalues & \# Conjugate Points \Tstrut \\ 
 $\varphi_\phi(x; \nu, \mu)$ & &  \Bstrut \\
 \hline
 \hline
 $\varphi_0(x; 1.6, 0.05)$   & 1 & 1 \Tstrut\Bstrut \\
 \hline
 $\varphi_\pi(x; 1.6, 0.05)$ & 2 & 2 \Tstrut\Bstrut \\
 \hline
 $\varphi_0(x; 1.6, 0.2)$ & \text{None} & \text{None} \Tstrut\Bstrut \\
 \hline
 \hline
\end{tabular}
\end{center}
 \caption{Number of unstable eigenvalues and conjugate points for the example symmetric pulse solutions. }\label{table: eigenvalues}
\end{table}

\begin{figure}[H]
     \centering
      \begin{subfigure}[b]{0.3\textwidth}
     \centering
     \includegraphics[width=\textwidth]{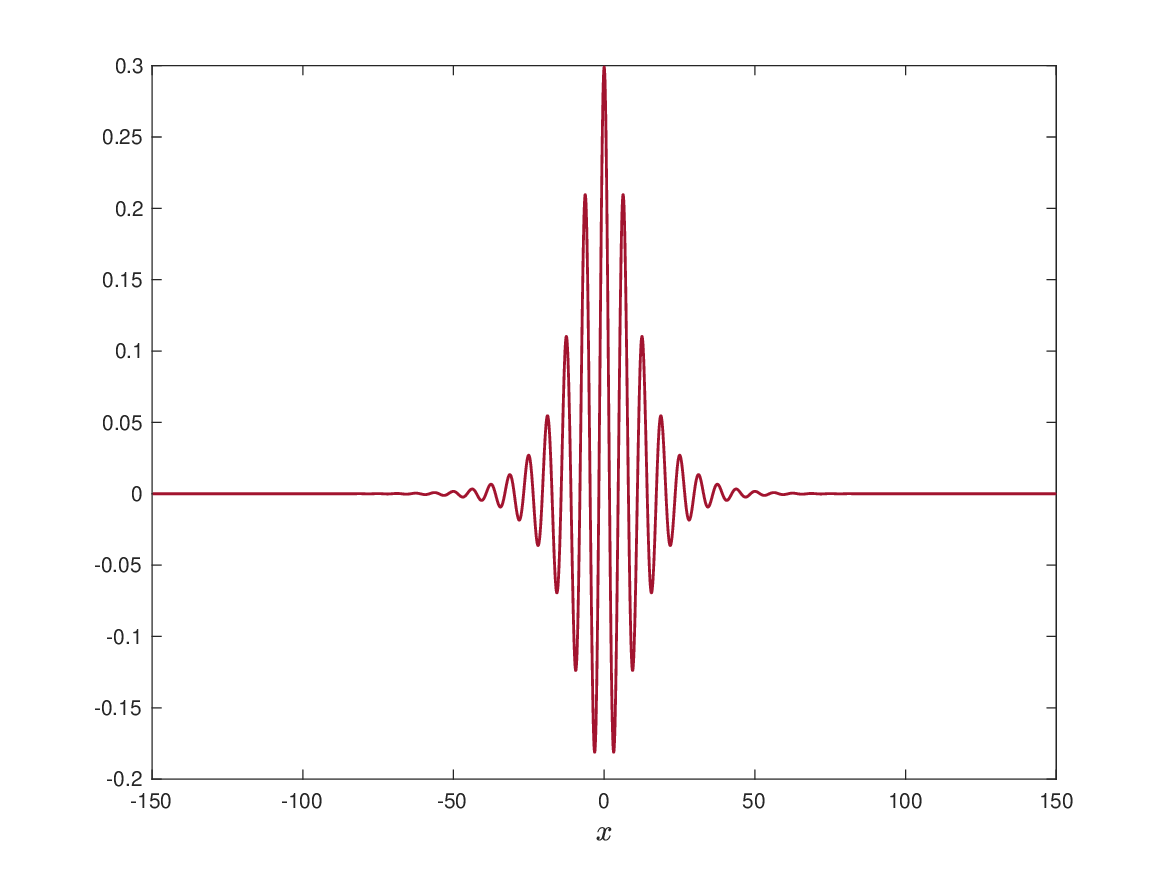}
     \caption{$\varphi_0(x;1.6, 0.05)$}
 	\end{subfigure}
     \begin{subfigure}[b]{0.3\textwidth}
         \centering
         \includegraphics[width=\textwidth]{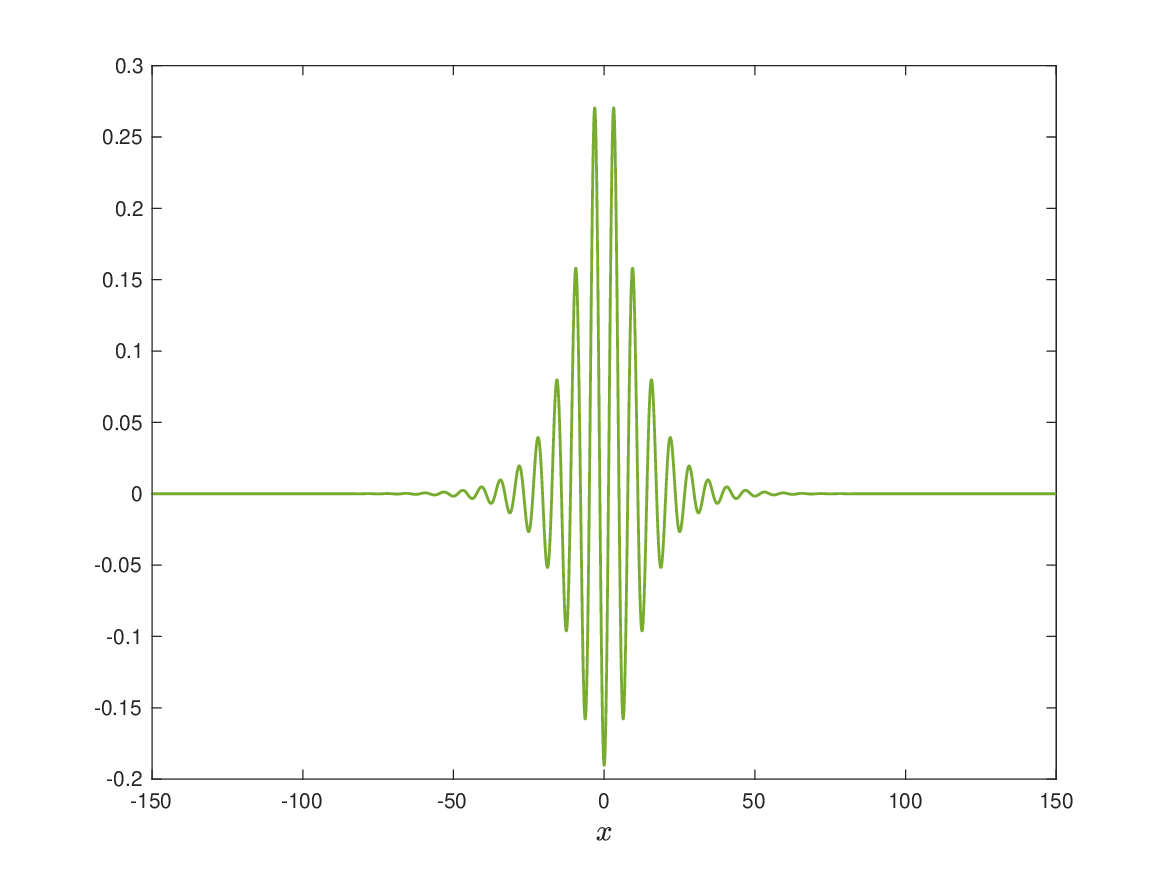}
         \caption{$\varphi_\pi(x;1.6, 0.05)$}
     \end{subfigure}
     \begin{subfigure}[b]{0.3\textwidth}
         \centering
         \includegraphics[width=\textwidth]{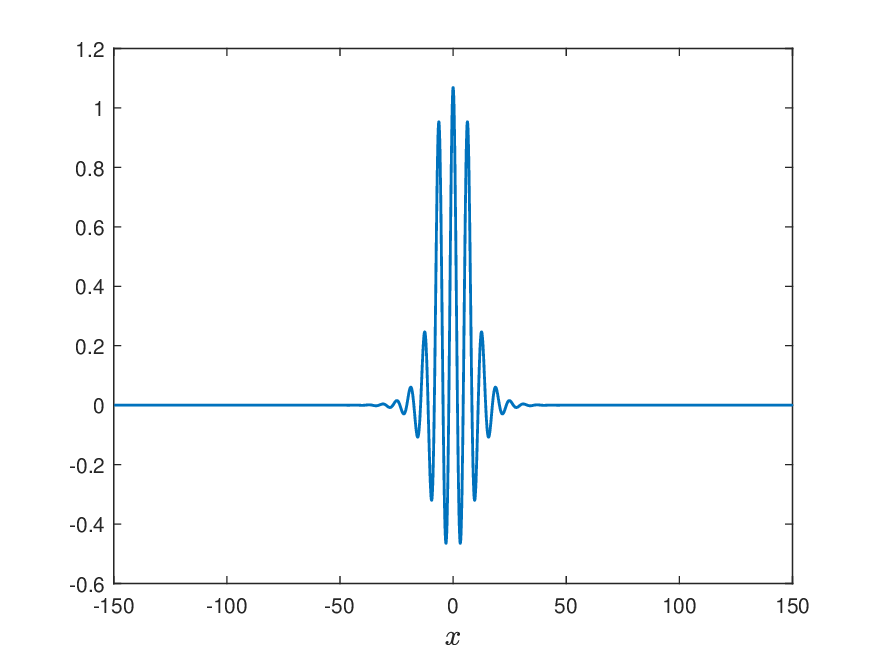}
         \caption{$\varphi_0(x; 1.6, 0.2)$}
     \end{subfigure}
    \caption{Example symmetric pulse profiles. }
        \label{fig:introduction results}
\end{figure}

In order to say more about the works we will build on \cite{BCJ18, MBJJ21} and to describe our results in more detail, we now introduce some assumptions and basic facts about our setting. We will also introduce relevant terminology concerning the Maslov index and the Lagrangian Grassmannian.

%-----------------------------------------------
\subsection{Mathematical Preliminaries}\label{Ss: mathematical preliminaries}
%-----------------------------------------------
When studying the spectral stability of a solution $\varphi$ to a PDE, one begins by linearizing about $\varphi$ to form a differential operator $\mathcal L$. The spectral stability of $\varphi$ is determined by the spectrum of this differential operator. If $\mathcal L$ has domain $D$, the point spectrum of $\mathcal L$ satisfies the eigenvalue problem 
\begin{equation}\label{eq: intro general eval problem}
\mathcal L u = \lambda u, \qquad u \in D.
\end{equation}
In our context, the first order system corresponding to \eqref{eq: intro general eval problem} will be given by a first order Hamiltonian system of the form
\begin{equation}\label{eq: general Hamiltonian}
q' = JC(x,\lambda)q, \qquad \qquad q \in \R^{2n}, \ J = \begin{pmatrix} 0 & I_n \\ -I_n & 0
\end{pmatrix},
\end{equation}
where $C(x, \lambda)$ is a $2n \times 2n$ symmetric matrix. Define the stable and unstable subspaces of solutions to \eqref{eq: general Hamiltonian} as
\begin{equation}\label{eq: general stable/unstable subspaces}
\begin{split}
\mathbb E^s_+(x, \lambda) & := \bigg\{ u \in \R^{2n} \ : \ u \text{ solves }\eqref{eq: general Hamiltonian} \text{ and } \lim_{x \to  \infty} \|u(x,\lambda)\| \to 0 \bigg\} \\ 
\mathbb E^u_-(x,\lambda) &  := \bigg\{ u \in \R^{2n} \ : \ u \text{ solves }\eqref{eq: general Hamiltonian} \text{ and } \lim_{x \to - \infty} \|u(x,\lambda)\| \to 0 \bigg\}.
\end{split}
\end{equation}
We will consider the evolution of $\mathbb E^u_-(x,\lambda)$ as either $x$ or $\lambda$ varies. More generally, both of these subspaces of solutions have a symplectic structure. To see why, let $u,v$ be solutions of \eqref{eq: general Hamiltonian} and observe 
$$\frac{d}{dx}\langle u, Jv\rangle = \langle u', Jv\rangle + \langle u, Jv'\rangle = \langle JCu, Jv\rangle + \langle u, JJC v\rangle = \langle u, C^Tv\rangle - \langle u, Cv \rangle = 0;$$
where we have used that $C$ is symmetric and $J^T = -J = J^{-1}$. We now introduce the  terminology and background necessary for further study of these subspaces through the lens of symplectic geometry.  

Given any $d$-dimensional vector space $V$, $G_k(V)$ denotes the set of all $k$-dimensional linear subspaces of $V$. When this set is given the structure of a smooth manifold, it is called the Grassmannian manifold. A symplectic form $\omega$ on $V$ is an anti-symmetric non-degenerate bilinear form on $V$ and taken together, the pair $(V, \omega)$ is called a symplectic vector space. If $V$ is finite dimensional, its dimension must be even. Thus, let $d = 2n$ and consider $G_n(\mathbb{R}^{2n})$. A Lagrangian plane is an $n$-dimensional subspace $\ell \in G_n(\mathbb{R}^{2n})$ with the property that $\omega(u,v) = 0$ for all vectors $u,v \in \ell$. The collection of all Lagrangian planes is called the Lagrangian Grassmannian and denoted by $\Lambda(n)$. In principle, the structure of $\Lambda(n)$ depends on the choice of $\omega$, but one can show the resulting spaces all have the same topological properties. We will be concerned only with the symplectic form defined by
\begin{equation}
\omega(u, v) = \langle u, Jv \rangle, \qquad J = \begin{pmatrix}
0 & I_n \\ -I_n & 0
\end{pmatrix},
\end{equation}
where $\langle \cdot , \cdot \rangle$ denotes the usual inner product on $\mathbb{R}^{2n}$. 

When referring to elements of $\Lambda(n)$, we will use the notation $\ell \in \Lambda(n)$ or another calligraphic letter. It will also be convenient to refer to a Lagrangian plane using a particular choice of a frame matrix. This means choosing $n \times n$ real matrices $X$ and $Y$ so that
\begin{equation}\label{E:frame}
\ell = \left\{ \begin{pmatrix} X \\ Y \end{pmatrix} u: u \in \mathbb{R}^n\right\}.
\end{equation}
To denote the frame matrix we will typically use block letters and write $L = \begin{pmatrix} X & Y
\end{pmatrix}^T$ inline or 
\[
L = \begin{pmatrix} X \\ Y \end{pmatrix}. 
\]
The choice of frame matrix is not unique. If $M$ is any invertible $n\times n$ real matrix, then $ \begin{pmatrix} XM & YM
\end{pmatrix}^T$ is also a frame matrix for the same Lagrangian plane. One can check that an $n$-dimensional subspace defined via \eqref{E:frame} is Lagrangian if and only if $X^TY = Y^T X$ and the column rank of $ \begin{pmatrix} X  & Y
\end{pmatrix}^T$ is $n$. We will be particularly interested in paths of Lagrangian planes on a compact domain, which we will denote by $\ell(t)$, or by its frame matrix $L(t)$, or via the component block matrices $X(t)$ and $Y(t)$ for $t$ in some interval.

This framework is particularly useful for the study of first-order Hamiltonian eigenvalue problems like \eqref{eq: general Hamiltonian}, because the associated stable and unstable subspaces in \eqref{eq: general stable/unstable subspaces} are paths of Lagrangian planes. %To see why these subspaces are Lagrangian, see \cite[Thm 2.6]{BCJ18}, which can be adapted to this setting in a fairly straightforward way. 

We now describe how to leverage this framework to study solutions to \eqref{eq: intro general eval problem}. One has an eigenvalue of $\mathcal L$ at $\lambda = \lambda_*$ if there exists a solution to \eqref{eq: general Hamiltonian} that decays to zero as $|x| \to \infty$. Note that one really needs the solution to live in an appropriate Banach space but for most spaces, if one assumes that the essential spectrum is bounded away from $0$, this is equivalent to asymptotic decay of the eigenfunction. This directly corresponds to a nontrivial intersection of the stable and unstable subspaces for that value of the spectral parameter: $\mathbb E^{u}_-(x;\lambda_*) \cap \mathbb E^{s}_+(x;\lambda_*) \neq \{0\}$. The operator $\mathcal L$ has real spectrum since it is self adjoint so one could detect unstable eigenvalues by asking, \textit{are there any values of $\lambda > 0$ such that $\mathbb E^{u}_-(x;\lambda) \cap \mathbb E^{s}_+(x;\lambda) \neq \{0\}$?} 

It is perhaps surprising that one can reformulate this question as follows. Pick any fixed Lagrangian plane $\ell_* \in \Lambda(n)$, which we will call the reference plane. Fix the spectral parameter to be $\lambda = 0$ and ask, \textit{are there any values of $x \in \mathbb{R}$ for which $\mathbb E^{u}_-(x; 0) \cap \ell_* \neq \{0\}$?} Values $x = x_0$ where $\mathbb E^{u}_-(x_0; 0) \cap \ell_* \neq \{0\}$ are called conjugate points. This reformulation amounts to saying that the number of unstable eigenvalues of $\mathcal{L}$ is equal to the number of conjugate points associated with the reference plane $\ell_*$ at $\lambda = 0$ and it can be understood via the use of the Maslov index. The main results of this paper are justifying this reformulation in the context of the Swift-Hohenberg equation and providing a framework for counting the conjugate points. 

The Maslov index is used in the study of symplectic geometry and the theory of Fredholm operators. It is a counting index that provides a measure of the number of times a certain type of curve, called a Lagrangian submanifold, intersects a given surface in a symplectic space. This count can be used to provide a topological characterization of solutions to Hamiltonian systems that determines their spectral stability.

The main properties of the intersection number of Lagrangian submanifolds were proven in \cite{Arnold} and followed by results that connect the Maslov index to self-adjoint operators \cite{Arnold85, duistermaat}. These results can be viewed as generalizations of classical Sturm Liouville theory \cite{smale65, bot56,edwards64}. Beginning with Jones' work \cite{jones88}, the Maslov index has been used as a tool to study stability in systems that are Hamiltonian in their spatial variables \cite{CoxJonesMarzoula15Schrodinger,CoxJonesMarzoula15elliptic,schro[01], HLS18,CJLS16}. While these results connect the number of unstable eigenvalues to the number of conjugate points, it is often unclear how to find the number of conjugate points, which is necessary for this framework to be useful in determining spectral stability. \ADD{We describe two examples of systems for which the number of conjugate points have been computed for specific solutions.} The first is \cite{BCJ18}, in which the Maslov index was used to show that a symmetric pulse solution to a system of reaction diffusion equations is unstable. This was followed by \cite{MBJJ21}, in which the Maslov index and computer assisted proofs were used to argue that a parameter dependent system of bistable equations has both stable and unstable fronts.

Our goal is to use the Maslov index to argue that a pulse solution to the Swift-Hohenberg equation is unstable in certain parameter regions. In order to do so, we further Maslov theory in two distinct directions. In the first direction, we apply the Maslov index to a fourth order system that has not been considered in previous literature. We note that the unstable eigenvalues of certain fourth-order potential systems can be counted using the Maslov index \cite{howard2020}, but the Swift-Hohenberg equation falls outside the class of considered operators. The second contribution that this work makes is the generalization of the crossing form introduced in \cite{robbinsalamon} to $k$th order degenerate crossings. \ADD{Most of the previous work analyzing the stability of coherent structures using the crossing form of \cite{robbinsalamon} relies on so-called regular crossings, with  exceptions such as \cite{dengjones11, CoxCurranLatushkinMarangell}, in which second order crossings are considered.} Since the Swift-Hohenberg equation has third order crossings, these results do not apply and our theory is not restricted to a particular order of crossing. 
 There is one important restriction in our result, however, which is the requirement that the intersection is fully degenerate, meaning that if the dimension of the intersection is $k$, then each of the $k$ crossings is a degenerate one. Although we do not prove that this must necessarily be true for the Swift-Hohenberg equation, we provide numerical evidence that suggests it is true for all of the example pulses that we consider. Furthermore, we discuss a subtlety concerning the formulation of Lagrangian planes as graphs that becomes critical in calculations for non-regular crossings.

We begin by stating some assumptions and basic facts about this setting, which will enable us to then explain our main results. First, we assume that the nonlinearity $f(u)$ supports the existence of a pulse solution. 
\begin{hypothesis}\label{hyp: phi pulse} 
There exists a stationary solution $\varphi$ to \eqref{eq: swift hohenberg} such that $\lim_{x \to \pm \infty} \varphi(x) = 0$.
\end{hypothesis}

Linearizing about a stationary solution $\varphi$ and formulating the associated eigenvalue problem yields 
\begin{equation}\label{eq: SH eigenvalue problem on R}
\begin{split}
 \mathcal Lu & = \lambda u, \\
 \mathcal L  = -\partial_x^4 - 2 & \partial_x^2 - 1+ f'\circ\varphi(x).
 \end{split}
\end{equation}
We are primarily interested in detecting spectral instabilities, which means determining whether or not the spectrum of $\mathcal{L}$ has positive real part. 
Since the spectrum can be divided into the essential spectrum and point spectrum, and the essential spectrum is typically easy to compute, we wish to assume the essential spectrum is stable and focus on detecting instabilities in the point spectrum. Therefore, we make the following second assumption, which implies that the essential spectrum of the pulse is contained in the open left half plane \cite{promislow}. 

\begin{hypothesis}\label{hyp: nonlinearity} The derivative of the nonlinearity $f(u)$ in \eqref{eq: swift hohenberg} evaluated at $u = 0$ is negative.
\end{hypothesis}
\begin{remark}
\ADD{This hypothesis is  equivalent to requiring that the essential spectrum of $\mathcal L$ lies strictly to the left of the imaginary axis.} This is important because our analysis for $\lambda \in [0, \infty)$ relies on the existence of exponential dichotomies on $\R^+$ and $\R^-$ \cite{coppel}.
\end{remark}

To detect unstable eigenvalues, we will exploit the symplectic structure of the eigenvalue equation \eqref{eq: SH eigenvalue problem on R}, which can be seen via the change of variables
\begin{equation}\label{eq: symplectic change of coords sh}
q_1 = u, \qquad q_2 = u_{xx}, \quad q_3 = u_{xxx}+2u_x, \quad q_4=u_x.
\end{equation}
This allows us to write \eqref{eq: SH eigenvalue problem on R} as the first order system
\begin{equation}\label{eq: linearized first order}
\begin{split}q'& =B(x,\lambda)q, \qquad B(x, \lambda) :=  \begin{pmatrix} 0 & 0 & 0 & 1 \\ 0 & 0 & 1 & -2 \\ -\lambda-1 + f'\circ\varphi(x) & 0 & 0 & 0 \\ 0 & 1 & 0 & 0\end{pmatrix}.
\end{split}
\end{equation}
Since this system is Hamiltonian, we can also formulate the above equation as 
\begin{equation}\label{eq: symm C}
\begin{split} q'& =JC(x,\lambda)q,  \\
J =  \begin{pmatrix} 0 & I_2 \\ -I_2  & 0 \end{pmatrix}, & \qquad 
C(x,\lambda)= \begin{pmatrix} \lambda+1-f'\circ\varphi(x) & 0 & 0 & 0 \\ 
0 & -1 & 0 & 0 \\ 0 & 0 & 0 & 1 \\ 0 & 0 & 1 & -2
\end{pmatrix}.
\end{split}
\end{equation}
%\RM{
%Note that we can also express this as $q' = -J\nabla_q H(x,q)$, where 
%\begin{equation}\label{eq: hamiltonian} 
%H(x,q) = \frac{1}{2} \big( -\lambda - 1 + f'\circ\varphi(x) \big)q_1^2 + \frac{1}{2} q_2^2 - q_3q_4 + q_4^2.
%\end{equation}
%\begin{remark}\label{rmk: 0 energy}
%If a solution $q(x;\lambda)$ lies in $\mathbb E^u_-(x;\lambda)$, then $q$ lies in the $0$ energy level and $H(x,q) = 0$.
%\end{remark}
%}

Using Hypothesis \ref{hyp: phi pulse}, we find the asymptotic limits of the coefficient matrix $B$ in \eqref{eq: linearized first order} to be given by 
\begin{equation}\label{eq: B_inf matrix}
B_\infty(\lambda) :=\lim_{x \to \pm \infty} B(x,\lambda) = \begin{pmatrix} 0 & 0 & 0 & 1 \\ 0 & 0 & 1 & -2 \\ -\lambda-1 + f'(0) & 0 & 0 & 0 \\ 0 & 1 & 0 & 0\end{pmatrix}.
\end{equation}
The asymptotic matrix $B_\infty(\lambda)$ is useful in understanding the dynamics of \eqref{eq: linearized first order} because it governs the asymptotic behavior of the system. We collect the relevant properties of $B_\infty(\lambda)$ into the below lemma. 

\begin{lemma}\label{lemma: B infinity hyperbolic} Let $B_\infty$ be as defined in \eqref{eq: B_inf matrix}. Then 
\begin{enumerate}[a)]
\item $B_\infty(\lambda)$ is hyperbolic for all $\lambda \geq 0$. 
% \item If 
% $$ \theta = \arctan \left(-\sqrt{\lambda - f'(0)}\right) \quad \text{ and } r = \sqrt{1 + \lambda - f'(0)},$$
% then basis vectors for the real unstable and stable eigenspaces of $B_\infty(\lambda)$ are given via 
% \begin{equation}\label{eq: B_inf unstable evecs}
% R^u_1 = \begin{pmatrix} \frac{1}{r}\cos \theta \\ 1 \\ \left(\frac{2}{\sqrt r} + \sqrt{r} \right) \cos \left(\frac{\theta}{2} \right)  \\ \frac{1}{\sqrt r} \cos \left(\frac{\theta}{2} \right)
% \end{pmatrix}, \quad  R^u_2 = \begin{pmatrix} -\frac{1}{r}\sin \theta \\ 0 \\ \left(\sqrt{r} - \frac{2}{\sqrt r} \right) \sin \left(\frac{\theta}{2} \right)  \\ -\frac{1}{\sqrt r} \sin \left(\frac{\theta}{2} \right)
% \end{pmatrix},
% \end{equation} 
% \begin{equation}\label{eq: B_inf stable evecs}
% R_1^s = \begin{pmatrix} \frac{1}{r}\cos \theta \\ 1 \\ \left(-\frac{2}{\sqrt{r}}  - \sqrt r \right) \cos \left( \frac{\theta}{2}\right) \\
% -\frac{1}{\sqrt r}\cos \left(\frac{\theta}{2}\right)
% \end{pmatrix}, \quad  R_2^s = \begin{pmatrix} -\frac{1}{r}\sin \theta \\ 0 \\ \left(\frac{2}{\sqrt{r}} - \sqrt r \right) \sin \left( \frac{\theta}{2}\right) \\
% \frac{1}{\sqrt r}\sin \left(\frac{\theta}{2}\right)
% \end{pmatrix}. 
% \end{equation}
\item The matrix $B_\infty$ depends analytically on $\lambda$. Also, there are positive constants $K_B$ and $C_B$, independent of $\lambda$, such that 
\begin{equation}\label{eq: matrix conversion B in hypothesis}
\|B(x;\lambda) - B_\infty(\lambda)\| \leq K_Be^{-C_B|x|} \text{ as } x \to \pm \infty.
\end{equation}
\end{enumerate}
\end{lemma}
\begin{proof}
See Section \ref{Subsec: proof of introduction lemma}. 
\end{proof}

%-----------------------------------------------
\subsection{Statement of Main Results}
%-----------------------------------------------
We now state the main result of our work.

\begin{theorem}\label{thm: main result 1}
Let $\ell(x) =  \mathbb E^u_-(x; 0)$ and fix the reference plane $\ell_* = \ell^{sand}_*$ with 
$$\ell^{sand}_* = \text{colspan} \begin{pmatrix} 0 & 0 \\ 1 & 0 \\ 0 & 1 \\ 0 & 0 
\end{pmatrix}.$$
 If Hypothesis \ref{hyp:degeneracy}  is satisfied, then the number of unstable eigenvalues of the operator in \eqref{eq: SH eigenvalue problem on R} acting on $H^4(\R)$ is equal to the number of conjugate points of the path $\ell(x)$ with respect to the reference plane $\ell_*^{sand}$. 
\end{theorem}

\begin{remark}
The notation $\ell_*^{sand}$ is chosen because the identity matrix is ``sandwiched'' between two rows of zeros. This choice of reference plane is used to study fourth-order potential systems in \cite{howardhormander,howard2020}.
\end{remark}

\begin{remark}
	In brief, Hypothesis \ref{hyp:degeneracy} states that the dimension of intersection $ \ell(x) \cap \ell^{sand}_*$ is at most one. 
\end{remark}

The remainder of the paper is organized as follows. We establish further background in \S\ref{section: grassmann} on the Lagrangian Grassmannian and related objects, which are necessary to provide a precise definition of conjugate points. We also provide a formulation of the Maslov index that applies to degenerate crossings and generalize the results of \cite{robbinsalamon}. In \S\ref{S:count} we prove Theorem \ref{thm: main result 1} and compute the conjugate points in \S\ref{S:numerics} using numerics. In \S\ref{S:future} we discuss future work and in the appendix \S\ref{S:appendix} we discuss an explicit example and collect some technical results about coordinate changes and representing subspaces as a graph. 

%%%%%%%%%%%%%%%%%%%%%%%%%%%%%%%%%%%%%%%%%%%%%%%%%%%%%%%%%%%%%%%%%%%%%%%%%%%%%%%%%%%%%%%%%%%%%%%
%%%%%%%%%%%%%%%%%%%%%%%%%%%%%%%%%%%%%%%%%%%%%%%%%%%%%%%%%%%%%%%%%%%%%%%%%%%%%%%%%%%%%%%%%%%%%%%
%%%%%%%%%%%%%%%%%%%%%%%%%%%%%%%%%%%%%%%%%%%%%%%%%%%%%%%%%%%%%%%%%%%%%%%%%%%%%%%%%%%%%%%%%%%%%%%
%%%%%%%%%%%%%%%%%%%%%%%%%%%%%%%%%%%%%%%%%%%%%%%%%%%%%%%%%%%%%%%%%%%%%%%%%%%%%%%%%%%%%%%%%%%%%%%

%\subfile{MaslovIndex}

%-----------------------------------------------------------------------------------------------------------------------------------
\section{Conjugate Points and the Crossing Form}\label{section: grassmann}
%-----------------------------------------------------------------------------------------------------------------------------------

In this section we establish further background on conjugate points and the crossing form. Recall that we are interested in detecting eigenvalues of \eqref{eq: SH eigenvalue problem on R} by finding objects called conjugate points. Conjugate points occur when the path of Lagrangian subspaces given by $\mathbb E^u_-(x,\lambda)$ intersect a fixed reference plane. These intersections are detected with a quadratic form called the crossing form that was developed in \cite{robbinsalamon} and the sign of this quadratic form determines the sign with which the conjugate points contribute to the Maslov index. Intuitively, the crossing form is performing a first derivative test on a particular function and if the function is increasing, then the crossing form is positive and the crossing contributes to the Maslov index with positive sign (similarly for the decreasing/negative case).

As an example, consider the function $f(t) = \pm t^j$ and suppose that we are interested in whether $f$ is increasing or decreasing at $t = 0$. If $f(t)$ is linear, then the first derivative will give us this information. This first derivative test is the analogue of the crossing form developed in \cite{robbinsalamon}. However, if $f$ is a higher order polynomial, then the first derivative at $t = 0$ vanishes and we must use a higher order derivative test by finding the lowest $j \in \mathbb N$ such that $\frac{d^j}{dt^j} f(t) \big|_{t = 0} = f^{(j)}(0) \neq 0$.

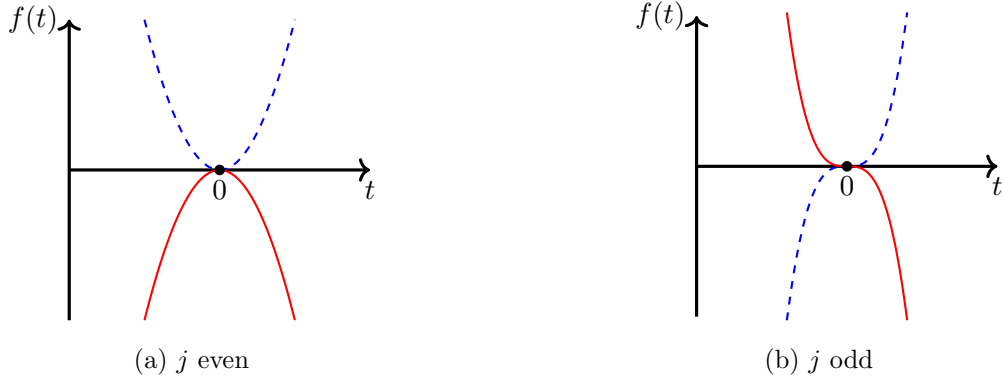
\begin{figure}[H]
    \begin{subfigure}[b]{0.5\textwidth}
        \centering
            \begin{tikzpicture}
                \draw[->, very thick] (-2,0) -- (2,0) node[below] {$t$};
                \draw[->, very thick] (-2,-2) -- (-2,2) node[left]{$f(t)$};
                \draw[scale=0.5, domain=-2:2, smooth, variable=\x, blue, thick, dashed] plot ({\x}, {\x*\x}) node[right]{};
                \draw[scale=0.5, domain=-2:2, smooth, variable=\x, red, thick] plot ({\x}, {-\x*\x}) node[right]{};
                \fill (0,0) circle (2pt) node[below] {$0$};
            \end{tikzpicture}
        \caption{$j$ even}
        \label{fig:n even example}
    \end{subfigure}
    \begin{subfigure}[b]{0.5\textwidth}
    \centering
            \begin{tikzpicture}
                \draw[->, very thick] (-2,0) -- (2,0) node[below] {$t$};
                \draw[->, very thick] (-2,-2) -- (-2,2) node[left]{$f(t)$};
                \draw[scale=0.5, domain=-1.6:1.6, smooth, variable=\x, blue, thick, dashed] plot ({\x}, {\x*\x*\x}) node[right]{};
                \draw[scale=0.5, domain=-1.6:1.6, smooth, variable=\x, red, thick] plot ({\x}, {-\x*\x*\x}) node[right]{};
                \fill (0,0) circle (2pt) node[below] {$0$};
            \end{tikzpicture}
        \caption{$j$ odd}   
        \label{fig:n odd example}
    \end{subfigure}
\caption{The function $f(t) = \pm t^j$ near $t = 0$. The case that $f^{(j)}(0) > 0$ is depicted by the dashed line and the case that $f^{(j)}(0)< 0$ is given by the solid line.} 
\label{fig: f example}
\end{figure}

In the case of the Swift-Hohenberg equation, we will need to use the equivalent of a third order derivative test to determine the sign of the contribution to the Maslov index. The goal of this section is to develop the analogue of a higher order derivative test via higher order crossing forms for the Maslov index. Compare Figure \ref{fig: f example} depicting $f(t)$ with Figure \ref{fig: lambda derivative trajectories} depicting the information given by the higher order crossing forms. 

To this end, we first discuss the formulation of a path of Lagrangian planes as the graph of a matrix. We will then see in subsequent sections that the Maslov index can be connected to the eigenvalues of this matrix. There are several subtleties associated with the graph formulation of a Lagrangian subspace that are critical in the computation of higher order crossing forms (see Remark \ref{rmk: graphing error}).

Let $\ell(t)$ be a Lagrangian path defined for $t \in [a,b]$ with frame matrix $L(t)$. Fix $t_0 \in [a,b]$ and let $\mathcal W$ be a Lagrangian plane such that $\ell(t_0) \cap \mathcal W = \{0\}$. Then, there exists an $\epsilon > 0$ and a $2n \times 2n$ matrix $A(t): \ell(t_0) \to \mathcal W$ such that for all $v \in \ell(t_0)$, we have $v \in \ker A(t_0)$ and
\begin{equation}
v + A(t)v \in \ell(t), \qquad  \ t \in (t_0 - \epsilon, t_0 + \epsilon).
\end{equation}

\begin{figure}[H]{}
  \centering
        \begin{tikzpicture}[scale=1]
        \draw[thick, -] (-3, 0) -- (3, 0) node[right] {$\ell(t_0)$};
        \draw[thick, -] (0, -3) -- (0, 3) node[above] {$\mathcal W$};
        \draw[thick, -] (-2.4, -2.4) -- (2.4, 2.4) node[above] {$\ell(t)$};

        \draw[dashed] (1.5, 1.5) -- (1.5, 0)  node[below] {$v$};
        \draw[dashed] (1.5, 1.5) -- (0, 1.5) node[left] {$A(t)v$};

        \filldraw [black] (1.5,1.5) circle (1.5pt) node[right] {$u = v + A(t)v$};
        \filldraw [black] (1.5,0) circle (1.5pt) node[right] {};    
        \filldraw [black] (0,1.5) circle (1.5pt) node[right] {};    
        \end{tikzpicture}
    \caption{Representing $\ell(t)$ as a graph for $t$ near $t_0$.} \label{fig:timing1}
\end{figure}
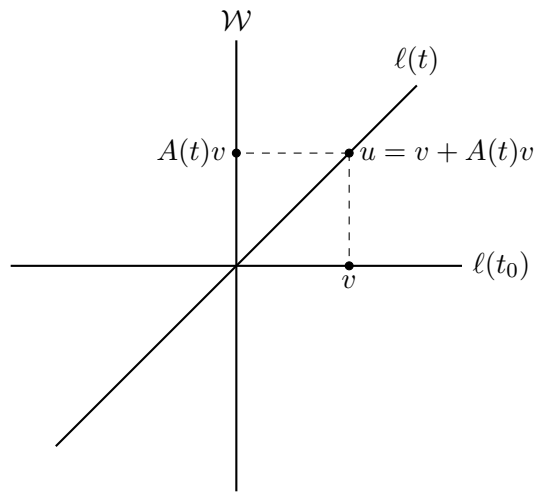

% \begin{example} As a low dimensional example, consider the Lagrangian plane $\ell(t) = \big(\cos t, \sin t \big)$. Fix $t_0 = 0$ so that $\ell(t_0) = \text{span}\{(1, 0)\}$ and $\ell(t_0)^\perp = \text{span}\{(0,1)\}$. If we choose the basis $\R^{2n} = \ell(t_0) \oplus \ell(t_0)^\perp$, since $\cos (t)$ is invertible on $(-\epsilon, \epsilon)$ for sufficiently small $\epsilon$, we can define $A(t) = \tan t$ and write  
% $$\ell(t) = \big\{(\tilde u, A(t) \tilde u) \ : \ \tilde u \in \R \big\} $$

% OR 

% we can define $w(t) = (0, \tan t)$ and epress $\ell(t) = v + w(t)$. 
% \end{example}

In particular, suppose we are given a Lagrangian subspace $\ell(t)$ with frame 
$$ L(t) = \begin{pmatrix} | &  & | \\ v_1(t) & \dots & v_n(t) \\ | &  & | 
\end{pmatrix}.$$
For $i = 1,\dots, n$, fix $k_i \in \R$, and define $v = \sum_{i = 1}^n k_i v_i(t_0) \in \ell(t_0)$. Then, there is a matrix $A(t): \ell(t_0) \to \mathcal W$ and functions $c_i(t): \R \to \R$, such that for all $v \in \ker A(t_0)$ and $t \in (t_0 - \epsilon, t_0 + \epsilon)$,
\begin{equation}\label{eq: graph A satisfies equation}
\begin{split}
v + A(t)v \ & = \sum_{i=1}^n c_i(t)v_i(t),  \\
& = L(t)c(t), \qquad c(t):=\big(c_1(t), \dots, c_n(t)\big)^T.\\
\end{split}
\end{equation}
The coefficient functions $c_i(t)$ may depend on $t$ and the constant coefficients $k_i$ that determine $v$.  See Section \ref{subsec: explicit example} for an explicit example. 

% \begin{remark}
% If a frame for $\mathcal V$ is given by $L_1 = \begin{pmatrix} X \\ Y
% \end{pmatrix}$, then $\mathcal V^\perp$ has frame $L_2 = \begin{pmatrix} -Y \\ X 
% \end{pmatrix}$. It is clear that $\mathcal V^\perp$ is Lagrangian and since 
% $$\left\langle \begin{pmatrix} X \\ Y
% \end{pmatrix}, \begin{pmatrix}
%  - Y \\ X
% \end{pmatrix} \right\rangle = \langle Y, X \rangle - \langle X, Y\rangle = 0$$
% that $\mathcal V^\perp$ is in fact the orthogonal complement of $\mathcal V$. 
% \end{remark}

%-----------------------------------------------------------------
\subsection{The Crossing Form and the Maslov Index}\label{subsec: crossing form}
%-----------------------------------------------------------------
We now introduce the terminology and framework developed in \cite{robbinsalamon} to understand the evolution of the Lagrangian subspace $\ell(t)$ with respect to another fixed Lagrangian plane. 
\begin{definition}\label{def: first order quadratic form} \cite{robbinsalamon} Suppose that $\ell(t)$ is a smooth path of Lagrangian planes in $\R^{2n}$. Fix $t_0$ and choose $\mathcal W$ such that $\mathcal W \cap \ell(t_0) = \{0\}$. For $v \in \ell(t_0)$ and $t$ near $t_0$ define the operator $A(t): \ell(t_0) \to \mathcal W$ such that its graph is $\ell(t)$. In particular, for all $v \in \ell(t_0)$, we have $v + A(t)v \in \ell(t)$.
Then, the first order quadratic form associated with $\ell(t)$ at $t = t_0$ acting on the vector $v$ is defined as 
\begin{equation}\label{eq: first order quad from}
Q^{(1)}(v) = \frac{d}{dt}\omega(v,A(t)v) \bigg|_{t = t_0}.
\end{equation}
\end{definition}

The following result from \cite{robbinsalamon} concerns the first order quadratic form. 
\begin{theorem}\cite{robbinsalamon}[Theorem 1.1]\label{thm: RS main result}
Let $\ell(t) \in \Lambda(n)$ be a $C^1$ path of Lagrangian planes with frame $L(t) = \begin{pmatrix} X(t) & Y(t) \end{pmatrix}$. For fixed $t_0$ and $v \in \ell(t_0)$, let $Q^{(1)}(v)$ be as in Definition \ref{def: first order quadratic form}. 
\begin{enumerate}[a)]
\item The first order quadratic form is invariant under symplectic changes of coordinates. 
\item The first order quadratic form is independent of the choice of plane $\mathcal W$. 
\item Let $v = L(t_0)u$ with $v \in \R^{2n}$ and $u \in \R^n$. Then $Q^{(1)}(v)$ can be calculated via the formula
$$ Q^{(1)}(v) = \big \langle X(t_0) u, \dot Y(t_0) u \big \rangle - \big \langle Y(t_0) u, \dot X(t_0)  u \big \rangle.$$
\end{enumerate}
\end{theorem}
\begin{remark}\label{rmk: symplectic coordinates reference}
We refer the reader to \cite{mcduffsalamon17}[Chapter 2] for a review of symplectic coordinates. For this work, it is sufficient to recall that a $2n \times 2n$ matrix $\Psi$ is symplectic if $\Psi^TJ\Psi = J$ and that a symplectic change of coordinates is done via a symplectic transition matrix.
\end{remark}

Fix a subspace $\ell_*$ such that $\ell_* \cap \ell(t_0) \neq \{0\}$. If we restrict the first order quadratic form to $v \in \ell(t_0) \cap \ell_*$, then we obtain the first order crossing form. Before giving this definition, we first formalize the notion of a crossing. 

\begin{definition}\cite{Arnold} A crossing for the path $\ell(t) \in \Lambda(n)$ with respect to a fixed $\ell_*$ is a value $t_0 \in [a,b]$ such that $\ell(t_0)\cap \ell_* \neq \{0\}.$ The plane $\ell_*$ is often called a reference plane. We use the following terminology to describe crossings:
\begin{enumerate}[a)]
\item a crossing is simple if $\ell(t_0) \cap \ell_*$ is a one dimensional subspace of $\R^{2n}$;
\item a crossing is isolated if there exists $\epsilon > 0$ such that $\ell(t) \cap \ell_* = \{0\}$ for all $t \in [t_0 - \epsilon, t_0 + \epsilon]\setminus\{t_0\}$.
\end{enumerate}
\end{definition}

Now we are in a position to define the crossing form and associated terminology.
\begin{definition}\label{def: first order crossing form}\cite{robbinsalamon}
The first order crossing form with respect to $\ell(t)$ and $\ell_*$ is defined to be the first order quadratic form restricted to $\ell(t_0) \cap \ell_*$:
\begin{equation}\label{eq: first order crossing form}
\Gamma(\ell, \ell_*, t_0) :=  Q^{(1)}(v), \qquad v \in \ell(t_0) \cap \mathcal \ell_*.
\end{equation}
We say a crossing $t_0$ is regular or non-degenerate if the crossing form $\Gamma$ is non-singular.

\end{definition} 
% \begin{remark}
% When constructing the crossing form in general, one often takes $\mathcal W = \ell_*^\perp$. We will do this moving forward unless we want to state results in more generality. 
% \end{remark}

Our notation for this quadratic form differs slightly from that in \cite{robbinsalamon} due to the addition of the superscript in $Q^{(1)}$. We introduce this extra notation to denote the first derivative contained in this quadratic form. This will be useful later in Section \ref{section: maslov index for nonregular crossings} when we take higher order derivatives to construct higher order quadratic forms. The crossing form can be used to explore the signed intersections of $\ell(t)$ with a reference plane, $\mathcal \ell_*$. If $\Pi$ denotes the projection operator onto $\ell_* \cap \ell(t_0)$, then the matrix that determines $Q^{(1)}(v)$ is $\Pi J \dot A(t_0)\Pi$; this follows from Definition \ref{def: first order crossing form} and the restriction of $v \in \ell(t_0) \cal \ell_*$. Proposition \ref{prop: Structure of A} says this matrix is symmetric so all of its eigenvalues are real. 

The signature of a matrix $B$ is defined as 
$ \text{sign} B = p -q $ where $p$ is the number of positive eigenvalues and $q$ is the number of negative eigenvalues of $B$. We define the signature of a quadratic form as the signature of the matrix that determines the quadratic form. For this reason, we write $\text{sign} Q^{(1)}$ and do not specify the argument because this signature does not depend on the argument $v$.

% \red{Is this necessary? 
% \begin{proposition}\label{prop: crossings as zero eigenvalues} Suppose we have a smooth path of Lagrangian subspaces $\ell:[a,b] \to \Lambda(n)$ and a Lagrangian subspace $\mathcal W$ such that $\ell(t_0) \cap \mathcal W = \{0\}$. If we represent $\ell(t)$ as the graph of an operator $A(t): \ell(t_0) \to \mathcal W$ and $v \in \ell(t_0)$ then $v \in \ker A(t_0)$. 
% \end{proposition}
% \begin{proof} 
% This follows from the definition as $A(t)$ giving the graph of $\ell(t)$ for $t$ near $t_0$. 
% \end{proof}
% }
\begin{definition}\label{def: Mas index} \cite{robbinsalamon,duistermaat} Following the setting and notation of Definitions \ref{def: first order quadratic form} and \ref{def: first order crossing form}, define $\Gamma^+(\ell, \ell_*,t)$ and $\Gamma^-(\ell, \ell_*, t)$ to the be right and left derivatives of $\omega(v,A(t)v)$ respectively. For a smooth curve $\ell: [a,b] \to \Lambda(n)$ with only regular crossings, we define the Maslov index to be 
$$ \text{Mas} (\ell(t), \ell_*; \ a \leq t \leq b) = \frac{1}{2} \text{sign } \Gamma^+ (\ell, \ell_*, a) + \sum_{a < t < b} \text{sign } \Gamma(\ell, \ell_*, t) + \frac{1}{2}\text{sign } \Gamma^- (\ell, \ell_*, b),$$
where the sum is taken over all crossings $t$. 
\end{definition}

The important features of the Maslov index for this work are given below:
\begin{lemma}\label{lemma: homotopic to regular}\cite{robbinsalamon}
Suppose that $\ell: [a,b] \to \Lambda(n)$ is a smooth path of Lagrangian subspaces and $\ell_*$ is a fixed reference plane. 
\begin{enumerate}[a)]
\item (Additivity) For $a < c < b$
$$ \text{Mas} \left(\ell(t), \ell_*; \ a\leq t \leq b\right) = \text{Mas} \left(\ell(t), \ell_*; \ a\leq t \leq c\right) + \text{Mas} \left(\ell(t), \ell_*; \ c\leq t \leq b\right).$$
\item (Homotopy invariance) Two paths $\ell_1, \ell_2: [a,b] \to \Lambda(n)$ with $\ell_1(a) = \ell_2(a)$ and $\ell_1(b) = \ell_2(b)$ are homotopic with fixed endpoints if and only if they have the same Maslov index with respect to a fixed reference plane $\ell_*$.
\end{enumerate}
\end{lemma}

We now discuss a formulation of the Maslov index from a spectral flow perspective. This formulation will be useful in the extension of the crossing form to non-regular crossings because it yields a definition of the Maslov index that does not rely on the non-degeneracy of the first order crossing form. Suppose that $\dim(\ell(t_0) \cap \ell_*) = k$. Then the matrix $\Pi J A(t_0)\Pi $ has a $k$-dimensional kernel and the signature of the matrix $\Pi J \dot A(t_0)\Pi$ tracks how many positive eigenvalues $\Pi J A(t)\Pi$ gains or loses as $t$ increases through $t_0$, which will be an integer in $\{-k,\dots, k\}$. Therefore, we can understand the contribution of a crossing to the Maslov index by studying the spectrum of $\Pi J A(t)\Pi$ for $t$ near $t_0$. We formalize this by following the work done in \cite{furutani}. We first have the following technical lemma. 

\begin{lemma}\cite{furutani}[Lemma 3.23]\label{lemma: furutani Q1} Fix $t_0$ and let $\{B(t)\}$ be a $C^1$ family of self-adjoint operators on $\R^{2n}$ for $t$ near $t_0$. Suppose that the set of $t$ for which $B(t)$ has a zero eigenvalue is comprised of singletons and that the quadratic form 
$$R^{(1)}(x) = \frac{d}{dt}\langle x, B(t) x\rangle \bigg|_{t = t_0} = \left\langle x, \dot B(t_0) x \right\rangle, \qquad x \in \ker B(t_0)  $$
is non-degenerate. Furthermore, suppose $\dim \ker B(t_0) = k$. Then, there exists $p, q \in \N \cup \{0\}$ such that $p + q = k$ and $\epsilon, \delta > 0$ such that $\text{sign} D^{(1)} = \text{sign} \dot B(t_0) = p-q$. Furthermore,
\begin{enumerate}[i)]
\item for $t_0 < t \leq t_0+ \delta$, there are $p$ positive and $q$ negative eigenvalues of $B(t)$ within $\epsilon$ of $0$; 
\item for $t_0-\delta \leq t < t_0$, there are $q$ positive eigenvalues and $p$ negative eigenvalues of $B(t)$ within $\epsilon$ of $0$.
\end{enumerate} 
\end{lemma} 
\begin{remark} The condition that $\{t_0 \ : \lambda(t_0) = 0\}$ in Lemma \ref{lemma: furutani Q1} is comprised of singletons is equivalent to the condition that the crossing at $t_0$ is isolated. 
\end{remark}

It is straightforward to relate Lemma \ref{lemma: furutani Q1} to the motion of the eigenvalues of $B(t)$ as $t$ increases through $0$. 

\begin{corollary}\label{corollary: regular crossing evals} Suppose we are in the setting of Lemma \ref{lemma: furutani Q1}. If $\text{sign }R^{(1)} = p-q > 0$, 
 then the sign of the matrix $B(t)$ increases by $p-q$ as $t$ increases through $0$. If $\text{sign }R^{(1)} = p-q < 0$, then the sign of $B(t)$ decreases by $p-q$ as $t$ increases through $0$. 
\end{corollary}

Suppose the graph of $A(t): \ell(t_0) \to \mathcal W$ is $\ell(t)$ for $t$ near $t_0$ and define $B(t) = JA(t)$ restricted to $\ell(t_0) \cap \ell_*$. Recall that the projection operator onto $\ell(t_0) \cap \ell_*$ is given by $\Pi$. Then via Corollary \ref{corollary: regular crossing evals}, the sign of the crossing form at a crossing $t = t_0$ satisfies
$$ \text{sign}\Gamma(\ell, \ell_*, t_0) = \text{sign}Q^{(1)} = \text{sign}\Pi J\dot A(t_0) \Pi  = p-q,$$
which is precisely the number of positive eigenvalues that $\Pi JA(t)\Pi $ gains or loses as $t$ passes through $t = t_0$. This result is captured in the localization property of the Maslov index. The below theorem allows us to define the Maslov index in terms of the eigenvalues of $A(t)$ for $t$ near $t_0$ instead of relying on the derivative $\dot A(t_0)$.
\begin{theorem}\cite{robbinsalamon}[Theorem 2.3] \label{thm: spectral flow RS}
Suppose that $t_0$ is an isolated crossing of $\ell(t): [a,b] \to \Lambda(n)$ with reference plane $\ell_*$. Then if $t_0 \in (a,b)$, we have for sufficiently small $\epsilon$, the Maslov index of $\ell: [t_0 - \epsilon, t_0 + \epsilon] \to \Lambda(n)$ is given by the spectral flow formulation 
$$ \text{Mas}(\ell, \ell_*; |t-t_0| \leq \epsilon) = \frac{1}{2}\text{sign}\Pi JA(t_0 + \epsilon)\Pi - \frac{1}{2}\text{sign}\Pi JA(t_0 - \epsilon)\Pi.$$
If $t_0$ is equal to $a$ or $b$, then 
\begin{align*} \text{Mas}(\ell, \ell_*; \ t_0 \leq  t \leq t_0 + \epsilon) & = \frac{1}{2}\text{sign}\big[\Pi JA(t_0 +\epsilon)\Pi\big] - \frac{1}{2}\text{sign}\big[\Pi JA(t_0)\Pi \big] \\
\text{Mas}(\ell, \ell_*; \ t_0 -\epsilon \leq  t \leq t_0 ) & = \frac{1}{2}\text{sign}\big[\Pi JA(t_0)\Pi \big]  - \frac{1}{2}\text{sign}\big[\Pi JA(t_0 - \epsilon)\Pi\big].
\end{align*}
\end{theorem}

\begin{remark}\label{rmk: spectral flow form of the maslov index}
It is important to note that Theorem \ref{thm: spectral flow RS} does not rely on the non-degeneracy of the crossing form because there is no differentiation. This fact will be useful in the proof of Theorem \ref{theorem: maslov gen tang} when we extend the definition of the Maslov index via crossing forms to degenerate crossings. Additionally, our formulation of Theorem \ref{thm: spectral flow RS} differs slightly from the formulation in \cite{robbinsalamon}. This is because we choose to represent $\ell(t)$ as the graph of a $2n \times 2n$ matrix in the standard coordinates on $\R^{2n}$ instead of changing coordinates and representing $\ell(t)$ as the graph of an $n \times n$ matrix. 
\end{remark}

%-------------------------------------------------------------------------------------------------------------------------
\subsection{Extending the Crossing Form to Nonregular Crossings}\label{section: maslov index for nonregular crossings}
%-------------------------------------------------------------------------------------------------------------------------
The formulation of the Maslov index discussed in the previous section only applies to regular crossings with the critical exception of Theorem \ref{thm: spectral flow RS}. We will show in Section \ref{S:count} that the Swift-Hohenberg equation has non-regular crossings and as a result, the formulation of the Maslov index using the crossing form that was previously introduced does not apply. A framework for computing the Maslov index for paths with second-order degenerate crossings has been established in the infinite dimensional case in \cite{dengjones11}. As we will see in in Section \ref{S:count}, the Swift-Hohenberg equation has third order degenerate crossings so we cannot directly apply a finite dimensional version of this result. We modify the result in \cite{dengjones11} and extend the theory introduced in Section \ref{subsec: crossing form} to degenerate crossings of general order in the finite dimensional case by using techniques similar to those used in \cite{furutani}. 

\begin{definition}\label{def: higher order quadratic forms}
Suppose that $\ell: [a,b] \to \Lambda(n)$ is a $C^m$ curve of Lagrangian subspaces and choose $\mathcal W$ such that $\ell(t_0) \cap  \mathcal W = \{0\}$ for fixed $t_0 \in [a,b]$. For  $v \in \ell(t_0)$ and $t$ near $t_0$ suppose the graph of $A(t): \ell(t_0) \to \mathcal W$ is $\ell(t)$ and fix $j \leq m$. Then, the quadratic form of order $j$ associated with $\ell(t)$ at $t = t_0$ acting on $v$ is defined as 
\begin{equation}\label{eq: nth order quadratic form}
Q^{(j)}(v) = \frac{d^j}{dt^j}\omega(v,A(t)v) \bigg|_{t = t_0}.
\end{equation}
% If we express $\ell(t) = \{(\tilde u, A(t)\tilde u) \ : \ \tilde u \in \ell(t_0)\}$ with $A: \ell(t_0) \to \mathcal W$ and $v = L(t_0)\tilde v$, then the $j$th order quadratic form associated with $\ell$ has the form 
% \begin{equation}\label{eq: nth order quadratic form, graph formulation}
% Q^{(j)}(v) =  \left\langle \tilde v, A^{(j)}(t) \tilde v \right\rangle \bigg|_{t = t_0}.
% \end{equation}
Fix a Lagrangian plane $\ell_*$ and we suppose at $t_0$, $\ell(t_0) \cap \ell_* \neq \{0\}$. Then, the crossing form of order $j$ with respect to $\ell(t)$ and $\ell_*$ is defined to be 
\begin{equation}\label{eq: nth order crossing form}
 \Gamma^{(j)}(\ell, \ell_*, t_0) =  Q^{(j)}(v), \qquad v \in \ell(t_0) \cap \ell_*.
\end{equation}
\end{definition}

Below we prove a generalization of parts a) and b) of Theorem \ref{thm: RS main result} for degenerate crossings via a different approach than the one used in \cite{robbinsalamon}. Refer to Remark \ref{rmk: symplectic coordinates reference} for a brief discussion of symplectic coordinates. 

% \begin{lemma}
% Suppose that $\ell(t) \in \Lambda(n)$ be a $C^j$ curve of Lagrangian subspaces with frame $L(t)$ and $\ell(t)$ can be represented as the graph of $A(t): \mathcal W \to \mathcal W^\perp$ for $t$ sufficiently close to $t_0$. In other words, there exists $c(t) = \begin{pmatrix} c_1(t) & \dots & c_n(t)
% \end{pmatrix}$ such that for any $v \in \ell(t_0)$,
% $$L(t)c(t) = v + A(t)v.$$ 
% Then, the derivatives of $c(t_0)$ can be represented in terms of the derivatives of $L(t_0)$ and projections onto $L(t_0)$.
% \end{lemma}
% \begin{proof}
% Define the projection matrix onto $\ell(t_0)$ as $\Pi_{\ell(t_0)}$. 
% \end{proof}

\begin{theorem}\label{thm: higher order quad form RS extension} Suppose we are in the setting of Definition \ref{def: higher order quadratic forms}. The higher order crossing form of order $j$ is
\begin{enumerate}[a)]
\item invariant under symplectic changes of coordinates; 
\item independent of the choice of $\mathcal W$. 
\end{enumerate}
\end{theorem}
\begin{proof} ~\\
\begin{enumerate}[a)] 
\item Suppose that $\ell(t)$ has frame given by $L(t)$ and fix $v \in \ell(t_0)$. Let $\mathcal W$ and $A(t)$ be as in Definition \ref{def: higher order quadratic forms}. Additionally, let $\Psi$ be a symplectic matrix and define the following quantities 
$$ v_\Psi = \Psi v, \qquad \mathcal W_\Psi = \text{colspan} \big( \Psi W\big), \qquad \ell_\Psi(t) = \text{colspan} \big(\Psi \ell(t)\big).$$

Via \eqref{eq: graph A satisfies equation}, we have that 
$$v + A(t)v = L(t)c(t).$$
By left multiplying by $\Psi$ and recalling that $v = \Psi^{-1}v_\Psi$, we can write 
$$ \Psi v + \Psi A(t)v = v_\Psi + \Psi A(t)\Psi^{-1}v_\Psi = \Psi L(t)c(t).$$
Thus, the graph of the matrix $A_\Psi(t):= \Psi A(t)\Psi^{-1}$ is $\ell_\Psi(t)$.

Define $Q_\Psi(v): = \omega(v_\Psi, A_\Psi(t) v_\Psi )$ and observe the following inner product computation:
  \begin{align*} Q_\Psi(v) & := \omega(v_\Psi, A_\Psi(t) v_\Psi )  \\
  & = \langle v_\Psi, JA_\Psi(t)v_\Psi\rangle \\
  & = \langle \Psi v, JA_\Psi(t) \Psi v\rangle \\
  & = \langle v, \Psi^T JA_\Psi(t)\Psi v \rangle \\
  & = \langle v, J\Psi^{-1}A_\Psi(t) \Psi v\rangle \\
  & = \langle v, JA(t) v\rangle,
  \end{align*}
where we have used that $\Psi$ is a symplectic matrix, so $\Psi^{-1} = J^{-1}\Psi^T J$. Thus, 
$$Q_\Psi(v) = \langle v, J\Psi^{-1}A_\Psi \Psi v\rangle = \langle v, JA(t)v\rangle: = Q(v).$$
Since this is true for $t$ near $t_0$, the derivatives of $Q_\Psi(v)$ and $Q(v)$ coincide and we see that the crossing forms are invariant under symplectic changes of coordinates. 
\item 
  By the previous claim and \cite{mcduffsalamon17}[Lemma 2.3.2], we can change coordinates and assume that $\ell(t_0) = \text{col span} \begin{pmatrix} I_n & 0 \end{pmatrix}^T$ and $\mathcal W = \text{col span} \begin{pmatrix} B & I_n \end{pmatrix}^T$ for some symmetric matrix $B$. Any $\mathcal W$ having trivial intersection with $\ell(t_0)$ will have this structure and is determined by the matrix $B$.  Suppose that the graph of $A: \ell(t_0) \to \mathcal W$ is $\ell(t)$ for $t$ near $t_0$. Write $A$ with the block structure
  $$A(t) = \begin{pmatrix} A_{11}(t) & A_{12}(t) \\ A_{21}(t) &  A_{22}(t)
  \end{pmatrix},$$
  with the block matrices $A_{ij}(t) \in \R^{n \times n}$. Since $A(t)$ is required to map into $\mathcal W$ we see that for some $y(t) \in \R^n$, we have $$ A(t) \begin{pmatrix} x \\ 0
  \end{pmatrix} = \begin{pmatrix} A_{11}(t) x \\ A_{21}(t)x
  \end{pmatrix} = \begin{pmatrix} By(t) \\ y(t)
  \end{pmatrix}.$$
  This equality holds if $y(t) = A_{21}(t)x$ and $A_{11}(t) = BA_{21}(t)$. Thus, the matrix $A(t)$ depends on the choice of $\mathcal W$. However, by computing the quadratic form, we see that 
  \begin{align*} \omega(v, A(t)v) & = \left\langle \begin{pmatrix} x \\ 0
  \end{pmatrix}, \begin{pmatrix} 0 & I_n \\ - I_n & 0
  \end{pmatrix} \begin{pmatrix}
  By(t) \\ y(t)
  \end{pmatrix} \right\rangle\\
  & = \langle x, y(t)\rangle \\
  & = \langle x, A_{21}(t)x\rangle.
  \end{align*}
  Since $A_{21}(t)$ depends only on $\ell(t)$ and not on $\mathcal W$, we see the quadratic form is independent of the choice of $\mathcal W$. Because this is true for all $t$ near $t_0$, it follows that the derivatives of $\omega(v, A(t)v)$ do not depend on $\mathcal W$ either. 
\end{enumerate}
\end{proof}

The goal is to now connect the sign of the higher order crossing forms to the behavior of the eigenvalues of $JA(t)$ for $t$ near $t_0$. In the case of regular crossings, we can quantify how many positive eigenvalues the matrix $JA(t)$ gains or loses as $t$ increases through $t_0$ via taking a first derivative and forming the first order crossing form. In the case of non-regular crossings, we can relate the change in the number of positive eigenvalues of $JA(t)$ on a small interval around $t_0$ to the sign of higher order crossing forms, which are formed by higher order derivatives. We first extend Lemma \ref{lemma: furutani Q1} to a more general setting. In order to do so, we will need the following two technical results. 

\begin{lemma}\label{claim: proj speed} Given a collection of $p$ orthonormal vectors $\{u_i(t)\}_{i=1}^p\subset \R^n$ that are $C^1$ on $(-\epsilon, \epsilon)$, define $C(t) = [u_1(t), \dots, u_p(t)] \in \R^{n \times p}$ and the projection matrix onto the subspace spanned by $\{u_i(t)\}_{i=1}^p$ as $P_t = C(t)C(t)^T. $
Suppose that $y \in \text{span}\{u_i(t_0)\}_{i=1}^p$ so that $P_0 y = y$ and fix $m \geq 1$. Then, the function $F(t^m) := P_{t^m}y - y$ satisfies 
 $$ \lim_{t \to 0} \frac{|F(t^m)|}{|t^m|} = \lim_{t \to 0}\frac{|P_{t^m}y - y |}{|t^m|} <  C,$$
where $|.|$ denotes the Euclidean norm on $\R^n$ and $C$ is a finite constant independent of $m$. 
\end{lemma}
\begin{proof} First notice $F \in C^1\big((-\epsilon, \epsilon), \R^n\big)$ by definition. For $t \in (-\epsilon, \epsilon)$, we can differentiate $F$ to obtain
\begin{align*}|F(t^m)| &  = \left| \int_0^t F'(\tau^m) m \tau^{m-1} \ d\tau \right| \\
& \leq \sup_{0 < \tau \leq t} |F'(\tau^m)| \left|\int_0^t m\tau^{m-1} \ d\tau \right| \\
& = \sup_{0 < \tau \leq t} |F'(\tau^m)| |t^m|.
\end{align*}
Thus, 
\begin{align*}\lim_{t \to 0} \frac{|F(t^m)|}{|t^m|}  & \leq \lim_{t \to 0}\sup_{0 < \tau \leq t} \frac{ |F'(\tau^m)| |t^m|}{|t^m|} = |F'(0)| < \infty;
\end{align*}
since $F$ is continuously differentiable. 
\end{proof}

Lemma \ref{claim: proj speed} is necessary for the proof of the following technical lemma. 
\begin{lemma}\label{lemma: general crossing intermediate lemma}
\ADD{Suppose $\{B(t)\}$ is a $C^m(\R, \R^{2n \times 2n})$ family of $2n \times 2n$ self-adjoint matrices and fix $t_0$. For $u \in \ker B(t_0)$ define the quadratic form 
$$R^{(i)}(u): = \frac{d^i}{dt^i} \left\langle u, B(t)u \right\rangle \bigg|_{t=t_0} = \left\langle u,  B^{(i)}(t_0)u \right\rangle.$$
For fixed $j \leq m$, assume the following: 
\begin{enumerate}[i)]
\item The quadratic forms $R^{(j)}(u)$ for $u \in \ker B(t_0)$ are degenerate for $i = 1, \dots, j-1$.
\item The $j$th order quadratic form $R^{(j)}$ on $\ker B(t_0)$ is non-degenerate.
\end{enumerate}
Furthermore, for $\epsilon > 0$, let $P_t$ be the spectral projection onto the eigenspace of $B(t)$, restricted to eigenvalues within $\epsilon$ of $0$. In other words, $P_t$ projects onto $G_t$ with $G_t$ defined as
 $$G_t =\big\{\ker (cI_n - B(t)) \ : \ \forall c \in \R, |c| < \epsilon \big\}.$$ 
Then, on the kernel of $B(t_0)$, 
$$ \lim_{t \to t_0}j!\left\langle P_{(t - t_0)^j}u, \frac{1}{(t - t_0)^j}B(t - t_0) P_{(t - t_0)^j}u \right\rangle = \langle u, B^{(j)}(t_0)u \rangle.$$}
\end{lemma}
\begin{proof} Without loss of generality, set $t_0 = 0$. We can add and subtract terms to compute 
\begin{align*} \left\langle P_{t^j}u, \frac{1}{t^j}B(t) P_{t^j}u \right\rangle &  - \frac{1}{j!} \langle u, B^{(j)}(0)u \rangle \\
&  =  \left\langle P_{t^j}u, \frac{1}{t^j}B(t) P_{t^j}u \right\rangle  - \frac{1}{j!} \langle u, B^{(j)}(0)u \rangle \pm \left\langle P_{t^j}u, \left[ \sum_{i = 1}^{j-1} \frac{t^{i-j}}{i!} B^{(i)}(0) \right] P_{t^j}u \right\rangle  \\
& \qquad \pm \left\langle u, \frac{1}{j!}B^{(j)}(0) P_{t^j}u \right\rangle \pm \left\langle P_{t^j}u, \frac{1}{t^j}B(0)P_{t^j}u \right\rangle  \pm \left\langle P_{t^j}u, \frac{1}{j!}B^{(j)}(0) P_{t^j}u \right\rangle \\
&  =  \underbrace{\left\langle P_{t^j}u, \left[\frac{1}{t^j} \left(B(t) - B(0) - \sum_{i=1}^{j-1} \frac{t^i}{i!} B^{(i)}(0) \right) - \frac{1}{j!} B^{(j)}(0) \right]P_{t^j}u \right\rangle}_{I}\\
& \qquad + \underbrace{\frac{1}{j!} \left\langle u, B^{(j)}(0)\left[P_{t^j}u - u \right]\right\rangle}_{II} + \underbrace{\frac{1}{j!}\left\langle P_{t^j}u - u, B^{(j)}(0)P_{t^j}u \right\rangle}_{III} \\
& \qquad + \underbrace{\left\langle P_{t^j}u, \left[\sum_{i=0}^{j-1}\frac{t^{i-j}}{i!} B^{(i)}(0) \right] P_{t^j}u \right\rangle}_{IV}.
\end{align*}

We will argue that each of these terms vanishes as $t \to 0$ to obtain the desired conclusion. We can use Taylor's remainder theorem to see that term $I$ vanishes,
\begin{align*} \lim_{t \to 0} \left\langle P_{t^j}u , \left[\frac{1}{t^j} \left(B(t) - B(0) - \sum_{i=1}^{j-1} \frac{t^i}{i!} B^{(i)}(0) \right) \right.\right. & \left.\left. - \frac{1}{j!} B^{(j)}(0) \right]P_{t^j}u \right\rangle  \\
&  = \left \langle u, \frac{1}{j!} \big(B^{(j)}(0) - B^{(j)}(0)\big)u \right\rangle = 0.
\end{align*}
Now looking at terms $II$ and $III$, we can use $u \in \ker B(0)$ and $\lim_{t \to 0} P_{t^j}u = u$ to write
\begin{align*} \lim_{t \to 0} \left\{\frac{1}{j!} \left\langle u, B^{(j)}(0)\left[P_{t^j}u - u \right]\right\rangle + \frac{1}{j!}\left\langle P_{t^j}u - u, B^{(j)}(0)P_{t^j}u \right\rangle \right\} & = \frac{1}{j!}\bigg\{ \left\langle u,0 \right\rangle + \left\langle0, B^{(j)}u \right\rangle \bigg\} = 0.
\end{align*}
Finally, we show term $IV$ vanishes in the limit. We must treat the $i = 0$ and $i \neq 0$ terms in the sum separately. 
\begin{align*} \underbrace{\left\langle P_{t^j}u, \left[\sum_{i=0}^{j-1}\frac{t^{i-j}}{i!} B^{(i)}(0) \right] P_{t^j}u \right\rangle}_{IV} & = \underbrace{\left\langle P_{t^j}u, \frac{1}{t^j}B(0)P_{t^j}u \right\rangle}_{IV(a)} + \underbrace{\left\langle P_{t^j}u, \left[\sum_{i=1}^{j-1}\frac{t^{i-j}}{i!} B^{(i)}(0) \right] P_{t^j}u \right\rangle}_{IV(b)}.
\end{align*}
First addressing $IV(a)$, we can use that $B(0)u = 0$, Lemma \ref{claim: proj speed} and the Cauchy-Schwarz inequality to write 
\begin{align*}\lim_{t \to 0} \left|\left\langle P_{t^j}u, \frac{1}{t^j}B(0)P_{t^j}u \right\rangle \right| & = \lim_{t \to 0} \left|\left\langle P_{t^j}u, \frac{1}{t^j}B(0)P_{t^j}u \right\rangle - \left\langle P_{t^j}u, \frac{1}{t^j} B(0) u \right\rangle\right| \\
& = \lim_{t \to 0} \left|\left\langle  B(0)P_{t^j}u, \frac{1}{t^j}  \big[P_{t^j}u - u\big] \right\rangle \right| \\
& \leq \lim_{t \to 0} \left\|  B(0)P_{t^j}u \right\| \left\|\frac{1}{t^j}  \big[P_{t^j}u - u\big] \right\| \\
& = \left\| B(0)u\right\| C, \qquad \text{($C$ constant)} \\
& = 0.
\end{align*}
For term $IV(b)$ we can use  $\langle u, B^{(i)}(0)u \rangle = 0$ for $i = 1, \dots, j-1$ to write 
\begin{align*}\left\langle P_{t^j}u, \frac{1}{t^{j-i}}B^{(i)}(0)P_{t^j}u  \right\rangle & = \left\langle P_{t^j}u, \frac{1}{t^{j-i}}B^{(i)}(0)P_{t^j}u  \right\rangle \pm \left\langle P_{t^j}u, \frac{1}{t^{j-i}}B^{(i)}(0)u  \right\rangle \\
& \qquad - \left\langle u, \frac{1}{t^{j-i}}B^{(i)}(0)u\right\rangle\\
& = \left\langle P_{t^j}u, \frac{1}{t^{j-i}}B^{(i)}(0)\left[P_{t^j}u-u\right] \right\rangle + \left\langle \frac{1}{t^{j-i}}\left[P_{t^j}u - u\right], B^{(i)}(0)u  \right\rangle.
\end{align*}
Via Lemma \ref{claim: proj speed}, we know $P_{t^j}u - u = \mathcal O(t^j)$, meaning the numerator in each term is approaching $0$ faster than the denominator. Thus, 
\begin{align*} \lim_{t \to 0} \left\langle P_{t^j}u, \frac{1}{t^{j-i}}B^{(i)}(0)P_{t^j}u  \right\rangle  & = \lim_{t \to 0} \left\{ \left\langle P_{t^j}u, \frac{1}{t^{j-i}}B^{(i)}(0)\left[P_{t^j}u-u\right] \right\rangle + \left\langle \frac{1}{t^{j-i}}\left[P_{t^j}u - u\right], B^{(i)}(0)u  \right\rangle\right\} \\
& = \left\langle u, 0 \right\rangle + \left\langle 0, B^{(i)}(0)u \right\rangle \\
& = 0.
\end{align*}
Thus, as $t \to 0$, we have that $\left\langle  P_{t^j}u, \frac{1}{t^j}B(t) P_{t^j}u \right\rangle \to \frac{1}{j!}\left\langle u, B^{(j)}(0) u \right\rangle$. Multiplying by $j!$ gives the desired result.
\end{proof}

With these technical lemmas, we prove the following spectral result that generalizes Lemma \ref{lemma: furutani Q1}. Heuristically, Lemma \ref{lemma: deng/jones general ops} says that the parity of the order of the lowest non-degenerate higher order crossing form tells you if and how the number of positive eigenvalues of $B(t)$ changes as $t$ increases through $t_0$. 
\begin{lemma}\label{lemma: deng/jones general ops}  
Suppose we are in the setting of Lemma \ref{lemma: general crossing intermediate lemma}. Recall that $\dim \ker B(t_0) = k$ and $j$ is the lowest order non-degenerate quadratic form on the kernel of $B(t_0)$. Then, there exists $p, q \in \N \cup \{0\}$ with $p + q = k$ and $\epsilon, \delta > 0$ such that 
\begin{enumerate}[i)]
\item for $t_0 < t \leq t_0 + \delta$ the operator $B(t)$ has $k$ eigenvalues within $\epsilon$ of $0$ with $p$ of them positive and $q$ of them negative; 
\item $\text{sign} R^{(j)} = p - q$;
\item if $j$ is odd and $t_0-\delta \leq t < t_0$, then the operator $B(t)$ has $q$ positive eigenvalues and $p$ negative eigenvalues within $\epsilon$ of $0$;
\item if $j$ is even and $t_0-\delta \leq t < t_0$, then the operator $B(t)$ has $p$ positive eigenvalues and $q$ negative eigenvalues within $\epsilon$ of $0$.
\end{enumerate}
\end{lemma}
\begin{proof} Without loss of generality, we set $t_0 = 0$. \\\\
\noindent\textit{Proof of i):} Since $B(t)$ is continuously differentiable and self-adjoint, the eigenvalues $\{\lambda_i(t)\}_{i=1}^n$ and eigenvectors $\{u_i(t)\}_{i=1}^n$ perturb smoothly in $t$ \cite{kato}. Additionally, $\dim\ker B(0) = k$, so there is an $\epsilon$ and a $\delta$ such that if $P_t$ is the spectral projection onto  
$$G_t :=\big\{\ker (cI_n - B(t)) \ : \ c \in \R, |c| < \epsilon \big\}$$ 
and $0 < |t| < \delta$ then $P_t$ has constant rank equal to $\dim \ker B(0) = k$. Note that this means for $t \in (- \delta,  \delta)$, $B(t)$ has $k$ eigenvalues, $\lambda_i(t)$ such that $0 < | \lambda_i(t) | < \epsilon$. Without loss of generality, we can suppose $p$ of them are positive and $q$ of them are negative for $t \in (0, \delta]$. \\\\
% \begin{align*} 0 & \leq \lambda_1(t) \leq \dots \leq \lambda_p(t) \leq \epsilon \\
% 0 & > \lambda_{-1}(t)  \geq \dots \geq \lambda_{-q}(t) \geq - \epsilon.
% \end{align*}

\noindent\textit{Proof of ii):} Note the spectrum of $B(t)$ restricted to an $\epsilon$ neighborhood of $0$ for $t$ sufficiently small is given by the spectrum of $P_tB(t)P_t$ and fix $t > 0$ sufficiently small. We can compute,
\begin{align*} p-q & = \text{sign} P_t B(t)P_t & \text{$\big($via previous work in part $i)\big)$} \\
& = \text{sign} P_{t^j}B(t)P_{t^j} & \text{($P_t$ perturbs smoothly in $t$)} \\
& = \text{sign} \frac{1}{t^j}P_{t^j}B(t)P_{t^j} & \text{($t^j > 0$)} \\
& = \text{sign}R^{(j)} & \text{(Lemma \ref{lemma: deng/jones general ops})}.
\end{align*}

\noindent\textit{Proof of iii):} Suppose $-\delta \leq t < 0$ and $j$ is odd. We can compute 
\begin{align*} p - q & = \text{sign}R^{(j)} & \text{$\big($via previous work in part $ii)\big)$} \\
& = \text{sign} \frac{1}{t^j} P_{t^j}B(t)P_{t^j} & \text{(Lemma \ref{lemma: deng/jones general ops})} \\
& = -\text{sign}P_{t^j}B(t)P_{t^j} & \text{($t^j < 0$)} \\
& = -\text{sign}P_tB(t)P_t & \text{($P_t$ perturbs smoothly in $t$)}.
\end{align*}
Thus, $\text{sign} P_tB(t) P_t = q - p$, implying $B(t)$ has $q$ positive eigenvalues and $p$ negative eigenvalues within $\epsilon$ of $0$ for $t \in [-\delta, 0)$.\\\\

 %We now argue that $P_tB(t)P_t$ has $p$ positive eigenvalues and $q$ negative eigenvalues for $t \in [-\delta, 0)$. Suppose for $t \in [-\delta, 0)$ that $B(t)$ has $\tilde p$ positive eigenvalues and $\tilde q$ negative eigenvalues. Using that $\tilde p + \tilde q = p + q = k$ and $\tilde p - \tilde q = q -p$, we can solve for $\tilde q$ and $\tilde p$ to see that $\tilde p = q$ and $\tilde q = p$. 

\noindent\textit{Proof of iv):} This follows via a virtually identical argument for that of $iii)$. 
\end{proof}

\begin{remark} This result can be viewed as a finite-dimensional generalization both of \cite{furutani}[Lemma 3.23], which contains the first order result, and of \cite{dengjones11}[Lemma 5.5], which contains the second order result. 
\end{remark}

In order to relate the result of Lemma \ref{lemma: deng/jones general ops} to the Maslov index, we need the following generalization of Corollary \ref{corollary: regular crossing evals}.
\begin{corollary}\label{corollary: general tan evals} Suppose we are in the setting of Lemma \ref{lemma: deng/jones general ops}. Recall that $\dim\ker B(t_0) = k$ and $j$ is the lowest order non-degenerate quadratic form on the kernel of $B(t_0)$. Furthermore, suppose that $\text{sign} R^{(j)} = p-q$. 
\begin{enumerate}[i)]
\item Suppose $j$ is even. Then the signature of the operator $B(t)$ does not change as $t$ increases through $t_0$. 
\item Suppose $j$ is odd. If $p - q > 0$ then $p-q$ is the number of positive eigenvalues the operator $B(t)$ gains as $t$ increases through $t_0$. If $p-q < 0$, then $|p-q|$ is the number of positive eigenvalues lost as $t$ increases through $t_0$. 
\end{enumerate}
\end{corollary}
% \begin{proof} If $n$ is odd, this argument is virtually identical to that in Corollary \ref{corollary: regular crossing evals}. If $n$ is even, this is immediate from the fact that there are $p$ positive eigenvalues and $q$ negative eigenvalues within $\epsilon$ of $0$ for $|t| < \delta$ and that the rest of the $n - p - q$ eigenvalues of $A(t)$ are bounded away from $0$. 
% \end{proof}
\begin{proof}
This follows directly from Lemma \ref{lemma: deng/jones general ops}.
\end{proof}

\begin{example}
We can visualize the results in Lemma \ref{lemma: deng/jones general ops} and Corollary \ref{corollary: general tan evals} with Figures \ref{fig: B j even} and \ref{fig: B j odd}. 

Suppose that $B(t)$ is an $n \times n$ self-adjoint matrix and $\dim \ker B(t_0) = 3$. Lemma \ref{lemma: deng/jones general ops} says that there exists $\epsilon, \delta > 0$ such that for $t \in (t_0, t_0 + \delta]$, $B(t)$ has $3$ eigenvalues within $\epsilon$ of $0$. Suppose there is $1$ positive eigenvalue, $\lambda_1(t)$ and $2$ negative eigenvalues, $\lambda_2(t)$ and $\lambda_3(t)$. Note that $B(t)$ has $n-3$ other eigenvalues but they do not play a role in this analysis. In the notation of Lemma \ref{lemma: deng/jones general ops} $p = 1$ and $q = 2$. Suppose for all $u \in \ker B(t_0)$,
$$R^{(1)}(u) = \left\langle u, \frac{d}{dt}B(t)u \right\rangle \bigg|_{t = t_0} = 0 \quad \text{and} \quad R^{(2)}(u) = \left\langle u, \frac{d^2}{dt^2}B(t)u\right\rangle \big|_{t=t_0} \neq 0.$$ 
Then, Corollary \ref{corollary: general tan evals} says that $\lambda_i(t)$ for $i = 1,2,3$ do not change sign as $t$ increases on the interval $[t_0 - \delta, t_0 + \delta]$. This is illustrated in Figure \ref{fig: B j even}.

\begin{figure}[H]
  \centering
\begin{subfigure}[t]{0.3\textwidth}
        \centering
        \begin{tikzpicture}[scale=.75]
          \draw[thick, -] (-3, 0) -- (3, 0); 
          \draw[thick, -] (0, -.05) -- (0, .05) node[below] {$0$};  
          \draw[thick, -] (-3, -.05) -- (-3, .05)  node[below] {$-\epsilon$}; 
          \draw[thick, -] (3, -.05) -- (3, .05) node[below] {$\epsilon$}; 
          \filldraw[red] (-1.5, 0) circle (3pt) node[above left] {$\lambda_2(t)$};
          \filldraw[blue] (-.5, 0) circle (3pt) node[above] {$\lambda_3(t)$};
          \filldraw[orange] (2, 0) circle (3pt) node[above] {$\lambda_1(t)$};
          \draw[red][->] (-1.5, -.2) -- (-1, -.2);
          \draw[blue][->] (-.5, -.2) -- (-.15, -.2);
          \draw[orange][<-] (1.5, -.2) -- (2, -.2);
        \end{tikzpicture}
  \caption{$t \in [t_0 - \delta, t_0)$}
  \end{subfigure}
\hfill
    \begin{subfigure}[t]{0.3\textwidth}
      \centering
      \begin{tikzpicture}[scale=.75]
        \draw[thick, -] (-3, 0) -- (3, 0);
        \draw[thick, -] (0, -.05) -- (0, .05) node[below] {$0$};  
        \draw[thick, -] (-3, -.05) -- (-3, .05)  node[below] {$-\epsilon$}; 
        \draw[thick, -] (3, -.05) -- (3, .05) node[below] {$\epsilon$}; 
        \filldraw[black] (0, 0) circle (3pt) node[above] {$\substack{\color{orange}{\lambda_1(t_0)} \\ \color{red}{\lambda_2(t_0)} \\ \color{blue}{\lambda_3(t_0)}}$};
      \end{tikzpicture}
    \caption{$t = t_0$} 
    \end{subfigure}
    \hfill
    \begin{subfigure}[t]{0.3\textwidth}
    \centering 
    \begin{tikzpicture}[scale=.75]
      \draw[thick, -] (-3, 0) -- (3, 0);
      \draw[thick, -] (0, -.05) -- (0, .05) node[below] {$0$};  
      \draw[thick, -] (-3, -.05) -- (-3, .05)  node[below] {$-\epsilon$}; 
        \draw[thick, -] (3, -.05) -- (3, .05) node[below] {$\epsilon$}; 
        \filldraw[red] (-.5, 0) circle (3pt) node[above] {$\lambda_2(t)$};
        \filldraw[blue] (-2.5, 0) circle (3pt) node[above] {$\lambda_3(t)$};
        \filldraw[orange] (.5, 0) circle (3pt) node[above right] {$\lambda_1(t)$};
        \draw[red][->] (-.15, -.2) -- (-.5, -.2);
        \draw[blue][->] (-2, -.2) -- (-2.5, -.2);
        \draw[orange][->] (.15, -.2) -- (.5, -.2);
    \end{tikzpicture}
    \caption{$t \in (t_0, t_0 + \delta]$} 
    \end{subfigure}
\caption{The motion of the eigenvalues of $B(t)$ within $\epsilon$ of $0$ in the case that $R^{(2)}(u)$ is the lowest order non-vanishing quadratic form.}\label{fig: B j even}
\end{figure}
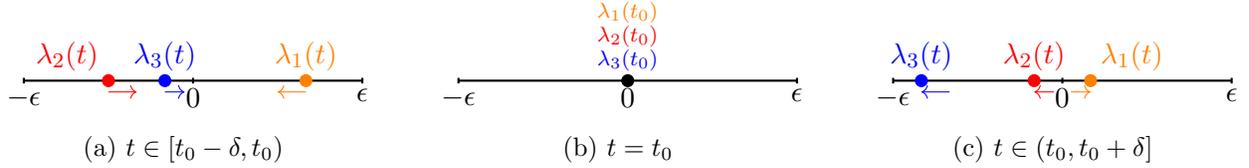

Now suppose that we are in the same setting on the interval $(t_0, t_0 + \delta]$; meaning that $\dim \ker B(t_0) = 3$ and $B(t)$ has $1$ positive eigenvalue $\lambda_1(t)$ and $2$ negative eigenvalues $\lambda_2(t)$ and $\lambda_3(t)$ within $\epsilon$ of $0$ so that $p = 1$ and $q = 2$. However, now suppose for all $u \in \ker B(t_0)$, 
$$R^{(i)}(u) = \left\langle u, \frac{d^i}{dt^i} B(t)u \right\rangle \bigg|_{t = t_0} = 0, \ i = 1,2 \quad \text{and} \quad  R^{(3)}(u) = \left\langle u, \frac{d^3}{dt^3} B(t)u \right\rangle \big|_{t = t_0} \neq 0.$$  
Then, Corollary \ref{corollary: general tan evals} says that the three eigenvalues $\lambda_1(t), \lambda_2(t)$ and $\lambda_3(t)$ change sign as $t$ increases through $t_0$ and that the number of positive eigenvalues gained by the matrix $B(t)$ as $t$ increases on the interval $[t_0 - \delta, t_0 + \delta]$ is given by $\text{sign}R^{(3)} = p-q = 1-2 = -1$. In other words, $B(t)$ has one fewer positive \ADD{eigenvalues} on the interval $(t_0, t_0 + \delta]$ than on $[t_0 - \delta, t_0)$. This is illustrated in Figure \ref{fig: B j odd}.

\begin{figure}[H]
  \centering
\begin{subfigure}[t]{0.3\textwidth}
        \centering
      \begin{tikzpicture}[scale=.75]
          \draw[thick, -] (-3, 0) -- (3, 0);  
          \draw[thick, -] (0, -.05) -- (0, .05) node[below] {$0$};  
          \draw[thick, -] (-3, -.05) -- (-3, .05)  node[below] {$-\epsilon$}; 
          \draw[thick, -] (3, -.05) -- (3, .05) node[below] {$\epsilon$}; 
          \filldraw[red] (1.5, 0) circle (3pt) node[above] { $ \quad\lambda_2(t)$};
          \filldraw[blue] (.5, 0) circle (3pt) node[above] {$\lambda_3(t)$};
          \filldraw[orange] (-2, 0) circle (3pt) node[above] {$\lambda_1(t)$};
          \draw[red][->] (1.5, -.2) -- (1, -.2);
          \draw[blue][->] (.5, -.2) -- (.15, -.2);
          \draw[orange][->] (-2, -.2) -- (-1.5, -.2);
        \end{tikzpicture}
  \caption{$t \in [t_0 - \delta, t_0)$}
  \end{subfigure}
\hfill
    \begin{subfigure}[t]{0.3\textwidth}
      \centering
      \begin{tikzpicture}[scale=.75]
        \draw[thick, -] (-3, 0) -- (3, 0);
        \draw[thick, -] (0, -.05) -- (0, .05) node[below] {$0$};  
        \draw[thick, -] (-3, -.05) -- (-3, .05)  node[below] {$-\epsilon$}; 
        \draw[thick, -] (3, -.05) -- (3, .05) node[below] {$\epsilon$}; 
        \filldraw[black] (0, 0) circle (3pt) node[above] {$\substack{\color{orange}{\lambda_1(t_0)} \\ \color{red}{\lambda_2(t_0)} \\ \color{blue}{\lambda_3(t_0)}}$};
      \end{tikzpicture}
    \caption{$t = t_0$} 
    \end{subfigure}
    \hfill
    \begin{subfigure}[t]{0.3\textwidth}
    \centering 
    \begin{tikzpicture}[scale=.75]
        \draw[thick, -] (-3, 0) -- (3, 0);
        \draw[thick, -] (0, -.05) -- (0, .05) node[below] {$0$};  
        \draw[thick, -] (-3, -.05) -- (-3, .05)  node[below] {$-\epsilon$}; 
        \draw[thick, -] (3, -.05) -- (3, .05) node[below] {$\epsilon$}; 
        \filldraw[red] (-1, 0) circle (3pt) node[above] {$\qquad\lambda_2(t)$};
        \filldraw[blue] (-1.5, 0) circle (3pt) node[above] {$\lambda_3(t) \qquad$};
        \filldraw[orange] (2.5, 0) circle (3pt) node[above] {$\lambda_1(t)$};
        \draw[red][->] (-.5, -.2) -- (-1, -.2);
        \draw[blue][->] (-1.1, -.2) -- (-1.5, -.2);
        \draw[orange][->] (2, -.2) -- (2.5, -.2);
    \end{tikzpicture}
    \caption{$t \in (t_0, t_0 + \delta]$} 
    \end{subfigure}
\caption{The motion of the eigenvalues of $B(t)$ within $\epsilon$ of $0$ in the case that $R^{(3)}(u)$ is the lowest order non-vanishing quadratic form.}\label{fig: B j odd}
\end{figure}
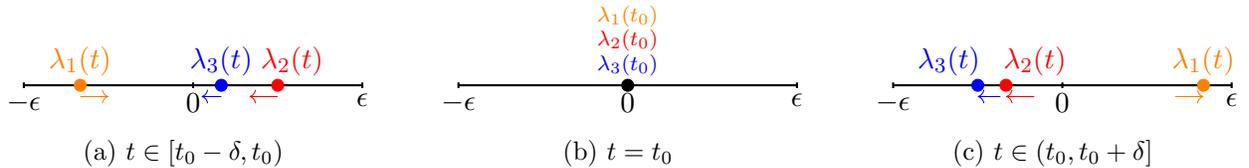
\end{example}

Finally, we state the main result that characterizes the Maslov index of paths via higher order crossing forms.

\begin{theorem}\label{theorem: maslov gen tang} Let $\ell: [a,b] \to \Lambda(n)$ be $C^k$ and fix $j \leq k$. Suppose $t_0 \in (a,b)$ is an isolated crossing with a fixed reference plane $\ell_*$ and $v \in \ell(t_0) \cap \ell_*$. If every crossing form of order $i$ with $i = 1, \dots, j-1$ is degenerate at the crossing
$$Q^{(i)}(v) \equiv 0, \qquad i = 1, \dots, j-1 \text{ and }v \in \ell(t_0) \cap \ell_*$$ 
and 
$Q^{(j)}(v) \big|_{\ell(t_0) \cap \ell_*}$ is non-degenerate, then there exists $\delta > 0$ such that 
\begin{enumerate}[i)]
\item if $j$ is even, $$ \text{Mas}(\ell(t), \ell_*; \ |t-t_0| < \delta) = 0;$$
\item if $j$ is odd,$$ \text{Mas}(\ell(t), \ell_*; \ |t-t_0| < \delta) = \text{sign}Q^{(j)}(v).$$
\end{enumerate}

If $t_0$ occurs at an endpoint, the contribution to the Maslov index from the crossing is given by $\frac{1}{2} \text{sign}Q^{(j)}(v).$
\end{theorem}

\begin{remark}
It is important to recognize that this framework as stated only applies when the crossing forms of order $1, \dots, j-1$ are degenerate on the entire kernel of $B(t_0)$. This is sufficient for our application to the Swift-Hohenberg equation. 
\end{remark}
\begin{proof}[Proof of Theorem \ref{theorem: maslov gen tang}.] Following the discussion at the beginning of this section, we first write $\ell(t)$ as a graph near the crossing. Assume $\ell(t_0) \cap \ell_* \neq \{0\}$ and $\ell(t_0) \cap \mathcal W = \{0\}$ for some other Lagrangian subspace $\mathcal W$. Let $A(t): \ell(t_0) \to \mathcal W$ be such that $v + A(t)v \in \ell(t)$ and $A(t_0)v = 0$ for all $v \in \ell(t_0)$.

Let $\Pi$ be the projection onto $\ell(t_0) \cap \ell_*$ and define $B(t) = \Pi JA(t)\Pi$. If the dimension of $\ker \Pi JA(t_0)\Pi$ is $k$, then Lemma \ref{lemma: deng/jones general ops} tells us there is a $p, q \in \N \cup \{0\}$ such that $p + q = k$ and an $\epsilon, \delta > 0$ such that there are $p$ positive eigenvalues and $q$ negative eigenvalues of $\Pi JA(t)\Pi$ on the interval $(t_0, t_0 + \delta]$. Corollary \ref{corollary: general tan evals} and Theorem \ref{thm: spectral flow RS} allow us to conclude
\begin{align*} \text{Mas}(\ell, \ell_*; \ |t-t_0| \leq \delta) & = \frac{1}{2} \text{sign}\big[\Pi JA(t_0 + \delta) \Pi\big] -\frac{1}{2} \text{sign} \big[\Pi JA(t_0-\delta)\Pi \big]\\
& = \begin{cases}\frac{1}{2}(p - q) - \frac{1}{2}(q-p), & j \text{ odd} \\
\frac{1}{2}(p - q) - \frac{1}{2}(p-q), & j \text{ even} \\
\end{cases} \\
& = \begin{cases} p-q, & j \text{ odd} \\
0, & j \text{ even} \\
\end{cases} \\
& = \begin{cases} \text{sign} Q^{(j)}(v), & j \text{ odd} \\
0 ,& j \text{ even} \\
\end{cases}.
\end{align*}
\end{proof}
\begin{remark} The contribution to the Maslov index at the endpoints of the interval can be chosen with some flexibility as long as the homotopy property in Theorem \ref{lemma: homotopic to regular} is respected. See \cite{dengjones11} for a different convention.
\end{remark}

\begin{remark}
To our knowledge, using higher order derivatives of $\omega(v, A(t)v)$ to understand how the signs of the eigenvalues of the matrix $JA(t)$ change on an interval around some $t_0$ is first suggested in \cite{GPP04}. Our contribution provides proofs of the arguments for this perspective and makes explicit connections between the higher order crossing forms and the definition of the Maslov index developed in \cite{robbinsalamon}. 
\end{remark}

%-------------------------------------------------------------------------------------------------------------------------
\subsection{Motion of Eigenvalues}\label{S:eval-motion}
%-------------------------------------------------------------------------------------------------------------------------
We can relate the $j$th order crossing form to the $j$th derivative of the zero eigenvalue of $JA(t_0)$ with eigenvector $v = L(t_0) u$ with $v \in \ell(t_0) \cap \ell_*$ if $t_0$ is a crossing of order $j$. We first recall a result from \cite{lancaster}. 
\begin{theorem}\cite{lancaster}[Lemma 1]\label{thm: lancaster eigenvalue derivatives} Fix $t_0\in \R$ and suppose that $B(t)$ is a parameter dependent family of symmetric matrices that are diagonalizable and have semi-simple eigenvalues in a neighborhood of $t_0$. Furthermore, let $\lambda(t_0)$ be a simple eigenvalue of $B(t_0)$ with eigenvector $v(t_0)$ normalized so $| v(t_0)| = 1$ and the $i$th derivatives of $B$ vanish at $t_0$ for $i = 1, \dots j-1$, $B^{(i)}(t_0) \equiv 0 $ and $B^{(j)}(t_0)$ is not identically $0$. 

Then, the first $j-1$ derivatives of $\lambda(t_0)$ are zero at $t = t_0$ and 
$$\frac{d^j\lambda}{dt^j}\bigg|_{t = t_0} = \left \langle  v(t_0),  B^{(j)}(t_0) v(t_0)\right \rangle.$$
\end{theorem}

We now relate this to our setting. Denote the eigenvalue associated with a crossing at $t= t_0$ as $\lambda(t)$, so that $\lambda(t_0) = 0$, and denote its associated eigenvector as $v(t)$. Furthermore, let $\Pi$ be the projection onto $\ell(t_0) \cap \ell_*$. If we set $B(t) = \Pi JA(t)\Pi$, and $v = v(t_0)$, Theorem \ref{thm: lancaster eigenvalue derivatives} says 
$$ \lambda^{(j)}(t_0)  = \big\langle v, JA^{(j)}(t_0)  v \big\rangle = Q^{(j)}(v) .$$
Suppose that $\lambda^{(j)}(t_0) \neq 0$ and $j$ is odd. If $\lambda(t)$ is increasing through $0$ at $t = t_0$, then $\text{sign}Q^{(j)} > 0$. If $\lambda(t)$ is decreasing through $0$ at $t = t_0$ and $j$ is odd, we have that $\text{sign}Q^{(j)} < 0$. On the other hand, suppose $j$ is even. If $\lambda(t)$ decreases until $\lambda(t_0) = 0$ and then increases again, then $\text{sign}Q^{(j)} > 0$. If $\lambda(t)$ increases with $t$ until $\lambda(t_0) = 0$ and then decreases and $j$ is even, then $\text{sign}Q^{(j)} < 0$. These situations are summarized in Figure \ref{fig: lambda derivative trajectories}. In this way, Theorem \ref{theorem: maslov gen tang} can be thought of as a higher order derivative test (compare this discussion/Figure \ref{fig: lambda derivative trajectories} with those in the beginning of this section). 

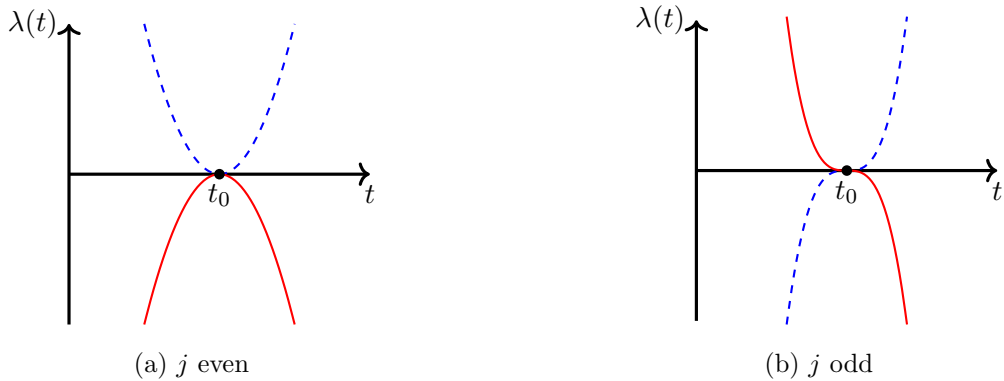
\begin{figure}[H]
    \begin{subfigure}[b]{0.5\textwidth}
        \centering
            \begin{tikzpicture}
                \draw[->, very thick] (-2,0) -- (2,0) node[below] {$t$};
                \draw[->, very thick] (-2,-2) -- (-2,2) node[left]{$\lambda(t)$};
                \draw[scale=0.5, domain=-2:2, smooth, variable=\x, blue, thick, dashed] plot ({\x}, {\x*\x}) node[right]{};
                \draw[scale=0.5, domain=-2:2, smooth, variable=\x, red, thick] plot ({\x}, {-\x*\x}) node[right]{};
                \fill (0,0) circle (2pt) node[below] {$t_0$};
            \end{tikzpicture}
        \caption{$j$ even}
        \label{fig:n even}
    \end{subfigure}
    \begin{subfigure}[b]{0.5\textwidth}
    \centering
            \begin{tikzpicture}
                \draw[->, very thick] (-2,0) -- (2,0) node[below] {$t$};
                \draw[->, very thick] (-2,-2) -- (-2,2) node[left]{$\lambda(t)$};
                \draw[scale=0.5, domain=-1.6:1.6, smooth, variable=\x, blue, thick, dashed] plot ({\x}, {\x*\x*\x}) node[right]{};
                \draw[scale=0.5, domain=-1.6:1.6, smooth, variable=\x, red, thick] plot ({\x}, {-\x*\x*\x}) node[right]{};
                \fill (0,0) circle (2pt) node[below] {$t_0$};
            \end{tikzpicture}
        \caption{$j$ odd}   
        \label{fig:n odd}
    \end{subfigure}
\caption{The graph of $\lambda(t)$ near $t = t_0$, assuming $\lambda^{(i)}(t_0) = 0,\ i = 1, \dots, j-1$ and $\lambda^{(j)}(t_0) \neq 0$. The case that $Q^{(j)}(v) = \lambda^{(j)}(t_0) > 0$ is depicted by the dashed line and the case that $Q^{(j)} = \lambda^{(j)}(t_0)< 0$ is given by the solid line.} 
\label{fig: lambda derivative trajectories}
\end{figure}

\section{Counting unstable eigenvalues via conjugate points}\label{S:count}
We now describe how to use the results of the previous section to prove that the number of unstable eigenvalues for a pulse solution of the Swift-Hohenberg equation is equal to the number of associated conjugate points. Much of this proof will be similar to \cite{BCJ18} with the key difference from that work is that not all of the crossings will be regular. However, now that we are equipped with Theorem \ref{theorem: maslov gen tang}, we will be able to handle this additional difficulty. 

Our strategy will be to consider the eigenvalue problem \eqref{eq: SH eigenvalue problem on R} restricted to the spatial domain $x \in (- \infty, L]$ for some $L > 0$. For $L, \lambda_\infty > 0$ sufficiently large, we will compute the Maslov index relative to $\ell_*^{sand}$ of the path of Lagrangian planes given by the unstable subspace of \eqref{eq: linearized first order}, $\ell(x; \lambda) = \mathbb E^u(x; \lambda)$ for $(x; \lambda)$ on the following infinite intervals: 
\begin{equation}
\begin{split}
D_{1,L} & = \big\{(x, \lambda) \ : \ x = - \infty, \ \lambda \in [0, \lambda_\infty] \big\} \\
D_{2,L} & = \big\{(x, \lambda) \ : \ x \in [- \infty, L], \ \lambda = 0 \big\} \\
D_{3,L} & = \big\{(x, \lambda) \ : \ x = L, \ \lambda \in [0, \lambda_\infty] \big\} \\
D_{4,L} & = \big\{(x, \lambda) \ : \ x \in [-\infty, L], \  \lambda = \lambda_\infty \big\}.
\end{split}
\end{equation}
 The entire path will be denoted as 
$$ D_L = \bigcup_{i = 1}^4 D_{i,L}.$$
See Figure \ref{fig: parameter box} for a depiction of this path. 

\begin{figure}[H]
\centering
\begin{tikzpicture}[scale=1]
      % lines of the box
      \draw[thick, <->] (4, 0) -- (-4, 0) node[below left] {$\lambda$};
      \draw[thick, ->] (-3, -1) -- (-3, 2) node[above] {$x$};
      \draw[thick] (3, -1.5) -- (3, -1);
      \draw[thick] (-3, -1.5) -- (-3, -1);
      \draw[thick, -] (3, -1) -- (3, 1.5);
      \draw[thick] (3, -1.5) -- (-3,-1.5);
      \draw[thick, -] (3, 1.5) -- (-3, 1.5);

      % % arrows 
      % \draw[-{Latex[scale=1.5]}] (3, 1) -- (3, .9) node[right] {};
      % \draw[-{Latex[scale=1.5]}] (-3, -1) -- (-3, -.9) node[right] {};
      % \draw[-{Latex[scale=1.5]}] (1.4, 1.5) -- (1.5, 1.5) node[right] {};
      % \draw[-{Latex[scale=1.5]}] (-1.4, -1.5) -- (-1.5, -1.5) node[right] {};

      % nodes 
      \filldraw[black] (-3,0) circle (1pt) node[below right] {$0$};
      \filldraw[black] (3,0) circle (1pt) node[below left] {$\lambda_\infty$};
      \filldraw[black] (-3,1.5) circle (1pt) node[above right] {$L$};
      \filldraw[black] (-3, -1.5) circle (1pt) node[below] {$-\infty$};

      % labeling path edges
      \draw[black]   (-3, -.5) node[below left]{$D_{2,L}$};
      \draw[black]   (3, -.5) node[below right]{$D_{4,L}$};
      \draw[black]   (0, -1.5) node[above]{$D_{1,L}$};
      \draw[black]   (0, 1.5) node[below]{$D_{3,L}$};

    \end{tikzpicture}
    \caption{The path $D_L$.}\label{fig: parameter box}
\end{figure}
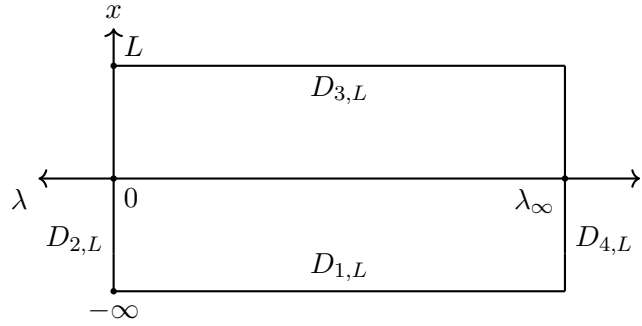

We now heuristically describe our approach to count eigenvalues of \eqref{eq: SH eigenvalue problem on R} on the domain $x \in (-\infty, L]$. First, we argue that the Maslov index of $\ell(x, \lambda)$ on $D_L$ with respect to $\ell_*^{sand}$ can be written as a sum over the $4$ segments of $D_L$ is $0$: 
\begin{equation}\label{eq: heuristic discussion 1}
  \text{Mas}(\ell(x,\lambda), \ell_*^{sand}, D_L) = \sum_{i = 1}^4 \text{Mas}(\ell(x,\lambda), \ell_*^{sand}, D_{i,L}) = 0.
\end{equation}
Next, we find that $\ell(x;\lambda) \cap \ell_*^{sand} = \{0\}$ on $D_{1,L}$ and $D_{4,L}$ so that we can rearrange \eqref{eq: heuristic discussion 1} to obtain
$$ \text{Mas}(\ell(x, \lambda), \ell_*^{sand}, D_{2,L}) = -\text{Mas}(\ell(x, \lambda), \ell_*^{sand}, D_{3,L}).$$
Intersections of $\ell(x, \lambda)$ and $\ell_*^{sand}$ on $D_{3,L}$ correspond to eigenvalues and intersections on $D_{2,L}$ are called conjugate points. Critical to concluding that the number of eigenvalues coincides with the number of conjugate points is the notion of path monotonicity. This means that all of the intersections of $\ell(x,\lambda)$ and $\ell_*^{sand}$ on $D_{2,L}$ contribute to the Maslov index with fixed sign and intersections on $D_{3,L}$ contribute to the Maslov index with opposite fixed sign. This is the third thing we will justify and requires the use of the higher order crossing forms developed in the previous section. 

This information is summarized in Figure \ref{fig: maslov box}. This figure is sometimes referred to as the Maslov box. We will rigorously justify this strategy to prove the following theorem:  
\begin{theorem}\label{thm: morse index half line} Suppose that Hypothesis \ref{hyp:degeneracy} is satisfied. Then the number of positive eigenvalues of $\mathcal L_L$ is equal to the number of conjugate points in $(-\infty, L)$ counted with multiplicities. 
\end{theorem}
 As mentioned in \S \ref{S:intro}, Hypothesis \ref{hyp:degeneracy} states that all crossings are simple. 
In order to evaluate $\ell(x, \lambda) = \mathbb E^u(x; \lambda)$ at $x = -\infty$, first define $\mathbb E^u_-(-\infty, \lambda)$ to be the spectral subspace corresponding to the eigenvalues of the matrix $B_\infty$ in \eqref{eq: B_inf matrix}. We then compactify the spatial domain. Define the functions 
$$s(\sigma) = \frac{1}{2} \ln \left(\frac{1+\sigma}{1-\sigma} \right), \qquad \sigma(s) = \tanh(s), \qquad \sigma\in[-1,1), \ s \in [-\infty, \infty).$$
We can now view the map $\ell(s, \lambda) = \ell(s(\sigma), \lambda)$ as
\begin{equation}\label{eq: compactified phi}
\ell: [-1, 1)\times [0, \lambda_\infty] \to \Lambda(2), \qquad (\sigma,\lambda) \mapsto \ell(s(\sigma),\lambda).
\end{equation}
In most of this discussion, we suppress the $\sigma$ notation and simply refer to $s \in [-\infty, \infty)$. Note that $s(\sigma)$ and $\sigma(s)$ are both continuous on $[-\infty, \infty)$.

\begin{figure}[H]
\centering
\begin{tikzpicture}[scale=1.5]

      % lines of the box
      \draw[thick, <->] (4, 0) -- (-4, 0) node[below left] {$\lambda$};
      \draw[thick, ->] (-3, -1) -- (-3, 2) node[above] {$s$};
      \draw[thick, dash pattern={on 3pt off 4pt}] (3, -1.5) -- (3, -1);
      \draw[thick, dash pattern={on 3pt off 4pt}] (-3, -1.5) -- (-3, -1);
      \draw[thick, -] (3, -1) -- (3, 1.5);
      \draw[thick, dash pattern={on 3pt off 4pt}] (3, -1.5) -- (-3,-1.5);
      \draw[thick, -] (3, 1.5) -- (-3, 1.5);

      % arrows 
      \draw[-{Latex[scale=1.5]}] (3, 1) -- (3, .9) node[right] {};
      \draw[-{Latex[scale=1.5]}] (-3, -1) -- (-3, -.9) node[right] {};
      \draw[-{Latex[scale=1.5]}] (1.4, 1.5) -- (1.5, 1.5) node[right] {};
      \draw[-{Latex[scale=1.5]}] (-1.4, -1.5) -- (-1.5, -1.5) node[right] {};

      % nodes 
      \filldraw[black] (-3,0) circle (1pt) node[below right] {$0$};
      \filldraw[black] (3,0) circle (1pt) node[below left] {$\lambda_\infty$};
      \filldraw[black] (-3,1.5) circle (1pt) node[above right] {$L$};
      \filldraw[black] (-3, -1.5) circle (1pt) node[below] {$-\infty$};

      % labeling path edges
      \draw[black]   (-3, -.5) node[below left]{$D_{2,L}$};
      \draw[black]   (3, -.5) node[below right]{$D_{4,L}$};
      \draw[black]   (0, -1.5) node[above]{$D_{1,L}$};
      \draw[black]   (0, 1.5) node[below]{$D_{3,L}$};

      % eigenvalues
      \filldraw[black] (-1.5,2)  node[anchor=north] {eigenvalues};
      \filldraw[blue] (-1.5,1.5) circle (2pt) node[] {};
      \filldraw[blue] (-1.7,1.5) circle (2pt) node[] {};

      % conjugate points
      \filldraw[black] (-3.5,1)  node[anchor=east] {\shortstack{conjugate\\ points}};
      \filldraw[red] (-3,1) circle (2pt) node[] {};
      \filldraw[red] (-3,.5) circle (2pt) node[] {};

      % no crossings
      \filldraw[black] (-1.5,-1.5)  node[anchor=north] {no crossings};
      \filldraw[black] (3.5,1)  node[anchor=west] {\shortstack{no \\ crossings}};
    \end{tikzpicture}
    \caption{The domain of $\ell(s, \lambda)$ in $(s, \lambda)$ parameter space. We use $s$ instead of $x$ to indicate compactification of the spatial domain.}\label{fig: maslov box}
\end{figure}
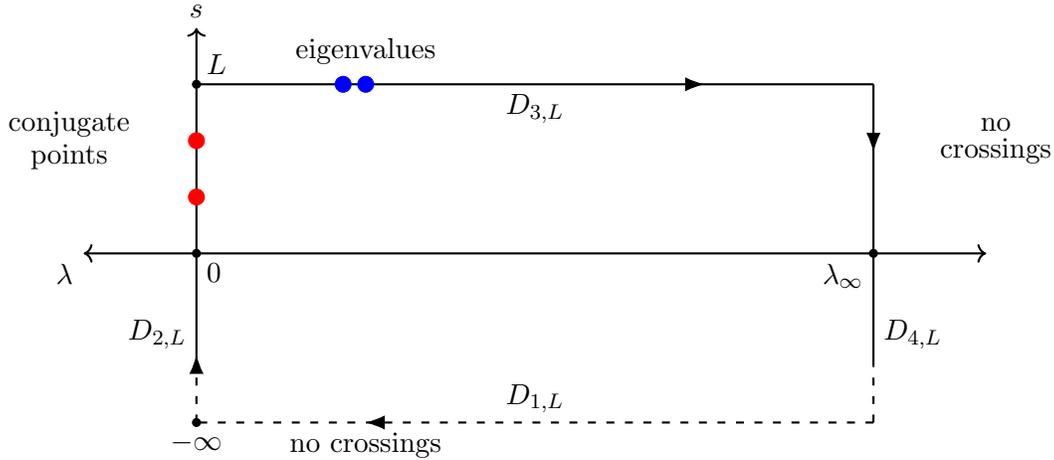

We will be interested in the intersection of $\ell(x, \lambda) = \mathbb E^u_-(x; \lambda)$ with a reference plane.
\begin{definition} For fixed $\lambda \in \R$, a point $s \in (-\infty, L]$ is called a $\lambda$-conjugate point if $\ell(s,\lambda) \cap \ell_* \neq \{0\}$. Since we will pay special attention to $\lambda$-conjugate points with $\lambda = 0$, we simply refer to these as conjugate points. 
\end{definition}

We will employ the same strategy as \cite{BCJ18} and first prove that for fixed $L > 0$, the number of unstable eigenvalues of the restriction of $\mathcal L$ to $(-\infty, L]$ is equal to the number of conjugate points. This proves Theorem \ref{thm: morse index half line}. Second, we will then extend this result to the entire real line.

%--------------------------------------------------------------
\subsection{Restriction to Half Line}\label{section: half line}
%--------------------------------------------------------------
We begin by establishing some necessary properties of $\ell(s,\lambda) = \mathbb E^u_-(s, \lambda)$. The proof of the following result can be found in \cite{BCJ18}[Lemma 2.5].
\begin{lemma} Fix $L > 0$ and $\lambda_\infty > 0$. Then, the map $(\sigma, \lambda) \mapsto \ell(s(\sigma),\lambda)$ is 
\begin{enumerate}[a)]
\item continuous on $[0, \lambda_\infty] \times [-1,\sigma(L)]$;
\item $C^1$ on $[0, \lambda_\infty]\times (-1,\sigma(L)]$. 
\end{enumerate} 
\end{lemma}

If we fix a reference plane $\ell_*$, then the Maslov index of $\ell(s,\lambda)$ with respect to $\ell_*$ on a given parameter domain is well defined. After making a suitable choice for $\ell_*$, the goal of this section is to justify Figure \ref{fig: maslov box}. 

With the extension of $\ell(s, \lambda)$ to $s = -\infty$, one can compute the path $\ell$ on the compactification of $D_L$ using $\sigma(s)$. Since $D_L$ is contractible, the homotopy property in  Lemma \ref{lemma: homotopic to regular} implies $\text{Mas}\left(\ell, \ell_*; D_L\right) = 0$, for any reference plane $\ell_*$. Finally, since $\ell(s, \lambda)$ is continuous on $D_L$, we can conclude that 
$$ \text{Mas}\left(\ell, \ell_*; D_L\right) = \sum_{i = 1}^4 \text{Mas}(\ell, \ell_*; D_{i,L}).$$
Putting these together implies that 
$$\sum_{i = 1}^4 \text{Mas}(\ell, \ell_*; D_{i,L}) = 0.$$
Therefore, we can compute the Maslov index of $\ell(s, \lambda)$ with respect to $\ell_*^{sand}$ on each segment $D_{i,L}$ individually and their sum must be $0$. 

We now restrict the operator $\mathcal L$ to the half line $(-\infty, L]$ for fixed $L > 0$ and choose our reference plane. 

\begin{definition}\label{def: sandwich plane} Define the Lagrangian plane 
$$ \ell_*^{sand} = \{(q_1, \dots, q_4) \ : \ q_1 = q_4 = 0, \ q_2, q_3 \in \R \} .$$
We refer to this plane as the sandwich plane.
\end{definition} 
This subspace has an orthonormal frame given by
$$ T = \begin{pmatrix} 0 & 0 \\ 1 & 0 \\ 0 & 1 \\ 0 & 0 
\end{pmatrix} .$$

In order to see why this frame is a natural choice, first denote the operator restricted to the half line as 
\begin{equation}\label{eq: SL}
\mathcal L_L u := -\partial_x^4 u - 2 \partial_x^2 u - u + f \circ \varphi(x) u, \qquad  u : (-\infty, L] \to \R,
\end{equation}
and consider the eigenvalue problem 
\begin{equation} 
\begin{split}
\mathcal L_L u  & = \lambda u, \qquad  u : (-\infty, L] \to \R \\
u(L) & = u'(L) = 0.
\end{split}
\end{equation}
The domain of this operator is given by 
$$\text{dom}\mathcal L_L = \left\{ u \in H^4\big((-\infty, L], \R \big) \ : \ u(L) = 0, \ u'(L) = 0\right\} .$$

If $u(x)$ is a solution to \eqref{eq: SL}, then using the symplectic change of coordinates given in \eqref{eq: symplectic change of coords sh}, we can compute 
$$ \begin{pmatrix} 0\\ 0 \end{pmatrix}  = \begin{pmatrix}  -u'(L) \\ u(L)
\end{pmatrix} = T^* J \begin{pmatrix} u(L) \\ u''(L) \\ u'''(L)+2u'(L) \\ u'(L)
\end{pmatrix} =0.$$
In other words, $\ell_*^{sand} \cap \mathbb E^u_-(L; \lambda) \neq \{0\}$ if and only if $\lambda$ is an eigenvalue. Since we seek to detect eigenvalues with the intersections of these two subspaces, $\ell_*^{sand}$ is a natural choice of reference plane. 

We will first prove our results on the half line and then extend them to the full line. In doing so, it will be helpful to think of the right endpoint of the spatial domain as a parameter. Within this framework, the operator and eigenvalue problem we are interested in are given by 
\begin{equation}
\ADD{\mathcal L_s = -\partial_x^4 - 2 \partial_x^2 - 1 + f' \circ \varphi(x),}
\end{equation}
\begin{equation}\label{eq: half line eval prob BC}
\begin{split}\mathcal L_s u &  =  \lambda u,    \\
u(s) & = u'(s) = 0.
\end{split}
 \end{equation}

The below lemma illustrates a duality between conjugate points and eigenvalues. 
\begin{lemma}\label{lemma: conj pt eval corr} Suppose $\lambda \in \R$ and $s \in (- \infty, L]$. Then $\lambda$ is an eigenvalue of the $4$th order ODE given in Equation \eqref{eq: half line eval prob BC} if and only if $s$ is a $\lambda$-conjugate point of the first order system given by Equation \eqref{eq: linearized first order}. Furthermore, the multiplicity of the eigenvalue coincides with the dimension of the subspace $\ell_*^{sand} \cap \ell(s,\lambda)$.
\end{lemma}
The proof is virtually identical to that in \cite{BCJ18}[Proposition 2.8].
%--------------------------------------------------------------------------------------------------
\subsubsection{No Crossings on the Bottom and Left Side}
%--------------------------------------------------------------------------------------------------

Recall that $\text{Mas}(\ell, \ell_*^{sand}; D_{4,L})$ and $\text{Mas}(\ell, \ell_*^{sand}; D_{1,L})$ denote the Maslov index of $\ell(s, \lambda)$ evaluated on the right and bottom sides of the Maslov box respectively (Figure \ref{fig: maslov box}). We will argue that $\text{Mas}(\ell, \ell_*^{sand}; D_{1,L}) = 0$ and $\text{Mas}(\ell, \ell_*^{sand}; D_{4,L}) = 0$ provided $\lambda_\infty$ is sufficiently large. 

First, we argue that the eigenvalues for the operator $\mathcal L_s$ defined in Equation \eqref{eq: half line eval prob BC} for $s\in (-\infty, L]$ are bounded from above. 

\begin{lemma}\label{lemma: mathcal L_s evals are bdd below}
Fix $s \in (-\infty, L]$ and consider the differential operator given in \eqref{eq: half line eval prob BC}.
Then, any eigenvalue $\lambda$ of $\mathcal L_s$ satisfies 
$$ \lambda \leq \|f\circ \varphi(x)\|_{L^\infty(-\infty, s)}.$$
\end{lemma}
\begin{proof}

Suppose that $\psi$ is an eigenfunction of $\mathcal L_s$ with eigenvalue $\lambda$ on the domain $(-\infty, s]$. By multiplying the both sides of Equation \eqref{eq: half line eval prob BC} by $\psi$ and integrating in $x$ from $-\infty$ to $s$ for any $s$ chosen to be in $(-\infty, L]$, we obtain: 
\begin{align*}  \lambda \|\psi\|^2_{L^2\big((-\infty, s), \R\big)}& =-\int_{-\infty}^s \left(\partial_x^4 \psi\right) \psi \ dx - 2\int_{-\infty}^s \left(\partial_x^2\psi\right) \psi 
 dx  - \int_{-\infty}^s \psi^2 \ dx + \int_{-\infty}^s f'\circ\varphi(x)\psi^2 \ dx .
\intertext{By performing integration by parts twice on the first integral and using $\psi(s) = \psi'(s) = 0$, we can write}
%& = -\psi(x)\partial_x^3\psi(x) + \partial_x^2 \psi(x)\partial_x \psi(x) \bigg|_{-\infty}^s - \int_{\infty}^s \left( \partial_x^2 \psi \right)^2 \ dx \\
%& \qquad \qquad -  \int_{-\infty}^s 2\left(\partial_x^2\psi\right) \psi + \psi^2 \ dx + \int_{-\infty}^s f'\circ\varphi(x)\psi^2 \ dx \\
% & = -\psi(s) \psi'''(s) + \psi''(s)\psi'(s) + \lim_{x \to - \infty} \left(\psi(s) \psi'''(s) + \psi''(s)\psi'(s) \right) \\
% & \qquad \qquad - \int_{\infty}^s \left( \partial_x^2 \psi \right)^2 \ dx  - \int_{-\infty}^s 2\left(\partial_x^2\psi\right) \psi   + \psi^2 \ dx + \int_{-\infty}^s f'\circ\varphi(x)\psi^2 \ dx \\
% \intertext{Since $\psi$ is an eigenvector on $(-\infty, s]$, we know that the boundary terms as $x \to - \infty$ must vanish. Furthermore, via our boundary conditions, we know that $\psi(s) = \psi'(s) = 0$. Thus the boundary terms from integration by parts vanish and we have,}
 \lambda \|\psi\|^2_{L^2\big((-\infty, s), \R\big)} & =  -\int_{\infty}^s \left( \partial_x^2 \psi \right)^2 \ dx - 2\int_{-\infty}^s \left(\partial_x^2\psi\right) \psi \ dx - \int_{-\infty}^s \psi^2 \ dx+ \int_{-\infty}^s f'\circ\varphi(x)\psi^2 \ dx \\
& =  -\int_{-\infty}^s \big(\partial_x^2\psi (x) + \psi(x) \big)^2 \ dx + \int_{-\infty}^s f'\circ\varphi(x)\psi^2 \ dx \\ 
& = -\left\|\partial_x^2\psi + \psi\right\|^2_{L^2((-\infty, s), \R)} + \int_{-\infty}^s f'\circ\varphi(x)\psi^2 \ dx \\
& \leq  \|f' \circ \varphi\|_{L^\infty (-\infty, s)} \|\psi\|^2_{L^2((-\infty, s), \R)}.
\end{align*}
Thus, 
$$\lambda \|\psi\|^2_{L^2\big((-\infty, s), \R\big)} \leq  \|f \circ \varphi\|_{L^\infty (-\infty, s)} \|\psi\|^2_{L^2((-\infty, s), \R)}.$$
Since $\psi$ is an eigenfunction, it is not identically $0$ and has nonzero norm. Cancelling the $\|\psi\|^2$ term from both sides of the above expression, we conclude 
$$\lambda  \leq  \|f \circ \varphi\|_{L^\infty (-\infty, s)}.$$
Thus, if we choose $\lambda_\infty$ such that $ \lambda_\infty > \|f \circ \varphi\|_{L^\infty (-\infty, s)}$, we see that $\mathcal L_s$ has no eigenvalues with $\lambda \geq \lambda_\infty.$
\end{proof}

As we will see in the proof of Proposition \ref{prop: zero left}, the fact that there are zero crossings on the right side of the box is immediate from the previous result.

\begin{proposition}\label{prop: zero left}
If $\lambda_\infty > \|f \circ \varphi(x)\|_{L^\infty}$, then 
$$\text{Mas} (\ell,\ell_*^{sand};D_{4,L})  =0.$$
\end{proposition}
\begin{proof} Suppose  $s$ is a $\lambda$ conjugate point. Then, by definition, there exists an eigenfunction $\psi$ such that
\begin{equation}\label{eq: zero left S} 
\begin{split}
\mathcal L_s\psi& :=-\partial^4_x\psi - 2 \partial^2_x \psi - \psi + f'\circ\varphi(x)\psi = \lambda \psi 
\end{split}
\end{equation}
where $\psi(s) = \psi'(s) = 0$. By Lemma \ref{lemma: mathcal L_s evals are bdd below}, we know that 
$ \lambda \leq \|f\circ \varphi(x)\|_{L^\infty}.$ Thus, if we choose $\lambda_\infty > \|f\circ \varphi(x)\|_{L^\infty}$, there will be no conjugate points on $D_{4,L}$ and so 
$ \text{Mas} (\ell,\ell_*^{sand}; D_{4,L})  =0,$
as desired.
\end{proof}

Next we show that there are no intersections between $\ell(s,\lambda)$ and $\ell_*^{sand}$ for parameters lying on the bottom of the Maslov box. This result will rely on an explicit characterization of the eigenbasis vectors of the matrix $B_\infty$ given in the below lemma. 

\begin{lemma}\label{lemma: evecs of B}
Let $B_\infty$ be as defined in \eqref{eq: B_inf matrix}. If 
$$ \theta = \arctan \left(-\sqrt{\lambda - f'(0)}\right) \quad \text{ and } r = \sqrt{1 + \lambda - f'(0)},$$
then basis vectors for the real unstable and stable eigenspaces of $B_\infty(\lambda)$ are given via 
\begin{equation}\label{eq: B_inf unstable evecs}
R^u_1 = \begin{pmatrix} \frac{1}{r}\cos \theta \\ 1 \\ \left(\frac{2}{\sqrt r} + \sqrt{r} \right) \cos \left(\frac{\theta}{2} \right)  \\ \frac{1}{\sqrt r} \cos \left(\frac{\theta}{2} \right)
\end{pmatrix}, \quad  R^u_2 = \begin{pmatrix} -\frac{1}{r}\sin \theta \\ 0 \\ \left(\sqrt{r} - \frac{2}{\sqrt r} \right) \sin \left(\frac{\theta}{2} \right)  \\ -\frac{1}{\sqrt r} \sin \left(\frac{\theta}{2} \right)
\end{pmatrix},
\end{equation} 
\begin{equation}\label{eq: B_inf stable evecs}
R_1^s = \begin{pmatrix} \frac{1}{r}\cos \theta \\ 1 \\ \left(-\frac{2}{\sqrt{r}}  - \sqrt r \right) \cos \left( \frac{\theta}{2}\right) \\
-\frac{1}{\sqrt r}\cos \left(\frac{\theta}{2}\right)
\end{pmatrix}, \quad  R_2^s = \begin{pmatrix} -\frac{1}{r}\sin \theta \\ 0 \\ \left(\frac{2}{\sqrt{r}} - \sqrt r \right) \sin \left( \frac{\theta}{2}\right) \\
\frac{1}{\sqrt r}\sin \left(\frac{\theta}{2}\right)
\end{pmatrix}. 
\end{equation}
\end{lemma}
\begin{proof}
See Section \ref{Subsec: proof of B evecs lemma}.
\end{proof}

With this technical result, we can prove that these are no intersections with $\ell(-\infty, \lambda)$ and $\ell_*^{sand}$. 
\begin{proposition}\label{prop: zero bottom}
Recall $\ell(-\infty, \lambda) = \mathbb E^u_-(- \infty, \lambda)$ is defined as the unstable spectral subspace of the matrix $B_\infty$. For $\lambda_\infty > 0$, 
$$\text{Mas} (\ell,\ell_*^{sand}; D_{1,L}) = \text{Mas}\big(\ell(-\infty,\lambda), \ell_*; \ 0 \leq \lambda \leq \lambda_\infty\big) =0.$$
\end{proposition}
\begin{proof} 
By way of contradiction, suppose 
$$\bigg[ \mathbb E^u_-(-\infty, \lambda) = \text{span} \{R_k^u\}_{k=1}^2 \bigg] \cap \ell_*^{sand} \neq \{0\};$$
where $R_k^u$ is given in \eqref{eq: B_inf unstable evecs}. Then there exists $c_1, c_2$ with at least one $c_i \neq 0$ such that 
\begin{align*} c_1 R_1^u + c_2 R_2^u & = c_1 \begin{pmatrix} \frac{1}{r}\cos \theta \\ 1 \\ \left(\frac{2}{\sqrt r} + \sqrt{r} \right) \cos \left(\frac{\theta}{2} \right)  \\ \frac{1}{\sqrt r} \cos \left(\frac{\theta}{2} \right)
\end{pmatrix} + c_2 \begin{pmatrix} -\frac{1}{r}\sin \theta \\ 0 \\ \left(\sqrt{r}-\frac{2}{\sqrt r} \right) \sin \left(\frac{\theta}{2} \right)  \\ -\frac{1}{\sqrt r} \sin \left(\frac{\theta}{2} \right)
\end{pmatrix}\\
& \in  \text{span} \left\{\begin{pmatrix} 0 \\ 1 \\ 0 \\ 0 
\end{pmatrix}, \begin{pmatrix} 0 \\ 0 \\ 1 \\ 0 
\end{pmatrix} \right\}.\\
\intertext{In particular, this means that }
 0 & = c_1 \cos \theta - c_2 \sin\theta \\
0 & = c_1\cos\left(\frac{\theta}{2}\right) - c_2 \sin\left(\frac{\theta}{2}\right) .
\end{align*}
Solving for $c_1$ in both equations, we obtain $$c_1 = \frac{c_2\sin\theta}{\cos\theta} = \frac{c_2\sin\left(\frac{\theta}{2}\right)}{\cos\left(\frac{\theta}{2}\right)}.$$
Note that since $\theta \in \left(\frac{\pi}{2}, \pi \right)$, the denominators in the above expression are nonzero so the quantities are defined. 

However, this implies that $\cos\theta\sin \left(\frac{\theta}{2} \right) = \sin\theta\cos \left(\frac{\theta}{2} \right) $ or $c_1 = c_2 = 0$. The first statement holds only at $\theta = 2\pi k, \ k \in \Z$ and at least one of $c_1,c_2$ must be nonzero; a contradiction. Thus,
$$ \ell(-\infty;\lambda) \cap \ell_* = \emptyset$$
for $\lambda \in [-\lambda_\infty , 0]$, so $\text{Mas}(\ell, \ell_*^{sand}; D_{1,L}) = 0$.
\end{proof}
\begin{remark}
This proof differs slightly from the analogous one in \cite{BCJ18}. In previous work addressing a system of scalar reaction diffusion equations, the eigenvectors of the matrix corresponding to $PQ$ could be assumed to be orthonormal, so their explicit characterization was not necessary. The structure of the coefficient matrix for the Swift-Hohenberg equation does not satisfy this assumption, which necessitates the detailed and explicit characterization of the eigenvectors. 
\end{remark}

%----------------------------------------------------------------------------------------------
\subsubsection{Monotonicity in the spectral parameter}
%----------------------------------------------------------------------------------------------
In order to justify Figure \ref{fig: maslov box}, we must argue that the crossings of $\ell$ and $\ell_*^{sand}$ in the parameter region corresponding to the top of the box all have the same sign. This is the monotonicity property of the Maslov index of $\ell(s, \lambda)$ on $D_{3,s}$ with respect to $\ell_*^{sand}$.

Since we previously argued that the operator $\mathcal L_s$ with $s \in (-\infty, L]$ has no eigenvalues greater than $\lambda_\infty$ (Lemma \ref{lemma: mathcal L_s evals are bdd below}), we immediately see for fixed $s$ 
$$ \text{Mas}\left(\ell(s,\lambda), \ell_*^{sand}; 0 \leq \lambda \leq \infty \right) = \text{Mas}\left(\ell(s,\lambda), \ell_*^{sand} ; 0 \leq \lambda \leq \lambda_\infty\right).$$
Therefore, we can increase $\lambda$ from $0$ to $\lambda_\infty$ and capture all of the non-negative eigenvalues.

\begin{proposition}\label{prop: lam mono} For any fixed $s \in (-\infty,L]$, the Langrangian path defined by $\ell(s,\lambda) = \mathbb E^u_-(s,\lambda)$ contributes negatively to the signed count of the Maslov index as we increase $\lambda$. In particular, 
$$ \text{Mas}(\ell, \ell_*^{sand}; D_{3,s}) = \text{Mas}\left(\ell(s,\lambda), \ell_*^{sand}; 0 \leq \lambda \leq \lambda_\infty\right) = -\sum_{0 \leq \lambda \leq \lambda_\infty} \dim \big(\ell(s,\lambda) \cap \ell_*^{sand} \big) \leq 0. $$
\end{proposition}
\begin{proof} Suppose that $\lambda_* \in [0,\lambda_\infty]$ is a crossing, meaning $\ell(s,\lambda_*) \cap \ell_*^{sand} \neq \{0\}$ and let $\mathcal W$ be a Lagrangian subspace such that $\ell(s, \lambda_*) \cap \mathcal W = \{0\}$. Then $\mathcal W$ is transverse to $\ell(s, \lambda)$ for $\lambda \in [\lambda_* - \epsilon, \lambda_* + \epsilon]$ for $\epsilon$ sufficiently small so we can express $\ell(\lambda, s)$ as the graph of an operator $A(\lambda): \ell(s, \lambda_*) \to \mathcal W$. In particular, if we fix $v \in \ell(s; \lambda_*) \cap \ell_*^{sand}$ then $v + A(\lambda)v \in \ell(\lambda, s)$ and $A(\lambda_*)v = 0$. Since $\ell(s, \lambda) = \mathbb E^u_-(s, \lambda)$, there are two families of solutions, $z(s,\lambda)$ and $y(s, \lambda)$ to \eqref{eq: linearized first order} such that $z(s;\lambda_*) = v$ and the pair of solutions $z(s; \lambda)$ and $y(s;\lambda)$ form a basis for $\ell(s,\lambda)$ with $\lambda$ close to $\lambda_*$. Furthermore, there exists functions $c_1,c_2: \R \to \R$ such that \eqref{eq: graph A satisfies equation} applies and for fixed $s$,
\begin{equation}\label{eq: graph spectral param}
z(s;\lambda_*) + A(\lambda)z(s; \lambda_*) = c_1(\lambda) z(s;\lambda) + c_2(\lambda) y(s;\lambda).
\end{equation}
Without loss of generality, we can assume that $c_1(\lambda_*) = 1$ and $c_2(\lambda_*) = 0$.

We show that the crossing form is negative definite, which implies that the contribution to the Maslov index at each crossing is negative. Using \eqref{eq: graph spectral param}, we can compute: 
\begin{align*} Q^{(1)}(v) & = \frac{d}{d\lambda} \omega(v, A(\lambda)v) \bigg|_{\lambda = \lambda_*} \\
%& = \frac{d}{d\lambda} \left(\omega(v, w(\lambda)) + \omega(v, v) \right) \bigg|_{\lambda = \lambda_*} \\
& = \frac{d}{d\lambda} \omega(v, v + A(\lambda)v) \bigg|_{\lambda= \lambda_*} \\
& = \left\langle z(s; \lambda_*), J \frac{d}{d\lambda} \big[ c_1(\lambda) z(s; \lambda) + c_2(\lambda)y(s;\lambda) \big] \right\rangle \bigg|_{\lambda = \lambda_*} \\
%& = \left\langle z_1(s; \lambda_*), J \big[ c_1'(\lambda_*) z_1(s; \lambda_*) + c_1(\lambda_*) z_1'(s; \lambda_*) + c_2'(\lambda_*)z_2(s;\lambda_*) + c_2(\lambda_*)z_2'(s;\lambda_*) \big] \right\rangle \\
& = (c_1)_\lambda (\lambda_*)\underbrace{\left\langle z(s; \lambda_*), J  z(s; \lambda_*) \right\rangle}_{=0 \ \text{(Lagrangian property)}} + \underbrace{c_1(\lambda_*)}_{=1}\left\langle z(s; \lambda_*), J z_\lambda(s; \lambda_*) \right\rangle \\
& \qquad \qquad + (c_2)_\lambda (\lambda_*)\underbrace{\left\langle z(s; \lambda_*), J y(s;\lambda_*) \right\rangle}_{=0 \ \text{(Langrangian property)}} + \underbrace{c_2(\lambda_*)}_{=0}\left\langle z(s; \lambda_*), Jy_\lambda(s;\lambda_*) \big] \right\rangle \\
%& = (c_1)_\lambda(\lambda_*)\underbrace{\left\langle z(s; \lambda_*), J  z(s; \lambda_*) \right\rangle}_{=0 \ \text{(Lagrangian property)}} + \underbrace{c_1(\lambda_*)}_{=1}\left\langle z(s; \lambda_*), J z_\lambda(s; \lambda_*) \right\rangle \\
%& \qquad \qquad +(c_2)_\lambda(\lambda_*)\underbrace{\left\langle z(s; \lambda_*), J y(s;\lambda_*) \right\rangle}_{=0 \ \text{(Langrangian property)}} + \underbrace{c_2(\lambda_*)}_{=0}\left\langle z(s; \lambda_*), Jy_\lambda(s;\lambda_*) \big] \right\rangle \\
& = \left\langle z(s; \lambda_*), J z_\lambda(s; \lambda_*) \right\rangle.
\end{align*}

First notice that if we have a solution $z(x,\lambda)$ to Equation \eqref{eq: linearized first order}, via the chain rule, the $\lambda$ derivative evaluated at $x = s$ is
\begin{equation}\label{eq: B lam}
z_{x\lambda} =  B(s,\lambda)z_\lambda +   B_\lambda(s,\lambda) z \ \text{ where } \  B_\lambda = \begin{pmatrix} 0 & 0 & 0 & 0 \\ 0 & 0 & 0 & 0 \\ -1 & 0 & 0 & 0 \\ 0 & 0 & 0 & 0 
\end{pmatrix}. 
\end{equation} 
We will use the following observation as well: 
\begin{equation}
J B_\lambda = \begin{pmatrix} 0 & I_2 \\ -I_2 & 0
\end{pmatrix} \begin{pmatrix} 0 & 0 & 0 & 0 \\ 0 & 0 & 0 & 0 \\ -1 & 0 & 0 & 0 \\ 0 & 0 & 0 & 0 
\end{pmatrix} =  \begin{pmatrix} -1 & 0 & 0 & 0 \\ 0 & 0 & 0 & 0 \\ 0 & 0 & 0 & 0 \\ 0 & 0 & 0 & 0 
\end{pmatrix}.
\end{equation}

Now we can compute: 
\begin{align*} \left\langle z(s; \lambda_*)\right. & \left., J \frac{d}{d\lambda} z(s; \lambda) \right\rangle  \bigg|_{\lambda = \lambda_*} \\
& = \big\langle z(s,\lambda_*), J z_\lambda (s,\lambda_*) \big\rangle \\
& = - \big \langle J z(s,\lambda_*), z_\lambda(s,\lambda_*) \big\rangle, \qquad (J^T = -J) \\
& = -\int_{-\infty}^s \frac{d}{dx} \big\langle Jz(x,\lambda_*), z_\lambda(x, \lambda_*) \big \rangle  \ dx\\
& = - \int_{-\infty}^s \bigg \langle J z_x(x,\lambda_*), z_\lambda(x,\lambda_*) \bigg\rangle  \ dx - \int_{-\infty}^s \big \langle Jz(x,\lambda_*), z_{x\lambda} (x,\lambda_*)  \big \rangle \ dx \\
& = -\int_{-\infty}^s \left \langle JB(x,\lambda_*) z(x,\lambda_*), z_\lambda(x,\lambda_*) \right\rangle \ dx\\
& \qquad  - \int_{-\infty}^s \left \langle Jz(x,\lambda_*), B(x,\lambda_*) z_\lambda (x,\lambda_*)  + B_\lambda z(x,\lambda_*) \right \rangle \ dx \\
& = -\int_{-\infty}^s \left \langle JB(x,\lambda_*) z(x,\lambda_*), z_\lambda(x,\lambda_*)\right\rangle \ dx\\
& \qquad - \int_{-\infty}^s \left\langle B^T(x,\lambda_*)Jz(x,\lambda_*), z_\lambda (x,\lambda_*) \right\rangle \ dx - \int_{-\infty}^s \big \langle Jz(x,\lambda_*), B_\lambda z(x,\lambda_*) \big \rangle \ dx.
\end{align*}
Collecting the first two integrals and recalling that $JB=-C$, where $C$ is a symmetric matrix given by Equation \eqref{eq: symm C}, we see:
\begin{align*}
\left\langle z(s; \lambda_*), J \frac{d}{d\lambda} z(s; \lambda) \right\rangle \bigg|_{\lambda = \lambda_*} & = \int_{-\infty}^s \bigg \langle \underbrace{\big[(JB(x,\lambda_*))^T - JB(x,\lambda_*) \big]}_{=0}z(x,\lambda_*), z_\lambda(s,\lambda) \bigg|_{\lambda = \lambda_*} \bigg \rangle \ dx \\
& \qquad - \int_{-\infty}^s \big \langle Jz(x,\lambda_*), B_\lambda z(x,\lambda_*) \big \rangle \ dx \\
& = \int_{-\infty}^s \big \langle z(x,\lambda_*), J B_\lambda z(s,\lambda) \big \rangle \ dx, \qquad (J^T=-J)\\
& = -\int_{-\infty}^s \big \langle z(x,\lambda_*), [z_1(x,\lambda_*), 0,0,0] \big \rangle \ dx, \qquad \eqref{eq: B lam}\\
& = -\int_{-\infty}^s z_1(x,\lambda_*)^2 \ dx.
\end{align*}
In order to see that this quantity cannot be $0$, we proceed by contradiction. Suppose that $z_1(x,\lambda_*) = 0$ for all $x \in (-\infty, s)$. Substituting this solution back into Equation \eqref{eq: linearized first order} yields: 
\begin{align*}\dot z_1&= z_4  = 0\\
\dot z_2 & =  z_3 - 2z_4 \\
\dot z_3 &=0 \\
\dot z_4 & =  z_2.
\end{align*}

From these, we see that if $z_1 \equiv 0$ for $x \in (-\infty, s)$, then all other components of the solution are $0$ as well, meaning $z(x,s)$ is the trivial solution. Finally, via Figure \ref{fig: maslov box}, we are increasing the spectral parameter, $\lambda$, as we traverse the top of the box. Therefore, the crossing form on this portion of the path of $\ell(s,\lambda)$ through $\Lambda(2)$ is negative definite.
\end{proof}

\begin{remark}\label{rmk: graphing error}
In \cite{BCJ18}, it is assumed that the solution $z(s; \lambda_*)$ is left invariant with respect to the operator $A(\lambda)$. In other words, $z(s; \lambda_*) + A(\lambda)z(s;\lambda_*) = z(s; \lambda)$ so that 
$$ \frac{d}{d\lambda} A(\lambda)z(s;\lambda_*) \bigg|_{\lambda = \lambda_*} = \frac{d}{d\lambda}z(s; \lambda)\bigg|_{\lambda = \lambda_*}.$$
This is not necessarily true; see Section \ref{subsec: explicit example} for a counterexample. However, the calculation performed this way yields the same result as the one with the extra terms that are included in \eqref{eq: graph spectral param}. As we will see in the next section, these extra terms become particularly important with higher order derivatives.  
\end{remark}

%----------------------------------------------------------------------------------------------
\subsubsection{Monotonicity in the spatial parameter}\label{subsec: monotonicity in spatial param}

Finally, we will argue that the crossing form on the left side of the box contributes positively to the Maslov index (see Figure \ref{fig: maslov box}). In order to do so, we will work with a family of solutions that decay to $0$ in backwards time and solve the linearized ODE \eqref{eq: linearized first order} on the domain $(-\infty, s]$ for fixed $\lambda$. We denote this family of solutions as $p(x,\lambda; s)$, where this notation denotes that $p(x,\lambda; s)$ is a function of $x$ and $\lambda$ but depends on $s$ as a parameter. For $x < s_i$ with $i = 1,2$, $p(x, \lambda; s_1) = p(x, \lambda; s_2)$. In other words, changing the endpoint of the domain for the solution $p$ does not effect the function value for points in both the original and new domain. 

Fix $\lambda$ and suppose that $x = s_*$ is a crossing, so that $\ell_*^{sand} \cap \mathbb E^u_-(s_*, \lambda)\neq \{0\}$. We will differentiate $p(x,\lambda; s)$ in $x$. In order to simplify this computation, we set $s = s_* + \epsilon$ for some $\epsilon > 0$ and write $p(x, \lambda; s)$ as $p(x,\lambda)$. 

We proceed in an analogous way to the previous claim. More specifically, define $\mathcal W$ to be a subspace in $\R^{2n}$ transverse to $\ell(s_*,\lambda)$. Thus, $\mathcal W$ is transverse to $\ell(s,\lambda)$ for $s \in [s_* - \epsilon, s_* + \epsilon]$ for sufficiently small $\epsilon$. Fix $v \in \ell(s_*, \lambda) \cap \ell_*^{sand}$ and define the matrix $A(s): \ell(s_*, \lambda) \to \mathcal W$ such that $v + A(s)v \in \ell(s, \lambda)$ and $A(s_*)v = 0$. Since $\ell(s, \lambda) = \mathbb E^u_-(s; \lambda)$, there exists a family of solutions $p(x;\lambda)$ on the interval $(-\infty, s_* + \epsilon]$ such that $p(s_*;\lambda) = v$. Additionally, there is another family of solutions $q(s;\lambda)$ such that $p(s;\lambda)$ and $q(s;\lambda)$ form a basis for $\ell(s, \lambda)$ on the domain $s \leq s_* + \epsilon$. Thus, there exists $d_1(s)$ and $d_2(s)$, defined for $s$ close to $s_*$ such that 
\begin{equation}\label{eq: spatial monotonicity graph formulation} 
\begin{split}
p(s_*; \lambda) + A(s)p(s_*;\lambda)  = d_1(s)p(s;\lambda) + d_2(s) q(s;\lambda), \qquad
d_1(s_*) & = 1 \\  d_2(s_*) & = 0.
\end{split}
\end{equation}

We perform a computation similar to the one done in Proposition \ref{prop: lam mono}, but we find that the first order crossing form could be zero. This will lead us to cases in which we will need to consider higher order crossing forms. We begin with a lemma regarding the first order crossing form.  

\begin{lemma} \label{lem:first-ordercrossing-form} If $v \in \ell(s_*, \lambda) \cap \ell_*^{sand}$ and $p(x, \lambda; s)$ is a solution of \eqref{eq: linearized first order} such that $v = p(s_*, \lambda; s) = (p_1(s_*, \lambda), p_2(s_*, \lambda), p_3(s_*, \lambda), p_4(s_*, \lambda))$, then 
\begin{equation}\label{E:first-order-crossing}
Q^{(1)}(v) = (p_2(s_*; \lambda))^2.
\end{equation}
\end{lemma}
\begin{proof} First we note that, since $p(s_*, \lambda) \in \ell_*^{sand}$, we have $p_1(s_*, \lambda) = p_4(s_*, \lambda) = 0$.  Using the definition of $Q^{(1)}(v)$, the Lagrangian structure, and the representation \eqref{eq: spatial monotonicity graph formulation}, we find
\begin{align*}
Q^{(1)}(v) &= \frac{d}{ds}\omega(v, A(s)v)|_{s = s_*} \\
&= \frac{d}{ds}\omega(v, v + A(s)v) \\
&= \frac{d}{ds}\langle v, J(d_1(s)p(s;\lambda) + d_2(s) q(s;\lambda))\rangle|_{s = s_*} \\
&= \langle v, J(d_1'(s)p(s;\lambda) + d_1(s)p_s(s;\lambda) + d_2'(s) q(s;\lambda)+ d_2(s) q_s(s;\lambda))\rangle|_{s = s_*} \\
&= \langle p(s_*, \lambda), Jp_s(s_*, \lambda; )\rangle \\
&= \langle p(s_*, \lambda), JBp(s_*, \lambda; )\rangle \\
&= (p_2(s_*))^2.
\end{align*}
\end{proof}
\begin{remark} Recall that the codomain of $A(s)$ is given by $\mathcal{W}$, which was discussed just above \eqref{eq: spatial monotonicity graph formulation}. The proof of the above Lemma did not involve the choice of $\mathcal{W}$ in any way. 
\end{remark}
As a result of this Lemma, whether or not the first order crossing form is degenerate will depend on whether or not $p_2(s_*)$ is nonzero. Thus, we investigate further the properties of the vectors that form a basis for $\ell(s_*, \lambda)$.
\begin{lemma}\label{lem:cases} Suppose $s_*$ is a crossing and $p(s_*; \lambda) \in \ell(s_*, \lambda) \cap \ell_*^{sand}$.  The basis solutions $p(s_*; \lambda)$ and $q(s_*; \lambda)$ for $\ell(s_*, \lambda)$ have the form
\begin{align*}
p(s_*;\lambda) & = \begin{pmatrix} 0 & p_2(s_*; \lambda)  & p_3(s_*; \lambda) & 0
\end{pmatrix} \\
q(s_*; \lambda) & = \begin{pmatrix} \alpha p_2(s_*; \lambda)   & q_2(s_*; \lambda) & q_3(s_*; \lambda) & \alpha p_3(s_*; \lambda) 
\end{pmatrix},
\end{align*}
for some $\alpha \in \R$.
\end{lemma}
\begin{proof} We suppress the $\lambda$ dependence in order to simplify notation. That is we write  $p(s_*) = p(s_*; \lambda)$ and $q(s_*) = q(s_*; \lambda)$ and note that neither can be the zero vector because then the corresponding solution would be the zero solution. Since $p(s_*) \in \ell_* \cap \ell(s_*)$,   it must have the form  $p(s_*) = \begin{pmatrix} 0 & p_2(s_*) & p_3(s_*) & 0 \end{pmatrix}$. To describe $q(s_*)$, observe that via the Lagrangian property we can compute
\[
0  = p(s_*)^TJ q(s_*) =p_2(s_*) q_4(s_*) -p_3(s_*) q_1(s_*).
\]
Hence $q(s_*)$ must have the form
\begin{align*}
q(s_*; \lambda) & = \begin{pmatrix} \alpha p_2(s_*)   & q_2(s_*) & q_3(s_*) & \alpha p_3(s_*) 
\end{pmatrix}.
\end{align*}
for some for some $ \alpha \in \R$. 
\end{proof}
This Lemma leads us to consider the following three cases. 
\begin{itemize}
\item[(I)] {\bf Simple, regular crossing: $\alpha \neq 0$ and $ p_2(s_*) \neq 0$.} If $\alpha \neq 0$, then the crossing is simple. If also $ p_2(s_*) \neq 0$, then Lemma \ref{lem:first-ordercrossing-form} implies $Q^{(1)}(v) > 0$, and we have monotonicity. 
\item[(II)] {\bf Simple, nonregular crossing: $\alpha \neq 0$ and $ p_2(s_*) = 0$.} If $\alpha \neq 0$, then the crossing is simple. If also $p_2(s_*) = 0$, then Lemma \ref{lem:first-ordercrossing-form} implies $Q^{(1)}(v) = 0$, and hence the crossing is non-regular. In this case, we show in Lemma \ref{lem:simple-nonregular} below that 
\begin{align}
Q^{(1)}(p(s_*)) &=0,  &
Q^{(2)}(p(s_*)) &=0, &
Q^{(3)}(p(s_*)) &>0,
\end{align} 
and so again we have monotonicity.
\item[(III)] {\bf Nonsimple crossing: $\alpha = 0$.} If $ \alpha = 0$ then both $p(s_*, \lambda)$ and $q(s_*, \lambda)$ are in $\ell_*^{sand}$, and so the crossing is not simple. 
Below we will see that this necessarily leads to a situation in which the crossing is {\it not} fully degenerate. This will prevent the theory developed in \S\ref{section: maslov index for nonregular crossings} from applying. Hence, we must rule out this case via Hypothesis \ref{hyp:degeneracy}.
\end{itemize}
We now state the lemma that corresponds to case (II) above, and which shows that in this case we still have monotonicity. 

\begin{lemma} \label{lem:simple-nonregular} Suppose we are in case (II) above. In other words, suppose the two basis solutions for $\ell(s_*, \lambda)$ are given by
\begin{align}
p(s_*;\lambda) & = \begin{pmatrix} 0 & 0  & p_3(s_*; \lambda) & 0
\end{pmatrix} \nonumber \\
q(s_*; \lambda) & = \begin{pmatrix} 0 & q_2(s_*; \lambda) & q_3(s_*; \lambda) & q_4(s_*; \lambda) 
\end{pmatrix} \label{E:basis-vector-form-II}
\end{align}
 where $  q_4(s_*; \lambda)  \neq 0$. 
Then 
\[
Q^{(1)}(p(s_*)) =0, \qquad  Q^{(2)}(p(s_*)) =0, \qquad  Q^{(3)}(p(s_*)) >0.
\]
\end{lemma}
\begin{proof}
We have via Lemma \ref{lem:first-ordercrossing-form} that $Q^{(1)}(p(s_*)) = 0$. Now we compute the second order crossing form. As in the proof of Lemma \ref{lem:first-ordercrossing-form}, we again use the definition of $Q^{(2)}(v)$, the Lagrangian structure, the representation \eqref{eq: spatial monotonicity graph formulation}, and the facts that $p_s = Bp$ and $q_s = Bq$ to find
\begin{align*} Q^{(2)}\big(v\big) & = \frac{d^2}{ds^2}\omega(v, A(s)v) \bigg|_{s = s_*} \\
& = \frac{d^2}{ds^2}\omega(v, v+ A(s)v) \bigg|_{s = s_*} \\
& = (d_1)_{ss}(s_*)\underbrace{\langle p(s_*), Jp(s_*)\rangle}_{=0}  + (d_2)_{ss}(s_*) \underbrace{\langle p(s_*), Jq(s_*) \rangle}_{=0} \\
& \qquad + \langle p(s_*), Jp_{ss}(s_*) \rangle + 2 \langle p(s_*), JB(s_*)\big( (d_1)_s(s_*)p(s_*) +  (d_2)_s (s_*)q(s_*)\big) \rangle.  
%& =  \langle p(s_*), Jp_{ss}(s_*) \rangle + 2 \underbrace{\langle p(s_*), JB(s_*)   (d_2)_s (s_*)q(s_*) \rangle}_{=0}  \\
%& = \langle p(s_*), Jp_{ss}(s_*) \rangle  \\
\end{align*}Up to this point, our calculation has not depended on the choice of $\mathcal{W}$. Indeed, as stated in Theorem \ref{thm: higher order quad form RS extension}, since we are in case (II) the ultimate value of the higher order crossing form does not depend on the choice of $\mathcal{W}$. Thus, we are free to make a convenient choice of $\mathcal{W}$ so as to enable us to see that some of the above terms in the above expression are zero. 

To that end, let $\mathcal W = \begin{pmatrix} I_n & 0 \end{pmatrix}$. Differentiating \eqref{eq: spatial monotonicity graph formulation} with respect to $s$ and evaluating at $s_*$, we obtain 
\[
A_s(s_*) p(s_*; \lambda) = p_s(s_*; \lambda) +  (d_1)_s(s_*)p(s_*; \lambda) + (d_2)_s(s_*)q(s_*; \lambda).
\]
The fact that $A(s)p(s_*; \lambda) \in \mathcal W$ for all $s$ near $s_*$ implies $A_s(s)p(s_*; \lambda) \in \mathcal W$ for all $s$ near $s_*$, and so we write
\begin{equation}\label{E:A-rep}
A(s) p(s_*; \lambda) = \begin{pmatrix} (a_1)(s) \\ (a_2)(s) \\ 0 \\ 0 \end{pmatrix}, \qquad A_s(s) p(s_*; \lambda) = \begin{pmatrix} (a_1)_s(s) \\ (a_2)_s(s) \\ 0 \\ 0 \end{pmatrix}.
\end{equation}
Using the characterizations of $p(s_*; \lambda)$ and $q(s_*; \lambda)$ in \eqref{E:basis-vector-form-II}, as well as the facts that $p_s = Bp$ and $q_s = Bq$, we arrive at 
\begin{align*}
\begin{pmatrix} (a_1)_s(s_*) \\ (a_2)_s(s_*) \\ 0 \\ 0 
\end{pmatrix} & = \begin{pmatrix} 0 \\ p_3(s_*) \\ 0 \\ 0
\end{pmatrix} + (d_1)_s(s_*) \begin{pmatrix} 0 \\ 0\\ p_3(s_*; \lambda) \\ 0
\end{pmatrix} + (d_2)_s(s_*) \begin{pmatrix} 0 \\ q_2(s_*) \\ q_3(s_*) \\ q_4(s_*)
\end{pmatrix} \\
&  =  \begin{pmatrix} 0 \\ p_3(s_*) + (d_2)_s(s_*)q_2(s_*) \\
(d_1)_s(s_*)p_3(s_*) + (d_2)_s(s_*) q_3(s_*) \\ 
 (d_2)_s(s_*) q_4(s_*)
\end{pmatrix}.
\end{align*}
By looking at the fourth component  which is assumed to be nonzero, we conclude that $(d_2)_s(s_*) = 0$. This then implies via the third component that $(d_1)_s(s_*) = 0$ as well. Returning to our calculation of the second order crossing form, we then have
\begin{align*} Q^{(2)}\big(v\big) & = \langle p(s_*), Jp_{ss}(s_*) \rangle + 2 \langle p(s_*), JB(s_*)\big( (d_1)_s(s_*)p(s_*) +  (d_2)_s (s_*)q(s_*)\big) \rangle \\
& = \langle p(s_*), Jp_{ss}(s_*) \rangle.  
\end{align*}
We now use the fact that $p_s = Bp$ to explicitly compute this quantity. We find 
\begin{align*}
Q^{(2)}\big(v\big) & = 
\langle p(s_*), Jp_{ss}(s_*) \rangle \\
& = \langle p(s_*), J \big[B(s_*)B(s_*) + B_s(s_*)\big]p(s_*) \rangle \\ 
& = \begin{pmatrix} 0 & 0 & p_3(s_*) & 0
\end{pmatrix}\begin{pmatrix}  \frac{d}{dx} f'\circ\varphi(s_*) & 0 & 0 &  f'\circ\varphi(s_*) -\lambda - 1  \\ 
0 & 0 & 1 & -2 \\ 0 & -1 & 0 & 0 \\ -f'\circ\varphi(s_*) + \lambda +1 & 2 & 0 & 0 \end{pmatrix} \begin{pmatrix} 0 \\ 0 \\ p_3(s_*) \\ 0
\end{pmatrix} \\
& = \begin{pmatrix} 0 & 0 & p_3(s_*) & 0 
\end{pmatrix}\begin{pmatrix} 0 \\ p_3(s_*; s_*) \\ 0  \\ 0
\end{pmatrix} \\
& = 0.
\end{align*}
Finally, we determine the third order crossing form. We have
\begin{align*} Q^{(3)}\big(v\big) & = \frac{d^3}{ds^3}\omega(v, v+ A(s)v) \bigg|_{s = s_*} \\
& = \underbrace{\left\langle p(s_*), J\big((d_1)_{sss}(s_*) p(s_*) + (d_2)_{sss}q(s_*)\big) \right\rangle}_{(3a)} + 3\underbrace{\left\langle p(s_*), J \left(  (d_1)_{ss}(s_*)p_s(s_*) + (d_2)_{ss}(s_*)q_s(s_*) \right)\right\rangle}_{(3b)} \\
& \qquad + 3\underbrace{\left\langle p(s_*), J \left(  (d_1)_s(s_*)p_{ss}(s_*) + (d_2)_s(s_*)q_{ss}(s_*) \right)\right\rangle}_{(3c)} + \underbrace{\left\langle p(s_*), J \left(  d_1(s_*)p_{sss}(s_*)(s_*) + d_2(s_*)q_{sss}(s_*) \right)\right\rangle}_{(3d)}.
\end{align*}
We address each of these terms individually. The first term vanishes due to the Lagrangian property: 
\[
(3a) = (d_1)_{sss}(s_*)\left\langle p(s_*), J p(s_*) \right\rangle  + (d_2)_{sss}(s_*) \left\langle p(s_*), Jq(s_*) \right\rangle = 0.
\]
The term $(3b)$ can be simplified by using the fact that above we showed $(d_2)_s(s_*) = 0$ and $(d_1)_s(s_*) = 0$. By differentiating \eqref{eq: spatial monotonicity graph formulation} twice, we can relate the second derivatives of $d_{1,2}(s_*)$ to their first derivatives, which are both zero. In particular, we find
\[
A_{ss}(s_*)p(s_*) = d_1''(s_*)p(s_*) + p_{ss}(s_*) + d_2''(s_*) q(s_*).
\]
This allows us to write
\begin{align*}
(3b) &= 3\left\langle p(s_*), J \left(  (d_1)_{ss}(s_*)p_s(s_*) + (d_2)_{ss}(s_*)q_s(s_*) \right)\right\rangle \\
& = 3\left\langle p(s_*), J B\left(  (d_1)_{ss}(s_*)p(s_*) + (d_2)_{ss}(s_*)q(s_*) \right)\right\rangle \\
&= 3\left\langle p(s_*), J B \left( A_{ss}(s_*) - p_{ss}(s_*) \right)\right\rangle.
\end{align*}
As above, using \eqref{E:A-rep} we find 
\[
\langle (JB)^Tp(s_*), A_{ss}(s_*)p(s_*)\rangle = \begin{pmatrix} 0 & 0 & 0 & -p_3(s_*)
\end{pmatrix} \begin{pmatrix} (a_1)_{ss}(s_*) \\ (a_2)_{ss}(s_*) \\ 0 \\ 0
\end{pmatrix} = 0.
\]
Thus, 
\[
(3b) = -3 \left\langle p(s_*), J B p_{ss}(s_*) \right\rangle.
\]
Next, since it was shown above that $d_1'(s_*) = d_2'(s_*) = 0$, we find
\[
(3c) = 0.
\]
Finally, the fact that $d_2(s_*) = 0$ and $d_1(s_*) = 1$ implies
\[
(3d) = \left\langle p(s_*), J p_{sss}(s_*) \right\rangle.
\]
Putting these four pieces together and again using the equation $p_s = Bp$, we find
\begin{align*}Q^{(3)}\big(v\big) & = \left\langle p(s_*), J p_{sss}(s_*)\right\rangle - 3 \big\langle p(s_*), J B p_{ss}(s_*)\big\rangle\\
& = \langle p(s_*), J \big[BBB + B_{ss}+ 2B_sB + BB_s \big] p(s_*) \rangle - 3 \langle p(s_*), JB \big[BB + B_s \big] p(s_*) \rangle \\
& = \langle p(s_*), J \big[-2BBB + B_{ss} +2 B_sB - 2BB_s \big] p(s_*) \rangle.
\end{align*}
Using \eqref{eq: linearized first order}, we can explicitly compute the matrix in this quantity. By doing so, we find
\begin{align*}Q^{(3)}\big(v\big) 
& = \langle p(s_*), J \big[-2BBB + B_{ss} +2 B_sB - 2BB_s \big] p(s_*) \rangle \\
& = \begin{pmatrix} 0 & 0 & p_3(s_*) & 0 
\end{pmatrix} \\
& \quad \cdot \begin{pmatrix} \frac{d^2}{ds^2} f'\circ\varphi(s_*) & 2 \lambda - 2f'\circ\varphi(s_*) +  2 & 0 & 2 \frac{d}{ds}f'\circ\varphi(s_*) \\ 2\lambda - 2 f'\circ\varphi(s_*) + 2 & 4 & 0 & 0 \\ 0 & 0 & 2 & -4 \\ 
2\frac{d}{ds} f'\circ\varphi(s_*) & 0  & -4 & 2f'\circ\varphi(s_*) - 2\lambda + 6
\end{pmatrix} \begin{pmatrix}
0 \\ 0 \\ p_3(s_*) \\0
\end{pmatrix} \\
& = 2 p_3^2(s_*) \\
& > 0. 
\end{align*} 
Note that $p_3(s_*) \neq 0$ or else $p(s; s_*)$ would be the zero solution. 
\end{proof}

Finally, we   consider case (III). In this case we can assume without loss of generality that 
\begin{align*}
			p(s_*;\lambda) & = \begin{pmatrix} 0 & 0 & p_3(s_*; \lambda) & 0
			\end{pmatrix} \\
			q(s_*; \lambda) & = \begin{pmatrix} 0   & q_2(s_*; \lambda) & 0 & 0
			\end{pmatrix}.
\end{align*}
In this case, the first order crossing form is a two-by-two matrix, and we need to determine its two eigenvalues. This means we must compute $Q^{(1)}(v)$ for $v = p(s_*)$ and $v = q(s_*)$. Since $p(s_*)$ has the same form as in Lemma \ref{lem:simple-nonregular}, we again find that $Q^{(1)}(p(s_*)) = 0$. To compute  $Q^{(1)}(q(s_*))$, we modify the calculation in the proof of Lemma \ref{lem:first-ordercrossing-form} to find
\[
Q^{(1)}(q(s_*)) = \langle q(s_*, \lambda), JB q(s_*, \lambda) \rangle = (q_2(s_*; \lambda))^2 > 0. 
\]
Thus, we find that the first order crossing for is degenerate, but it is not identically zero. Thus, this case lies outside the scope of the theory developed in \S\ref{section: maslov index for nonregular crossings}. Thus, we must rule it out with the following hypothesis.
\begin{hypothesis}\label{hyp:degeneracy} The dimension of the intersection of  $\ell(s, 0)$ with the sandwich plane $\ell_*^{sand}$ is at most one.  
%	In the context of the notation of Lemma \ref{lem:cases}, $\alpha \neq 0$. {\color{red} Change to simple crossing?}
\end{hypothesis}
This hypothesis ensures that all crossings are simple, and of the three cases laid out after Lemma \ref{lem:cases}, only cases I and II are possible. Moreover should a degenerate crossing occur, the entire crossing form is identically zero, and hence the theory of \S\ref{section: maslov index for nonregular crossings} applies. We note that, in \S \ref{S:numerics}, we provide evidence that, for all pulses that we consider, this hypothesis is indeed satisfied.

Thus, we have proven the following monotonicity result. 

\begin{proposition}\label{prop: s mono} Suppose that Hypothesis \ref{hyp:degeneracy} is satisfied. For any fixed $\lambda \in[0,\lambda_\infty]$, the path $s \mapsto \ell(s,\lambda)$ as $s$ runs from $-\infty$ to $L$ has a positive contribution to the Maslov index. In other words, 
$$  \text{Mas}(\ell, \ell_*^{sand}; D_{2,L}) = \sum_{s \in [-\infty, L]} \dim(\ell(s,\lambda) \cap \ell_*^{sand}) \geq 0.$$
\end{proposition}

Finally, we prove Theorem \ref{thm: morse index half line} given at the beginning of this section. \\

\noindent \textbf{Theorem \ref{thm: morse index half line}.} \textit{Suppose that Hypothesis \ref{hyp:degeneracy} is satisfied. The number of positive eigenvalues of $\mathcal L_L$ is equal to the number of conjugate points in $(-\infty, L)$ counted with multiplicities. }

\begin{proof}
Via homotopy invariance and Proposition \ref{prop: zero bottom} and Proposition \ref{prop: zero left}, we know that 
$$ \text{Mas}(\ell,\ell_*^{sand}; D_L) = \sum_{i=1}^4 \text{Mas}(\ell, \ell_*^{sand};D_{i,L}) = \text{Mas}(\ell, \ell_*^{sand}; D_{2,L}) +\text{Mas}(\ell, \ell_*^{sand}; D_{3,L}) =  0.$$
By the construction of the eigenvalue problem in Equation \eqref{eq: half line eval prob BC}, we see that all crossings on $D_{3,L}$ correspond to non-positive eigenvalues. Thus, if we define $\text{Mor} (\mathcal L_L)$ to be the number of strictly positive eigenvalues of $\mathcal L_L$, then
$$ \text{Mas} (\ell, \ell_*^{sand};D_{3,L}) = \text{Mor}(\mathcal L_L) + \dim\ker \mathcal L_L.$$
Lastly, if there was a crossing at $(0,L)$, this crossing would contribute to both $\text{Mas}(\ell, \ell_*^{sand};D_{2,L})$ and $\text{Mas}(\ell, \ell_*^{sand};D_{3,L})$. By Lemma \ref{lemma: conj pt eval corr}, this can only happen if $0$ is an eigenvalue of $\mathcal L_L$ and $L$ is a conjugate point. Therefore, the number of nonnegative eigenvalues of $\mathcal L_L$ is equal to the number of conjugate points in $(-\infty, L)$ counting multiplicity. 
\end{proof}

%%%%%%%%%%%%%%%%%%%%%%%%%%%%%%%%%%%%%%%%%%%%%%%%%%%%
\subsection{\texorpdfstring{Extension to $\R$}{Extension to R}}\label{sec: extension to R}
%%%%%%%%%%%%%%%%%%%%%%%%%%%%%%%%%%%%%%%%%%%%%%%%%%%%
In general, we do not expect the spectrum of an operator on a truncated domain and an unbounded domain to be the same. In order to argue that the number of positive eigenvalues of the operator $\mathcal L$ acting on $(-\infty, L]$, which we called $\mathcal L_L$, coincides with the number of positive eigenvalues of $\mathcal L$ acting on $\R$, we can rely on the results of \cite{BCJ18, SS00}. We note that some results on point spectra of truncated operators include Sections 2.5 and 4 of \cite{Daners} and \ADD{\cite{Kreiss1999StabilityOT}}. 

\begin{lemma}\label{lemma: eigenvalues vary} Let $0 \leq \lambda_1(L) \leq \lambda_2(L) \leq \dots \leq \lambda_n(L)$ be the non-negative eigenvalues for the operator $\mathcal L_L$. Each $\lambda_j(L)$ is
\begin{enumerate}[i)]
\item a continuously varying function of $L$; 
\item strictly increasing with $L$.
\end{enumerate}
\end{lemma}
\begin{proof}~
\begin{enumerate}[i)]
\item Since the essential spectrum of $\mathcal L_L$ is strictly negative, the eigenvalues of $\mathcal L_L$ can be thought of as the eigenvalues of a continuous family of finite dimensional operators parameterized by $L$. See for example, \cite{kato}[Section II.5] or \cite{Daners}[Sections 2.5, 4]. 
\item The proof \ADD{of} this statement is virtually identical to that of \cite{BCJ18}[Remark 3.5].
\end{enumerate}
\end{proof}

Finally, we have that Theorem \ref{thm: morse index half line} can be extended to $\R$. 

\begin{theorem}\label{thm: extension to R}
There exists $L_\infty \in \R$ such that for all $L > L_\infty$ 
the number of positive eigenvalues of the restricted operator $\mathcal L_L$ (Equation \ref{eq: half line eval prob BC}) coincides with that of the operator $\mathcal L$ on $\R$. Additionally, if we denote the number of positive eigenvalues of a differential operator $\mathcal H$ as $\text{Mor}(\mathcal H)$, then 
\begin{enumerate}[a)]
\item $L$ is not a conjugate point and $\mathcal L_L$ is invertible; 
\item the number of conjugate points of $\mathcal L_L$ is finite and independent of $L$, meaning 
$$ \text{Mor}(\mathcal L) = \#\text{ conjugate points in }(-\infty, + \infty). $$
\end{enumerate}
\end{theorem} 
\begin{proof} 
The argument relies on Lemma \ref{lemma: eigenvalues vary} and is virtually identical to \cite{BCJ18}[Theorem 4.1].
\end{proof}

%%%%%%%%%%%%%%%%%%%%%%%%%%%%%%%%%%%%%%%%%%%%%%%%%%%%%%%%%%%%%%%%%%%%%%%%%%%%%%%%%%%%%%%%%%%%%%%
%%%%%%%%%%%%%%%%%%%%%%%%%%%%%%%%%%%%%%%%%%%%%%%%%%%%%%%%%%%%%%%%%%%%%%%%%%%%%%%%%%%%%%%%%%%%%%%
%%%%%%%%%%%%%%%%%%%%%%%%%%%%%%%%%%%%%%%%%%%%%%%%%%%%%%%%%%%%%%%%%%%%%%%%%%%%%%%%%%%%%%%%%%%%%%%
%%%%%%%%%%%%%%%%%%%%%%%%%%%%%%%%%%%%%%%%%%%%%%%%%%%%%%%%%%%%%%%%%%%%%%%%%%%%%%%%%%%%%%%%%%%%%%%

%\subfile{numerics}

\section{Numerical Computation of Conjugate Points}\label{S:numerics}
Here, we describe how we numerically compute the conjugate points for three symmetric pulse solutions to the Swift-Hohenberg equation and compare the number of conjugate points to the number of unstable eigenvalues computed experimentally using Fourier spectral methods. 

Recalling the vector field parameters $\nu$ and $\mu$ from \eqref{eq: f - nonlinearity of SH}, we consider two specific pairs of parameters $(\nu, \mu)$. In the first, we follow \cite{BKLSWSnakes09,burkeknobloch06}, and set $\nu = 1.6$ and $\mu = 0.05$. These parameter values lie in the non-snaking region and there are two branches of symmetric pulse solutions. In the second set of parameter values, we set $\nu = 1.6$ and $\mu = 0.20$. This lies within an interesting parameter region in which the system displays homoclinic snaking. For these parameter values, there are infinitely many symmetric pulse solutions. 

Using perturbation theory, it was shown in \cite{burkeknobloch06} that the two branches of the symmetric pulse solutions can be approximated via 
\begin{equation}\label{eq: BK normal form}
u_{\phi}(x; \nu, \mu) = 2\sqrt{\frac{2\mu}{\gamma}}\text{sech}\left(\frac{x\sqrt\mu}{2}\right)\cos(x+\phi), \qquad \gamma = \frac{38\nu^2}{9}-3, \qquad \phi = 0, \pi.
\end{equation}
We will refer to stationary symmetric pulse solutions to \eqref{eq: swift hohenberg} as $\varphi_\phi(x; \nu, \mu)$, where $\phi$ denotes the phase condition from the approximation $u_\phi$. It was found via numerical experiments that in the non-snaking region, the $\phi = \pi$ branch has $2$ unstable eigenvalues and the $\phi = 0$ branch has $1$ unstable eigenvalue \cite{burkeknobloch06}. We corroborate these findings and show that the pulse solution in the snaking region has $0$ unstable eigenvalues and is therefore spectrally stable. \ADD{We note that these computations of eigenvalue counts are not mathematical proofs as such, and in future work we plan validate this via computer-assisted proof; see Section \ref{S:future} for more discussion.}
%----------------------------------------------------------------
\subsection{Eigenvalues via Fourier Spectral Methods}
%----------------------------------------------------------------
We reproduce results found in \cite{burkeknobloch06} using spectral methods similar to those found in \cite{ACLSswifthohenbergHex}. We will write the normal form approximation $u_{\phi}$ for a pulse solution on the interval from $[-L_f,L_f]$ in its Fourier expansion, up to order $N$, and formulate the accompanying system of ODEs. We can then perform Newton's method to refine the Fourier coefficients to obtain a better approximation of the pulse solution. 

We begin by writing the pulse in its Fourier expansion,
\begin{equation}
\varphi(x) = \sum_{k = -\infty}^\infty a_k e^{i\pi kx/L_f}.
\end{equation}
Truncating this expansion at order $N$ gives the truncated approximation for the pulse solution 
\begin{equation}\label{eq: truncated fourier expansion for pulse}
\varphi^{(N)}(x) = \sum_{k = -N}^N a_k e^{i\pi kx/L_f}.
\end{equation}
Define $a = (a_{-N}, \dots, a_N)$ to be the vector of Fourier coefficients. By plugging \eqref{eq: truncated fourier expansion for pulse} into \eqref{eq: swift hohenberg}, we obtain the $2N + 1$ dimensional system of ODEs given by the operator $F:  \R^{2N+1} \to \R^{2N+1}$, whose zeros correspond to the Fourier coefficients of a solution to  \eqref{eq: swift hohenberg}. For $k = -N, \dots, N$, the $k$th component of $F$ is given via
\begin{equation}\label{eq: k component of F fourier}
F_k(a) =   \left[-\mu -\left(1-\frac{k^2\pi^2}{L_f^2} \right)^2 \right]a_k - (a*a*a)_k + \nu(a*a)_k.
\end{equation}
The operation $*$ is defined by 
$$ (a*a)_k  = \sum_{\substack{k_1 + k_2 = k \\ k_i \in \Z}} a_{k_1}a_{k_2} \ \text{ and } (a*a*a)_k = \sum_{\substack{k_1 + k_2 + k_3 = k \\ k_i \in \Z}} a_{k_1}a_{k_2}a_{k_3}.$$

Our goal is to use the Fourier coefficients of $u_\phi$ as the initial condition and run Newton's method on $F(a)$ to obtain an approximation of the $N$th order truncated pulse solution $\varphi^{(N)}_\phi(x)$. However, due to translational symmetry, $DF(a)$ has a zero eigenvalue, so we cannot compute Newton's method on the Fourier coefficients for the entire pulse profile on $[-L_f, L_f]$. To get around this, we perform Newton's method on half of the pulse profile on $[0, L_f]$ and refine the coefficients $(a_0, \dots, a_N)$. Therefore, when performing Newton's method, we use an $N+1$ dimensional system of ODEs with the $k$th component given by \eqref{eq: k component of F fourier} for $k = 0, \dots, N+1$. Finally, since $\varphi$ is symmetric, it is the case that $a_{k} = a_{-k}$, which allows us to recover the approximation for the entire pulse.

By using the Fourier coefficients of $u_\phi(x)$ as the initial condition, we can run Newton's method to obtain the approximation of the $N$th order truncated pulse solution $\varphi^{(N)}_\phi(x)$ on the interval $[-L_f, L_f] = [-100, 100]$. We refer to the numerical approximation of the truncated pulse solution as $\bar \varphi^{(N)}_\phi(x)$. The approximations are given in Figure \ref{fig: approx  pulse solutions}. \ADD{We note that for $\mu = 0.2$, we are far from the regime where the normal form gives a good approximation, and we use $3u_\phi(x)$ as the initial seed for Newton's method.}

\begin{figure}[H]
    \begin{subfigure}[b]{\textwidth}
     \centering
     \begin{subfigure}[b]{0.49\textwidth}
         \centering
         \includegraphics[width=\textwidth]{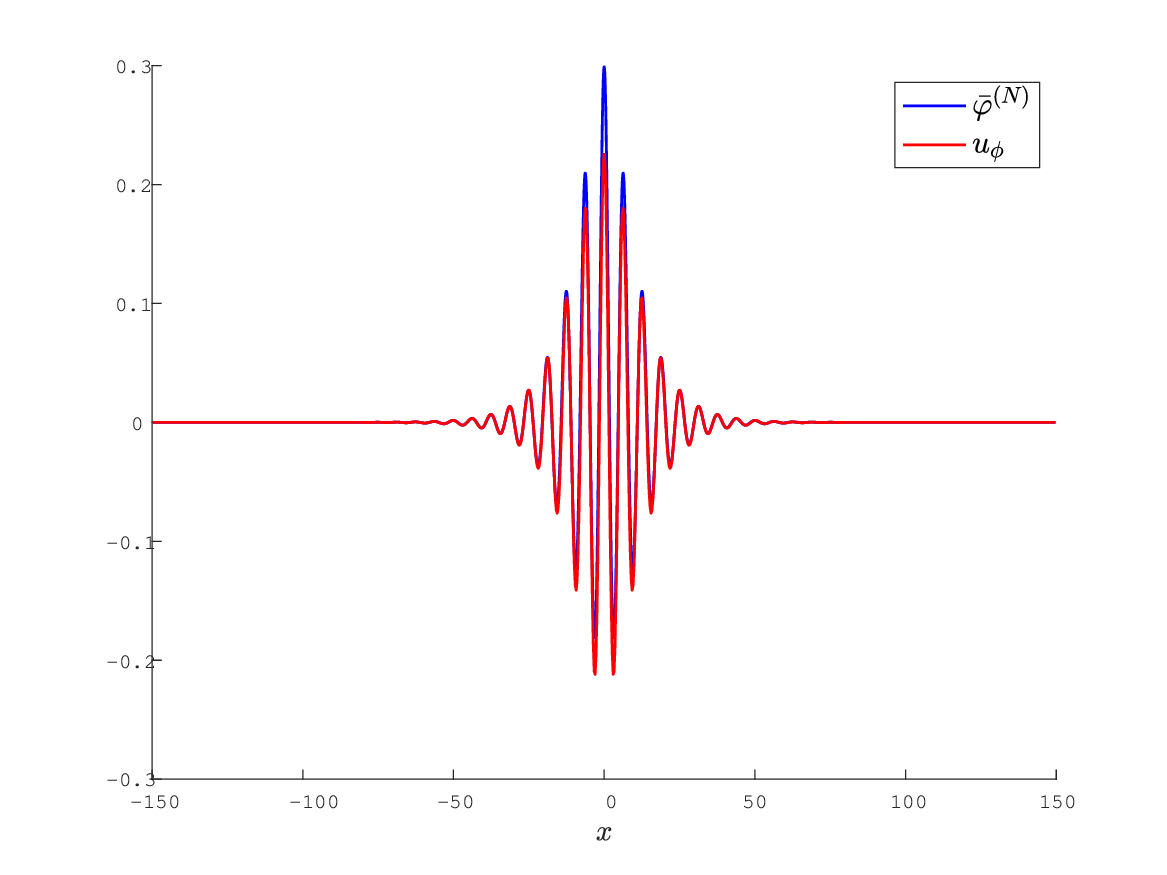}
         \subcaption{$\phi = 0$, $\nu = 1.6$ and $\mu = 0.05$}
         \label{fig:three sin x}
     \end{subfigure}
     \hfill 
     \begin{subfigure}[b]{0.49\textwidth}
         \centering
         \includegraphics[width=\textwidth]{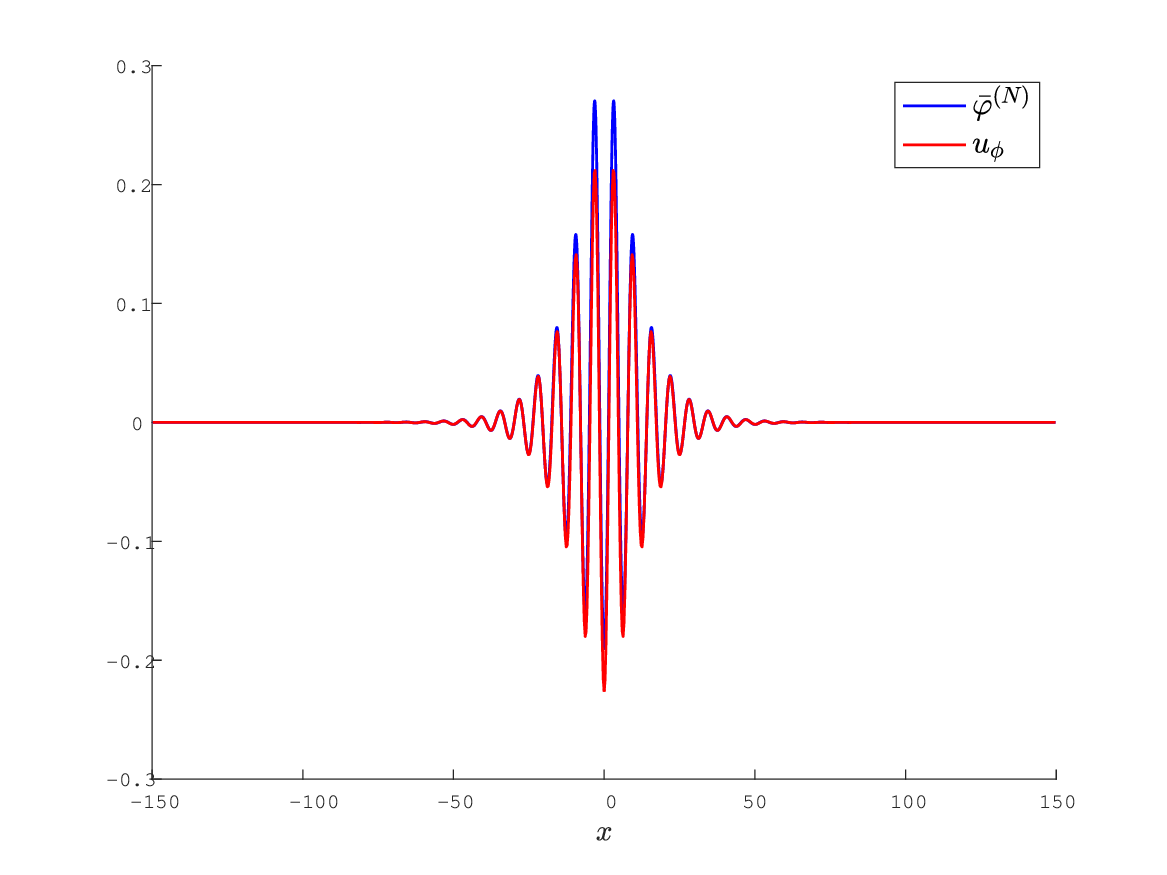}
         \subcaption{$\phi = \pi$, $\nu = 1.6$ and $\mu = 0.05$}
         \label{fig:y equals x}
     \end{subfigure}
        \label{fig: approx  pulse solutions}
    \end{subfigure}
    \centering
    \begin{subfigure}[b]{0.49\textwidth}
        \centering
         \includegraphics[width=\textwidth]{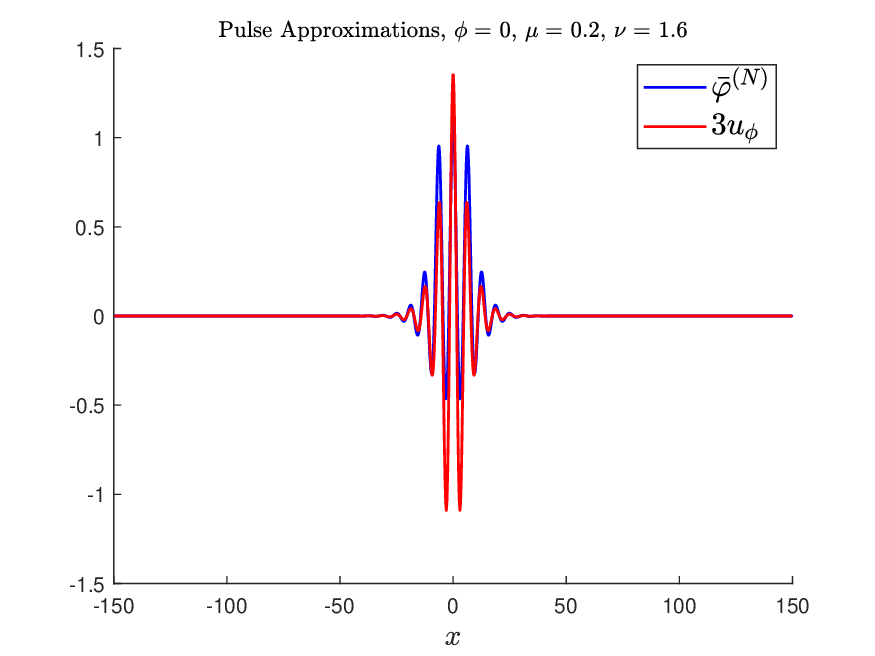}
         \caption{$\phi = 0$, $\nu = 1.6$ and $\mu = 0.2$.}
    \end{subfigure}
    \caption{Approximate symmetric pulse profiles using the normal form equation in \eqref{eq: BK normal form} and after using Newton's method on the Fourier coefficients.   }\label{fig: approx  pulse solutions}
\end{figure}

Denote the vector of approximate Fourier coefficients obtained via Newton's method for  a pulse solution as $\bar a$. One can estimate the spectrum of the associated pulse solution by computing the eigenvalues of the matrix $DF(\bar a)$. The eigenvalues for our example pulse solutions are given in Table \ref{table: unstable eigenvalues}. For the parameters $\nu = 1.6$ and $\mu = 0.05$, these results are consistent with the results found in \cite{burkeknobloch06}. 

\begin{table}[H]
\begin{center}
\begin{tabular}{||c c||} 
 \hline
 Symmetric Pulse  & Unstable Eigenvalues \Tstrut \\ 
 $\varphi_\phi(x; \nu, \mu)$ & \Bstrut \\
 \hline
 \hline
 $\varphi_0(x; 1.6, 0.05)$ & 0.1209 \Tstrut\Bstrut \\
 \hline
 $\varphi_\pi(x; 1.6, 0.05)$ &   0.0058 \Tstrut \\
 &  0.1179 \Bstrut \\
 \hline
 $\varphi_0(x; 1.6, 0.2)$ & \text{N/A} \Tstrut\Bstrut \\
 \hline
 \hline
\end{tabular}
\end{center}
 \caption{Approximate unstable eigenvalues for the three symmetric pulse solutions via Fourier methods. }\label{table: unstable eigenvalues}
\end{table}
%----------------------------------------------------------------
\subsection{Eigenvalues via Conjugate Points}
%----------------------------------------------------------------

We now compute the number of conjugate points for the two symmetric pulse solutions. We have shown there exists  $L_+ > 0$ such that the conjugate points are all contained in the interval $(-\infty, L_+]$. In this work, our numerics serve as a proof-of-concept, so it is sufficient to choose $L_{cp}$ large via experimentation and compute the conjugate points on the interval from $[-L_{cp},L_{cp}]$. 

In order to compute numerical approximations for the conjugate points associated to each pulse solution, we must first approximate the unstable subspace, $\E^{u, \phi}_-(x; 0)$ on the interval from $[-L,L]$. The basis vectors for $\E^{u, \phi}_-(x;0)$ solve the nonautonomous system given by \eqref{eq: SH eigenvalue problem on R}. We treat this as an autonomous problem by simultaneously computing the pulse $\varphi$ so that we can use \texttt{ode45}
 in Matlab to integrate. After computing an approximation for the unstable subspace, we can use the following result to detect conjugate points. 

\begin{lemma}\label{lemma: intersection det = 0} Denote the frame matrix of $\E^u_-(x;0)$ as 
$$\begin{pmatrix} M_1(x)  \\ M_2(x) \\ M_3(x)
\end{pmatrix}, \qquad M_1(x), M_3(x) \in \R^{1\times2}, M_2(x) \in \R^{2 \times 2}$$
and recall the reference plane $\ell_*^{sand}$ given in Definition \ref{def: sandwich plane}. Then $\E^u_-(x;0) \cap \ell_*^{sand} \neq \{0\}$ if and only if $\det \begin{pmatrix} M_1(x) \\ M_3(x)
\end{pmatrix} = 0$. 
\end{lemma}
\begin{proof} First, suppose that $\E^u_-(x) \cap \ell_*^{sand} \neq \{0\}$. Then by definition, there is a nonzero vector $v = (v_1,v_2)$ such that 
$$ \begin{pmatrix} M_1(x) \\ M_2(x) \\ M_3(x)
\end{pmatrix} \begin{pmatrix} v_1 \\ v_2
\end{pmatrix} = \begin{pmatrix} 0 \\ \tilde v_2 \\ \tilde v_3 \\ 0
\end{pmatrix}, \qquad \qquad \tilde v_2, \tilde v_3 \in \R .$$
However, this means $\begin{pmatrix} M_1(x) \\ M_3(x) \end{pmatrix}v = 0$. Thus, $\ker \begin{pmatrix} M_1(x) \\ M_3(x)
\end{pmatrix}$ is nontrivial so $\det \begin{pmatrix} M_1(x) \\ M_3(x)
\end{pmatrix} = 0$. 

Now suppose that $\det \begin{pmatrix} M_1(x) \\ M_3(x) \end{pmatrix} = 0$. Then there is a vector $v$ such that $\begin{pmatrix} M_1(x) \\ M_3(x)
\end{pmatrix}v = 0$. Since $\E^u_-(x)$ is a Lagrangian subspace, its frame matrix has full rank, so 
$$ \begin{pmatrix} M_1(x)  \\ M_2(x) \\ M_3(x)
\end{pmatrix} \begin{pmatrix} v_1 \\ v_2
\end{pmatrix} = \begin{pmatrix} 0 \\ \tilde v_2 \\ \tilde v_3 \\ 0
\end{pmatrix}, $$
with at least one of $\tilde v_2, \tilde v_3$ nonzero. Therefore, there is a nonzero vector in the intersection of $\mathbb E^u_-(x)$ and $\ell_*^{sand}$.
\end{proof}

\ADD{\begin{remark}
We note that, since we fully compute the frame for the unstable subspace, we could use this frame to determine the spectral flow associated with this path of Lagrangian frames \cite{furutani}. This would allow us to generalize our method to instances where monotonicity is not guaranteed. However, since we have shown the crossings are monotonic, we do not need to take this approach here.
\end{remark}}

Thus, we calculate the determinant of the submatrix formed from the first and fourth rows of the frame matrix for $\E^{u, \phi}_-(x; 0)$, calculate the zeros corresponding to conjugate points. We fix $L_{cp} = 60$ and compute 
$$\E^{u,\phi}_{-}(x;0) = \begin{pmatrix} | & | \\ W^{-,u}_1 & W^{-,u}_2 \\ | & |
\end{pmatrix}.$$
Define
$$ \det A(x) = \det\begin{pmatrix} (W^{-,u}_1)_1 & (W^{-,u}_2)_1 \\ (W^{-,u}_1)_4 & (W^{-,u}_2)_4
\end{pmatrix}.$$
Via Lemma \ref{lemma: intersection det = 0}, the zeros of $\det A(x)$ correspond to conjugate points. 
  To verify that each of the conjugate points correspond to a simple crossing, we check that $ \| A(x_*))\| > 10^{-3}$ at each of the conjugate points $x_*$. 

First, we look at the pulse solution with parameter values $\nu = 1.6$, $\mu = 0.05$ and phase condition $\phi = 0$. This solution lies outside of the snaking region and we denote it via $\varphi_0(x; 1.6, 0.05)$. Via Fourier methods, we find one unstable eigenvalue (Figure \ref{fig:unstable pulse phi = 0 evals}). When plotting the determinant of $A(x)$ associated to this pulse, we see that there is one zero (Figure \ref{fig:unstable pulse phi = 0 det}).

% unstable pulse, phi = 0
\begin{figure}[H]
     \centering
     \begin{subfigure}[b]{.65\textwidth}
         \centering
         \includegraphics[width=\textwidth]{figures/0_pulse_intro.eps}
         \caption{Approximate pulse profile.}
      \label{fig:unstable pulse phi = 0 approx}
     \end{subfigure}
     \hfill
\begin{subfigure}[b]{.3\textwidth}

\begin{tabular}{||c c||}
 \hline
 Unstable & Conjugate \Tstrut \\
 Eigenvalues ($\lambda$)  & Points ($x$) \Bstrut \\
 \hline
 \hline
 0.1209 & 1.2400 \Tstrut\Bstrut \\
 \hline
 \hline
\end{tabular}
\caption{Unstable eigenvalues and conjugate points.}\label{fig:unstable pulse phi = 0 evals}
     \end{subfigure}
         \begin{subfigure}[b]{\textwidth}
         \centering
          \includegraphics[width=\textwidth]{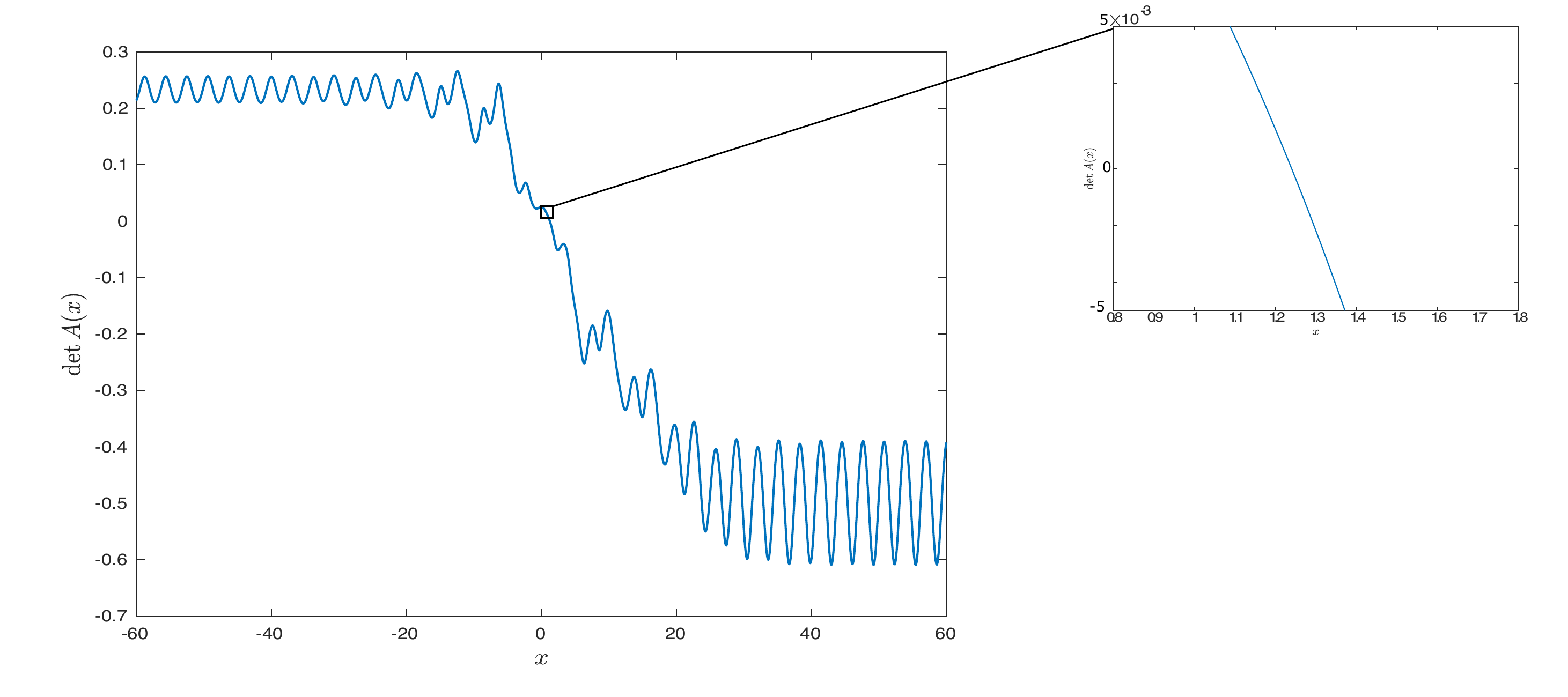}
         \caption{$\det A(x)$ and the zero at $10^{-3}$ scale.}
         \label{fig:unstable pulse phi = 0 det}
     \end{subfigure}
        \caption{Information for the pulse with $\phi = 0$, $\mu = 0.05$ and $\nu = 1.6$: $\varphi_0(x; 1.6, 0.05)$}
        \label{fig:unstable pulse phi = 0}
\end{figure}

Second, we look at the other symmetric pulse with parameter values $\nu = 1.6$ and $\mu = 0.05$ with phase condition $\phi = \pi$. We denote this pulse as $\varphi_\pi(x; 1.6, 0.05)$. Via Fourier spectral methods, we find two unstable eigenvalues (Figure \ref{fig:unstable pulse phi = pi evals}). By inspecting the plot of $\det A(x)$, we see that there are two zeros (Figure \ref{fig:unstable pulse phi = pi det}). 

% unstable pulse, phi = pi
\begin{figure}[H]
     \centering
     \begin{subfigure}[b]{.65\textwidth}
         \centering
         \includegraphics[width=\textwidth]{figures/pi_pulse_intro.eps}
         \caption{Approximate pulse profile.}
      \label{fig:unstable pulse phi = pi approx}
     \end{subfigure}
     \hfill
\begin{subfigure}[b]{.3\textwidth}
\begin{center}
\begin{tabular}{||c c||} 
 \hline
 Unstable & Conjugate \Tstrut  \\
 Eigenvalues ($\lambda$)  &  Points ($x$) \Bstrut \\ 
 \hline
 \hline
0.0058 &  -0.6310 \Tstrut  \\
 0.1179   & 17.5887 \Bstrut \\
 \hline
 \hline
\end{tabular}
\end{center}
\caption{Unstable eigenvalues and conjugate points.}\label{fig:unstable pulse phi = pi evals}
     \end{subfigure}
         \begin{subfigure}[b]{\textwidth}
         \centering
          \includegraphics[width=\textwidth]{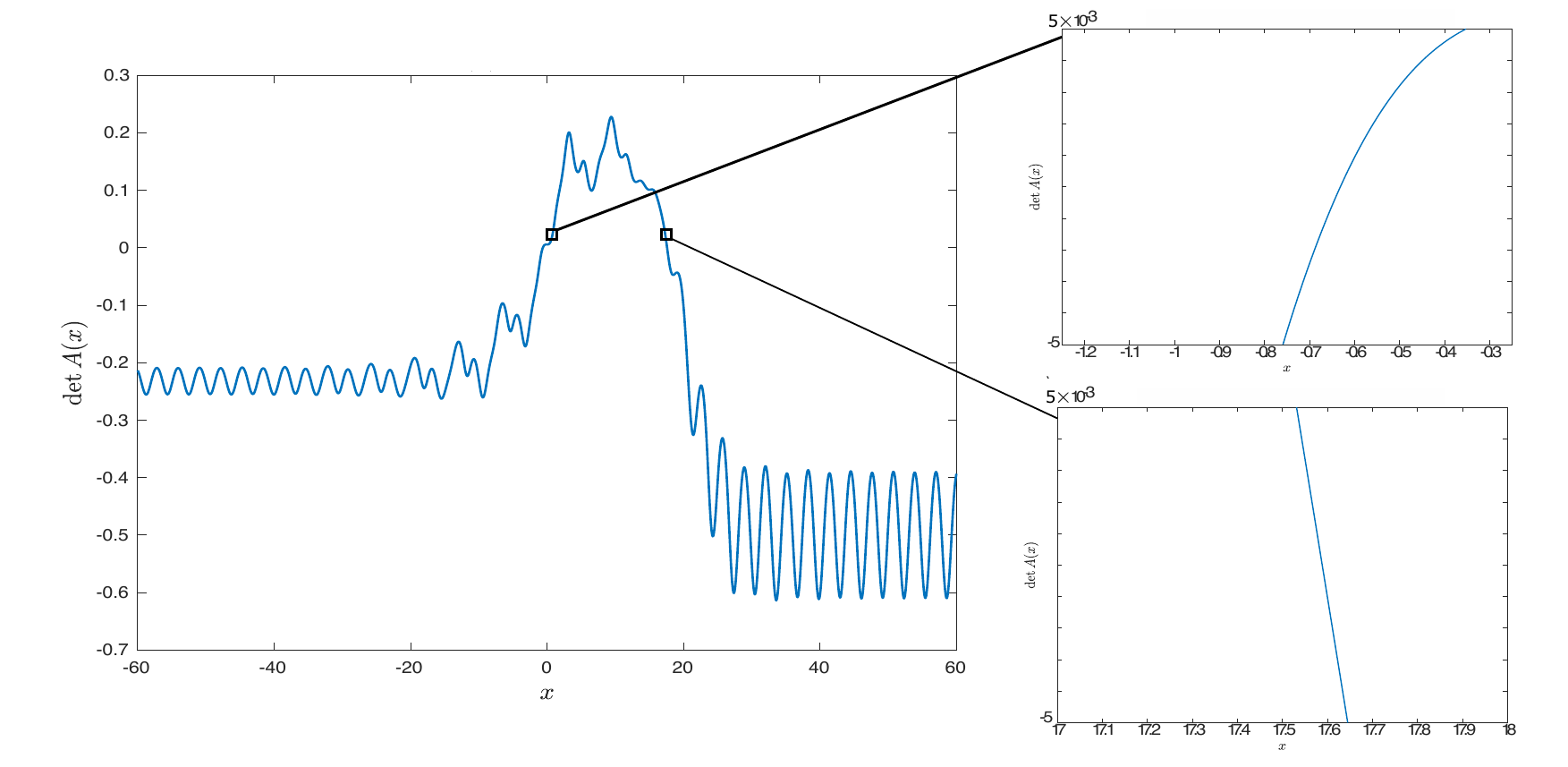}
         \caption{$\det A(x)$ and the zeros at $10^{-3}$ scale.}
         \label{fig:unstable pulse phi = pi det}
     \end{subfigure}
        \caption{Information for the pulse with $\phi = \pi$, $\mu = 0.05$ and $\nu = 1.6$: $\varphi_\pi(x; 1.6, 0.05)$}
        \label{fig: unstable pulse phi = pi}
\end{figure}
Taken together, the information about $\varphi_\phi(x; 1.6, 0.05)$, with $\phi = 0, \pi$ contained in Figures \ref{fig:unstable pulse phi = 0} and \ref{fig: unstable pulse phi = pi} are consistent with results found in \cite{burkeknobloch06} and allow us to conclude that these two pulses are spectrally unstable. Since spectral stability in this case is necessary for linear and nonlinear stability, we see that these two pulses are linear/nonlinearly unstable as well. 

Finally, we consider the symmetric pulse with parameter values $\nu = 1.6$, $\mu = .20$ and phase condition $\phi = 0$. This pulse lies within the parameter region that supports homoclinic snaking and is one of infinitely many asymmetric pulses for these parameter values. By performing Fourier spectral methods, we find no unstable eigenvalues. Additionally, when we plot $\det A(x)$, we see that there are no zeros. This information can be found in Figure \ref{fig: stable pulse}.

% stable pulse, phi = 0
\begin{figure}[H]
     \centering
     \begin{subfigure}[b]{.65\textwidth}
         \centering
         \includegraphics[width=\textwidth]{figures/stable_pulse_intro.eps}
         \caption{Approximate pulse profile.}
      \label{fig:stable pulse approx}
     \end{subfigure}
     \hfill
\begin{subfigure}[b]{.3\textwidth}
\begin{center}
\begin{tabular}{||c c||} 
 \hline
 Unstable & Conjugate \Tstrut \\
Eigenvalues ($\lambda$)  & Points ($x$) \Bstrut \\ 
 \hline
 \hline
 N/A & N/A \Tstrut\Bstrut \\ 
 \hline
 \hline
\end{tabular}
\end{center}
\label{fig:stable pulse phi = 0 evals}
\caption{Unstable eigenvalues and conjugate points.}
     \end{subfigure}
         \begin{subfigure}[b]{\textwidth}
         \centering
          \includegraphics[width=\textwidth]{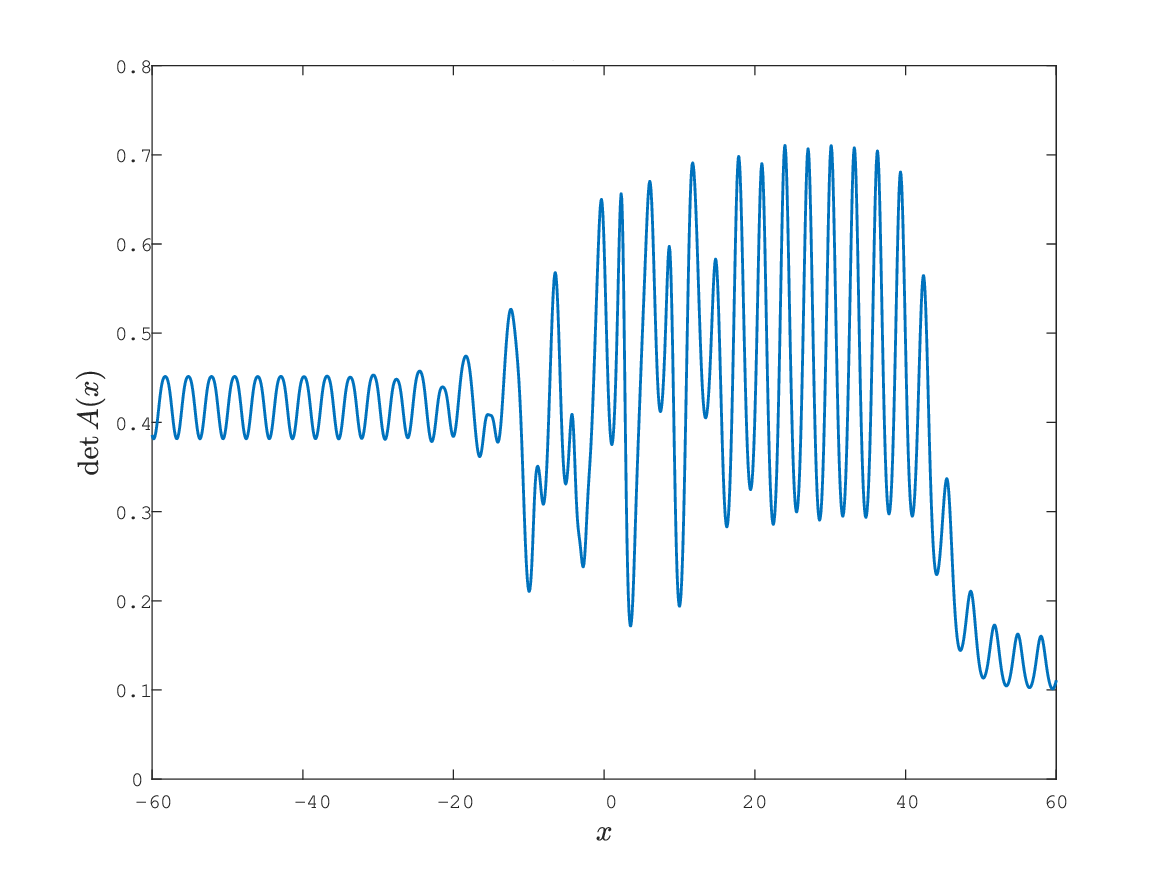}
         \caption{$\det A(x)$.}
         \label{fig:stable pulse det}
     \end{subfigure}
        \caption{Information for the pulse with $\phi = 0$, $\mu = 0.20$ and $\nu = 1.6$: $\varphi_0(x; 1.6, .20)$.  }
        \label{fig: stable pulse}
\end{figure}

\begin{remark} The behavior at the tails of these plots are determined by the asymptotic coefficient matrix $B_\infty$ given in \eqref{eq: B_inf matrix}. Since $B_\infty$ has complex conjugate eigenvalues, we observe periodic behavior. Indeed, the trajectory $\mathbb{E}_{-}^u(x)$ may be seen to be be a heteroclinic orbit within the space of Grassmanians, limiting to the point $ \mathbb{E}_{-}^u(-\infty)$ as $x \to -\infty$, and limiting to a periodic orbit as $ x \to + \infty$. 
\end{remark}

Another visualization of intersections \ADD{between} the unstable subspace and the sandwich plane can be done via the Pl\"ucker coordinates. Given two vectors in $\R^4$,  $v_1 = \sum_{j = 1}^4 a_j e_j$ and $v_2 = \sum_{j = 1}^4 b_je_j$, the wedge product $v_1 \wedge v_2$ is given by 
$$ v_1 \wedge v_2 = \sum_{1 \leq i <j \leq 4} (a_ib_j - a_jb_i)(e_i \wedge e_j).$$
 The wedge product defines an alternating tensor product of vectors (i.e. $ e_i \wedge e_j = - e_j \wedge e_i$) and the $k^{th}$ exterior product of $\R^n$  is commonly denoted by  $\bigwedge^k(\R^n) $ and has dimension $n \choose k$. %For example $ \bigwedge^2(\R^4) $ is spanned by $e_1\wedge e_2$,  $e_1\wedge e_3$, $e_1\wedge e_4$, $e_2\wedge e_3$, $e_2\wedge e_4$ and $e_3\wedge e_4$. 

	One defines the Pl\"ucker embedding of a Grassmannian by identifying the vectors spanning a subspace with their projectivized wedge product. For the case of, say, 2-dimensional subspaces of $\R^4$, that is 
	\[
	span \{ v_1,v_2\} \mapsto [ v_1 \wedge v_2]
	\]
	where $ [ v_1 \wedge v_2]$ denotes the equivalence class $ \alpha ( v_1 \wedge v_2)$ for all $ \alpha \in \R$. 

\begin{definition} 
		Consider the 2-dimensional subspace $ W \subseteq \R^4$ spanned by the vectors $v_1 = \sum_{j = 1}^4 a_j e_j$ and $v_2 = \sum_{j = 1}^4 b_je_j$. Given the vectors $v_1$ and $v_2$, for all $ 1 \leq i < j \leq 4$ we define their unnormalized Pl\"ucker coordinate $\tilde{P}_{ij} $, and their normalized Pl\"ucker coordinates $P_{ij} $  as:
\begin{align}
\tilde{P}_{ij} &= a_ib_j - a_jb_i,&
P_{ij} &= \frac{\tilde{P}_{ij}}{\|\tilde{P}\|}.
\end{align}
where $ \|\tilde{P}\|^2  = \sum_{1 \leq i <j \leq 4} \tilde{P}_{ij}^2$.
\end{definition}  
\noindent 
Note that if $ \{v_1,v_2\}$ and $ \{w_1, w_2\}$ span the same subspace and have Pl\"ucker coordinates $ P_{ij}$ and $Q_{ij}$ respectively, then either $ P_{ij} = Q_{ij}$ or  $ P_{ij} =- Q_{ij}$ for all $ 1 \leq i < j \leq 4$.

The Pl\"ucker coordinates provides us with a convenient coordinate system to  plot the trajectory of $\mathbb E^u_-(x)$ through $\Lambda(2)$ for the three pulse solutions we have considered thus far.   Note that the  Lagrangian Grassmannian  $\Lambda(2)$ is 3 dimensional and $ \bigwedge^2(\R^4) $  is 6 dimensional, so some amount of information will be discarded \cite{Lee}. In Figures \ref{fig:L_pm intro results unstable} and \ref{fig:L_pm intro results stable} we plot the first three Pl\"ucker coordinates $P_{12},P_{13},P_{14}$ of the trajectory of $\mathbb E^u_-(x)$. 

  To depict the train of $\ell_*^{sand}$, suppose that $ \ell  \cap \ell_*^{sand} \neq \{0\}$ and $ \ell = span \{ v_1,v_2\} $ with $v_1 \in \ell_*^{sand}$.
  Without loss of generality (due to the Lagrangian property) we may write 
  \begin{align} \label{eq:PluckerSpanningVectors}
  	v_1 &= \begin{pmatrix}
  		0\\
  		\cos \theta \\
  		\sin \theta \\
  		0
  	\end{pmatrix}
  	&
  	v_2 &= \begin{pmatrix}
  		\alpha \cos \theta \\
  		-\beta  	\sin \theta \\
  		\beta \cos  \theta \\
  		\alpha 	\sin \theta 
  	\end{pmatrix}
  \end{align}
for $ \alpha ,\beta , \theta \in \R$, 
and their wedge product may be computed to be 
\[
v_1\wedge v_2 =   \alpha  \left( -    \cos^2\theta\,  e_1 \wedge e_2 
+   \sin \theta \cos \theta  ( e_2  \wedge e_4 - e_1 \wedge e_3)+   \sin^2 \theta  e_3 \wedge e_4 \right) + \beta  
e_2 \wedge e_3.
\]
Hence their P\"ucker coordinates may be computed to be  
$$(P_{12}, P_{13}, P_{14},P_{23},P_{24},P_{34}) =
\frac{\alpha}{\sqrt{\alpha^2 + \beta^2 }} (
-    \cos^2\theta,
-\sin \theta \cos \theta,
0,
\frac{\beta}{\alpha},
\sin \theta \cos \theta,
\sin^2 \theta ).$$
We note that $ \ell(x)$ has a non-simple crossing with $\ell_*^{sand}$ at $ x_*$ if and only if $ \ell(x_*) = \ell_*^{sand}$. In this case  $\ell(x_*)$ is spanned by vectors $ v_1,v_2$ as in \eqref{eq:PluckerSpanningVectors} with $ \alpha =0$, and its Pl\"ucker coordinates are given by 
\[
(P_{12}, P_{13}, P_{14},P_{23},P_{24},P_{34}) =
(
0,0,0,\pm1,0,0).
\]

In Figures \ref{fig:L_pm intro results unstable}-\ref{fig:L_pm intro results stable} we display $\mathbb E^u_-(x)$ and the train of  $\ell_*^{sand}$  projected into the first three Pl\"ucker coordinates. After projecting into the first three coordinates and applying the double angle formula, we see that the projection of the train of $\ell_*^{sand}$ is given by  
\begin{align*}
	(P_{12}, P_{13}, P_{14} ) &=
	\frac{-\alpha}{2\sqrt{\alpha^2 + \beta^2 }} (
	1+\cos 2 \theta  ,
	\sin 2 \theta ,
	0 )  
\end{align*}
for $ \alpha , \beta ,\theta \in \R$. 
This may be succinctly described as the   region    $(x,y,0)$ for which $ (x\pm \frac{1}{2})^2 +y^2 \leq \frac{1}{4}$.

For the parameters $\mu = 0.05$ and $\nu = 1.6$, we plot the pulse solution in standard coordinates and the trajectory of the unstable subspace in these Plu\"cker coordinates. The train of $\ell_*^{sand}$ is depicted in purple. For the below two solutions, the trajectories intersect the sandwich plane a number of times that is consistent with the number of zeros in $\det A(x)$. This is not surprising, since these figures are a different visualization of the same analysis presented in Figures \ref{fig:unstable pulse phi = 0}, \ref{fig: unstable pulse phi = pi} and \ref{fig: stable pulse}. 
Note here that the sandwich plane, and the point of a potentially non-simple crossing, is projected onto the point $(0,0,0)$. In all cases the trajectory can be seen to have only simple crossings, as its Pl\"ucker coordinates remain bounded away from the origin.

\begin{figure}[H]
     \centering

    \begin{subfigure}[b]{\textwidth}
        \centering
        \begin{subfigure}[b]{0.39\textwidth}
        \centering
        \includegraphics[width=\textwidth]{figures/0_pulse_intro.eps}
        \end{subfigure}
         \begin{subfigure}[b]{0.6\textwidth}
         \centering
         \includegraphics[width=\textwidth]{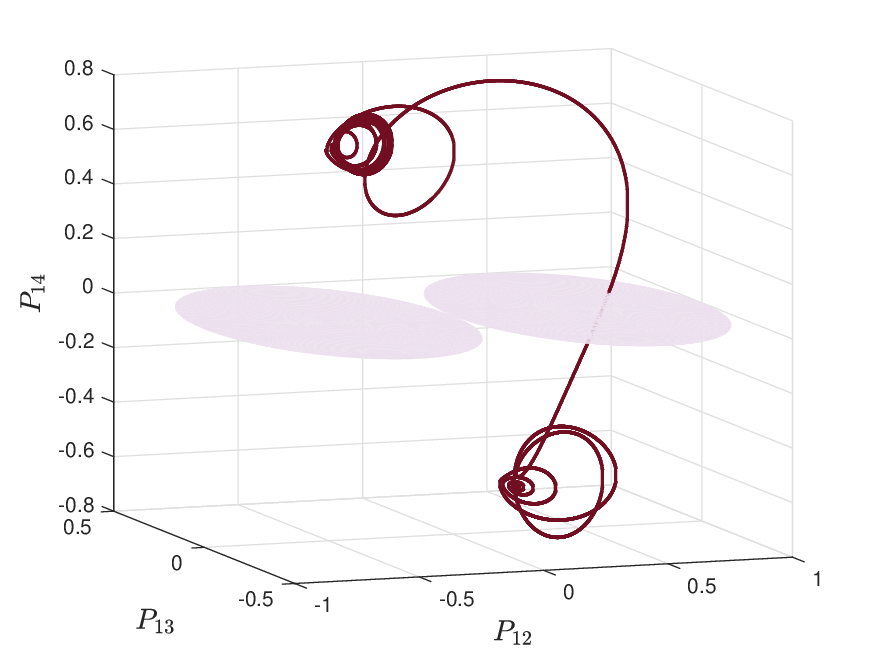}
          \end{subfigure}
     \caption{$\varphi_0(x;1.6, 0.05)$}
     \end{subfigure}
    \begin{subfigure}[b]{\textwidth}
        \centering
        \begin{subfigure}[b]{0.39\textwidth}
        \centering
        \includegraphics[width=\textwidth]{figures/pi_pulse_intro.eps}
        \end{subfigure}
         \begin{subfigure}[b]{0.6\textwidth}
         \centering
         \includegraphics[width=\textwidth]{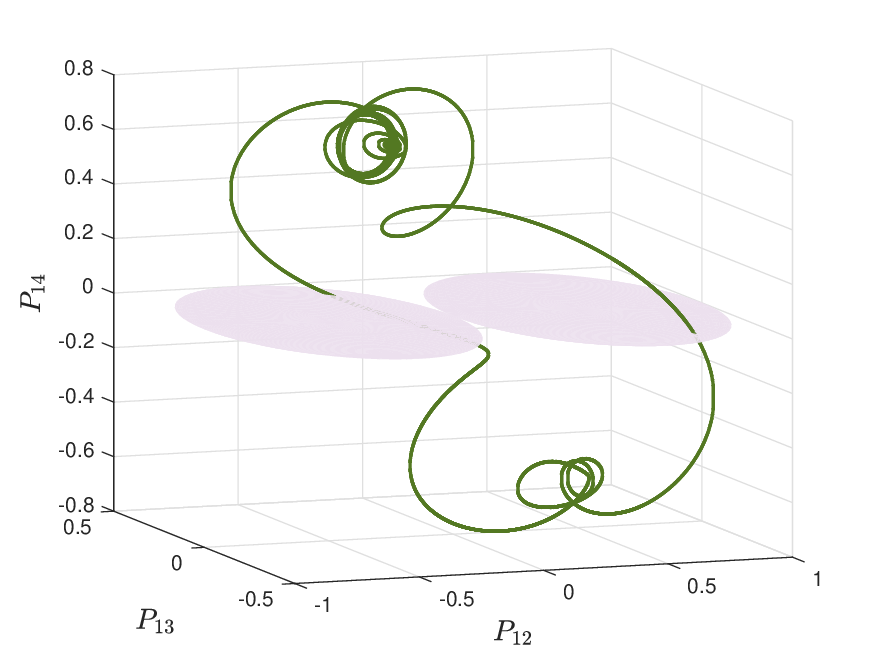}
          \end{subfigure}
     \caption{$\varphi_\pi(x;1.6, 0.05)$}
     \end{subfigure}
    \caption{Example symmetric pulse profiles (left) and the trajectory of the unstable subspace through $\Lambda(2)$ for each associated variational equation (right) for $\mu = 0.05$ and $\nu = 1.6$. The train of $\ell_*^{sand}$ is depicted in purple. Both of these pulse solutions are unstable.}
        \label{fig:L_pm intro results unstable}
\end{figure}

Additionally, we plot the unstable subspace for the variational equation corresponding to the stable pulse with parameter values $\mu = 0.20$ and $\nu = 1.6$. In this case, we see that the entire trajectory of $\mathbb E^u_-(x)$ is bounded away from $\ell_*^{sand}$ in $\Lambda(2)$. 

\begin{figure}[H]
     \centering
    \begin{subfigure}[b]{\textwidth}
        \centering
        \begin{subfigure}[b]{0.39\textwidth}
        \centering
        \includegraphics[width=\textwidth]{figures/stable_pulse_intro.eps}
        \end{subfigure}
         \begin{subfigure}[b]{0.6\textwidth}
         \centering
         \includegraphics[width=\textwidth]{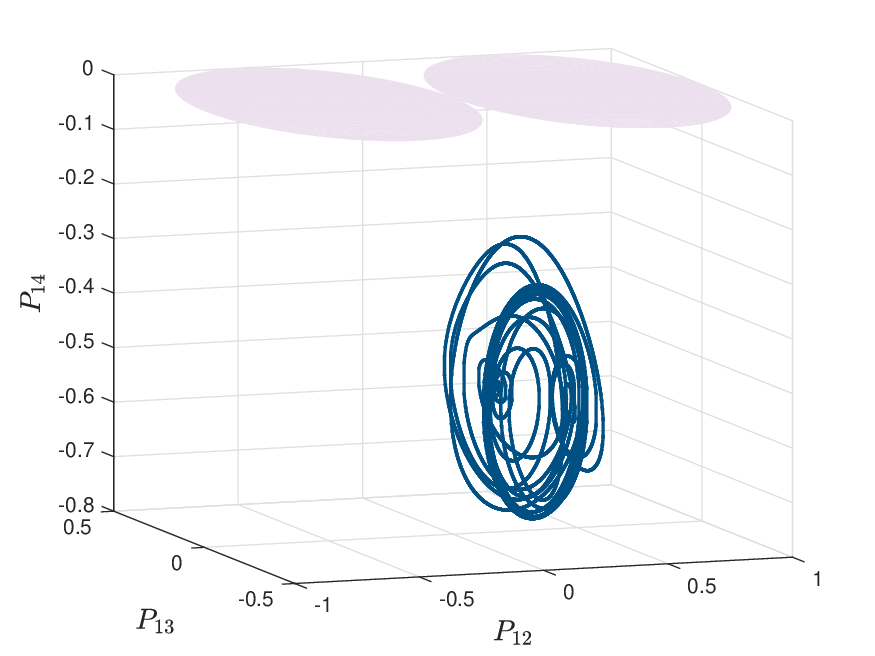}
          \end{subfigure}
     \caption{$\varphi_0(x;1.6, 0.20)$}
     \end{subfigure}
    \caption{Example symmetric pulse profile (left) and the trajectory of the unstable subspace through $\Lambda(2)$ for the variational equation (right) with $\mu = .20$ and $\nu = 1.6$. The train of $\ell_*^{sand}$ is depicted in purple. This pulse solution is spectrally stable.  }
        \label{fig:L_pm intro results stable}
\end{figure}

\begin{remark}
The Matlab code for the results in this section is available online \cite{BJPcode}.
\end{remark}

%%%%%%%%%%%%%%%%%%%%%%%%%%%%%%%%%%%%%%%%%%%%%%%%%%%%%%%%%%%%%%%%%%%%%%%%%%%%%%%%%%%%%%%%%%%%%%%
%%%%%%%%%%%%%%%%%%%%%%%%%%%%%%%%%%%%%%%%%%%%%%%%%%%%%%%%%%%%%%%%%%%%%%%%%%%%%%%%%%%%%%%%%%%%%%%
%%%%%%%%%%%%%%%%%%%%%%%%%%%%%%%%%%%%%%%%%%%%%%%%%%%%%%%%%%%%%%%%%%%%%%%%%%%%%%%%%%%%%%%%%%%%%%%
%%%%%%%%%%%%%%%%%%%%%%%%%%%%%%%%%%%%%%%%%%%%%%%%%%%%%%%%%%%%%%%%%%%%%%%%%%%%%%%%%%%%%%%%%%%%%%%

%\subfile{future_directions}

\section{Future Directions}\label{S:future}
In this work, we have shown that the number of unstable eigenvalues of pulse solutions to the Swift-Hohenberg equation corresponds to the number of conjugate points, as long as Hypothesis \ref{hyp:degeneracy} holds. Additionally, we have derived a novel numerical method for computing the stability of these solutions. However, the computations in Section \ref{S:numerics} do not provide a mathematical proof of the spectral stability of pulse solutions. We are currently developing a method using validated numerics to do so, similar to that done in \cite{MBJJ21}. In order to construct a computer assisted proof of the stability of a given solution $\varphi$, we must:  
\begin{enumerate}[i)]
\item Compute a numerical approximation of $\varphi$ with error bounds;
\item determine the interval $[-L_-, L_+]$ on which the conjugate points occur; 
\item approximate $\mathbb E^u_-(-L_-; 0)$ and integrate forward to $\mathbb E^u_-(L_+; 0)$, with rigorous error bounds; 
\item count the number of times $\mathbb E^u_-(x;0)$ intersects $\ell_*^{sand}$ for $x \in [-L_-, L_+]$,  and show that each intersection is 1 dimensional. 

%\item {\color{red} show that each intersection $\mathbb E^u_-(x;0) \cap \ell_*^{sand}$ is one-dimensional.}
\end{enumerate}
Computing $\mathbb E^u_-(x;0)$ on this interval is complicated by the presence of a so-called external resonance in the spatial eigenvalues. This resonance was also presence in the work of \cite{MBJJ21}, but the computer assisted proof in this work considered lower order approximations, so this was not an issue. We seek a higher order approximation of the subspace $\mathbb E^u_-(x; 0)$ and so we must address this resonance. 

It would be very interesting to try to remove the need for Hypothesis \ref{hyp:degeneracy}. This could be done in one of two ways. First, one could attempt to prove that this hypothesis in fact always holds for solutions of the Swift-Hohenberg equation. Our numerics for the example pulses that we consider here suggests that this might indeed be true. Additionally, one could also seek to extend the results from \S\ref{section: maslov index for nonregular crossings} to situations in which, at a degenerate crossing, the entire crossing form need not be degenerate. We suspect that proving such results should be possible, and indeed this would be an interesting direction to pursue, even if it were to turn out that, for Swift-Hohenberg, Hypothesis \ref{hyp:degeneracy} did always hold, since this might not be the case for other PDEs.

Additionally, the use of higher order crossing forms could be used to prove a result relating the number of conjugate points to the number of unstable eigenvalues of the fourth order Hamiltonian system 
\begin{equation}
u_t = -\partial_x^4 u - P \partial_x^2 u - u + u^2
\end{equation}
where $P$ is a real parameter. In a certain parameter region, this system admits stationary localized solutions, which are studied in \cite{Buffoni95,Buffoni96}. In \cite{Buffoni96}, these pulse solutions are classified via a sequence of integers $n(\ell_1, \dots, \ell_{n-1})$, called the BCT classification, where $n$ is the number of major local extrema and $\ell_i$ are related to the number of zeros between the major local extrema. Numerical experiments suggest that the number of conjugate points associated to these solutions can be computed via the information encoded in the BCT classification \cite{Chardard09}. If this were to be proved rigorously, this result could be viewed as an extension of Sturm-Liouville theory to fourth order systems. Rigorously justifying the results of \cite{Chardard09} will be the subject of future work.

%%%%%%%%%%%%%%%%%%%%%%%%%%%%%%%%%%%%%%%%%%%%%%%%%%%%%%%%%%%%%%%%%%%%%%%%%%%%%%%%%%%%%%%%%%%%%%%
%%%%%%%%%%%%%%%%%%%%%%%%%%%%%%%%%%%%%%%%%%%%%%%%%%%%%%%%%%%%%%%%%%%%%%%%%%%%%%%%%%%%%%%%%%%%%%%
%%%%%%%%%%%%%%%%%%%%%%%%%%%%%%%%%%%%%%%%%%%%%%%%%%%%%%%%%%%%%%%%%%%%%%%%%%%%%%%%%%%%%%%%%%%%%%%
%%%%%%%%%%%%%%%%%%%%%%%%%%%%%%%%%%%%%%%%%%%%%%%%%%%%%%%%%%%%%%%%%%%%%%%%%%%%%%%%%%%%%%%%%%%%%%%

%\subfile{appendix}

\appendix
\section{Appendix}\label{S:appendix}
In the appendix, we provide the proofs of several technical lemmas and two explicit examples for which we compute the Maslov index with respect to a particular reference plane. One example has regular crossings and the other has nonregular crossings. 

%------------------------------------------------
\subsection{Proofs of technical results}
%------------------------------------------------

%------------------------------------------------
\subsubsection{Proof of Lemma \ref{lemma: B infinity hyperbolic}}\label{Subsec: proof of introduction lemma}
%------------------------------------------------
Here we present the proof of Lemma \ref{lemma: B infinity hyperbolic}. This was relevant in the introduction when discussing the asymptotic behavior of the first order system we obtained by linearizing about the solution $\varphi$ given by \eqref{eq: linearized first order}. \\

\noindent\textbf{Lemma \ref{lemma: B infinity hyperbolic}.}
\textit{Let $B_\infty$ be as defined in \eqref{eq: B_inf matrix}. Then 
\begin{enumerate}[a)]
\item $B_\infty(\lambda)$ is hyperbolic for all $\lambda \geq 0$. 
\item The matrix $B_\infty$ depends analytically on $\lambda$. Also, there are positive constants $K_B$ and $C_B$, independent of $\lambda$, such that 
\begin{equation}\label{eq: matrix conversion B in hypothesis}
\|B(x;\lambda) - B_\infty(\lambda)\| \leq K_Be^{-C_B|x|} \text{ as } x \to \pm \infty.
\end{equation}
\end{enumerate}}

\begin{proof} ~\\
\begin{enumerate}[a)]
\item Using Hypothesis \ref{hyp: nonlinearity}, a direct computation shows that the eigenvalues of $B_\infty(\lambda)$ are given via $\pm \gamma_{i}$, $i = 1,2$ with $\gamma_1 = \bar \gamma_2$ and $\text{Re} \gamma_i > 0.$ We can express $\gamma_1$ as $\gamma_1 = \sqrt{r} e^{i \theta/2}$,
where 
$$ \theta = \arctan \left(-\sqrt{\lambda - f'(0)}\right) \quad \text{ and } r = \sqrt{1 + \lambda - f'(0)}. $$
Since $f'(0) < 0$ and $\lambda \geq 0$, we have that $\theta \in (\frac{\pi}{2}, \pi)$ and $r > 1$.  
\item Since $B_\infty$ is hyperbolic and $\lim_{x \to \pm} \varphi(x) = 0$, it is a standard result that $B(x;\lambda)$ converges to its asymptotic limits at an exponential rate \cite{promislow}. 
\end{enumerate}
\end{proof}

%------------------------------------------------
\subsubsection{Proof of Lemma \ref{lemma: evecs of B}}\label{Subsec: proof of B evecs lemma}
%------------------------------------------------

Here we present the proof of Lemma \ref{lemma: evecs of B}. This was a technical result that was used in Proposition \ref{prop: zero bottom}, in which we proved that there are no intersections of $\mathbb E^u_-(-\infty, \lambda)$ and $\ell_*^{sand}$. \\

\noindent\textbf{Lemma \ref{lemma: evecs of B}.}
\textit{Let $B_\infty$ be as defined in \eqref{eq: B_inf matrix}. If 
$$ \theta = \arctan \left(-\sqrt{\lambda - f'(0)}\right) \quad \text{ and } r = \sqrt{1 + \lambda - f'(0)},$$
then basis vectors for the real unstable and stable eigenspaces of $B_\infty(\lambda)$ are given via 
\begin{equation}\label{eq: B_inf unstable evecs}
R^u_1 = \begin{pmatrix} \frac{1}{r}\cos \theta \\ 1 \\ \left(\frac{2}{\sqrt r} + \sqrt{r} \right) \cos \left(\frac{\theta}{2} \right)  \\ \frac{1}{\sqrt r} \cos \left(\frac{\theta}{2} \right)
\end{pmatrix}, \quad  R^u_2 = \begin{pmatrix} -\frac{1}{r}\sin \theta \\ 0 \\ \left(\sqrt{r} - \frac{2}{\sqrt r} \right) \sin \left(\frac{\theta}{2} \right)  \\ -\frac{1}{\sqrt r} \sin \left(\frac{\theta}{2} \right)
\end{pmatrix},
\end{equation} 
\begin{equation}\label{eq: B_inf stable evecs}
R_1^s = \begin{pmatrix} \frac{1}{r}\cos \theta \\ 1 \\ \left(-\frac{2}{\sqrt{r}}  - \sqrt r \right) \cos \left( \frac{\theta}{2}\right) \\
-\frac{1}{\sqrt r}\cos \left(\frac{\theta}{2}\right)
\end{pmatrix}, \quad  R_2^s = \begin{pmatrix} -\frac{1}{r}\sin \theta \\ 0 \\ \left(\frac{2}{\sqrt{r}} - \sqrt r \right) \sin \left( \frac{\theta}{2}\right) \\
\frac{1}{\sqrt r}\sin \left(\frac{\theta}{2}\right)
\end{pmatrix}. 
\end{equation}}
\begin{proof}
Define $B_\infty(\lambda)$ by the block structure
$$B_{\infty}(\lambda) = \begin{pmatrix} 0 & 0 & 0 & 1 \\ 0 & 0 & 1 & -2 \\ -\lambda - 1 + f'(0) & 0 & 0 & 0 \\ 0 & 1 & 0 & 0 
\end{pmatrix} : = \begin{pmatrix} 0 & P \\ Q & 0
\end{pmatrix}.$$

We will write the eigenvalues and eigenvectors of $B_\infty(\lambda)$ using information about the eigenvalues and eigenvectors of the matrix $PQ$. Via the properties of block matrices, we know that 
$$ \det\left(B_\infty(\lambda) - \gamma I_{4}\right) = \det \left( \gamma^2 I_2 - PQ \right) .$$
Denote the eigenvalues of $PQ = \begin{pmatrix} 0 & 1 \\ -\lambda -1 + f'(0) & -2
\end{pmatrix}$ by $\alpha_k$ with associated eigenvectors $v_k$, $k = 1,2$. We can write these explicitly as 
\[\begin{matrix} \alpha_1 = -1 + \sqrt{f'(0) - \lambda} = -1 + i\sqrt{\lambda - f'(0)}  := re^{i\theta}, & & v_1 = \begin{pmatrix} 1/\alpha_1 \\ 1
\end{pmatrix}= \begin{pmatrix} \frac{1}{r}e^{-i\theta}  \\ 1
\end{pmatrix} \\
\alpha_2 = -1 - \sqrt{f'(0) - \lambda} = -1 - i\sqrt{\lambda - f'(0)} :=re^{-i\theta}, && v_2 = \begin{pmatrix}1/\alpha_2 \\ 1
\end{pmatrix} := \begin{pmatrix} \frac{1}{r}e^{i\theta} \\ 1
\end{pmatrix}.
\end{matrix}\]

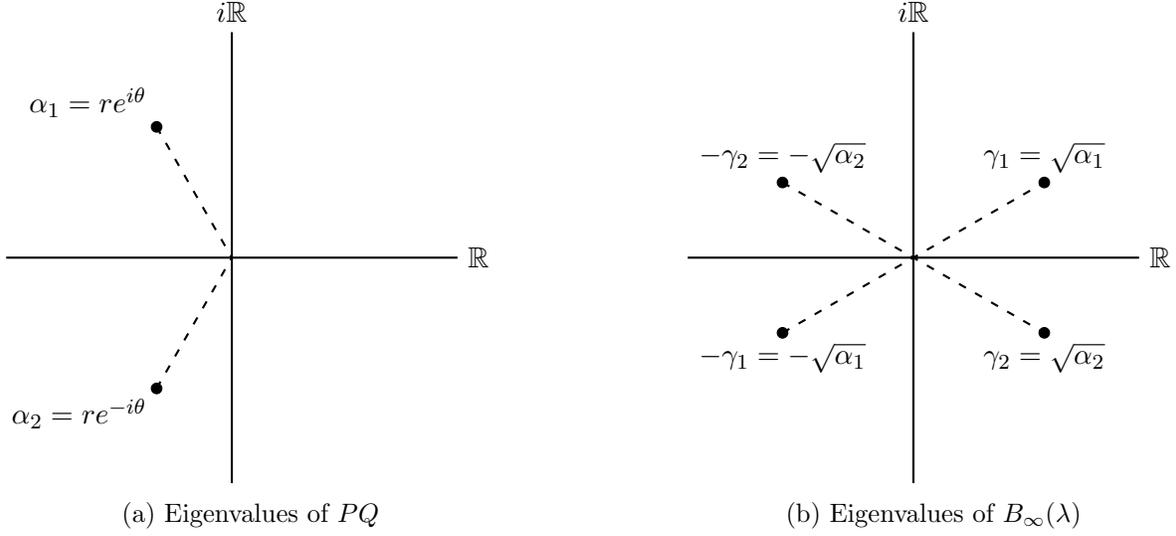
\begin{figure}[H]
    \centering
    \begin{subfigure}[t]{0.45\textwidth}
        \centering
        \begin{tikzpicture}[scale=1]
                  \draw[thick, -] (-3, 0) -- (3, 0) node[right] {$\R$};
                  \draw[thick, -] (0, -3) -- (0, 3) node[above] {$i\R$};
                  \draw[thick, dash pattern={on 3pt off 4pt}] (-2*.5, -2*.87) -- (0, 0);
                  \draw[thick, dash pattern={on 3pt off 4pt}] (-2*.5, 2*.87) -- (0, 0);
                 
                  \filldraw (-2*.5,-2*.87) circle (2pt) node[below left] {$\alpha_2 = re^{-i\theta} $};
                  \filldraw (-2*.5,2*.87) circle (2pt) node[ above left] {$\alpha_1 = re^{i\theta}$};
                  \end{tikzpicture}
        \caption{Eigenvalues of $PQ$}
    \end{subfigure}
    \hfill
    \begin{subfigure}[t]{0.45\textwidth}
        \centering
              \begin{tikzpicture}[scale=1]
                  \draw[thick, -] (-3, 0) -- (3, 0) node[right] {$\R$};
                  \draw[thick, -] (0, -3) -- (0, 3) node[above] {$i\R$};
                  \draw[thick, dash pattern={on 3pt off 4pt}] (-2*.87, -2*.5) -- (2*.87, 2*.5);
                  \draw[thick, dash pattern={on 3pt off 4pt}] (-2*.87, 2*.5) -- (2*.87, -2*.5);

                  \filldraw (-2*.87, 2*.5) circle (2pt) node[above] {$-\gamma_2 = -\sqrt{\alpha_2}$};
                  \filldraw (2*.87, 2*.5) circle (2pt) node[above] {$\gamma_1 = \sqrt{\alpha_1}$};
                  \filldraw (-2*.87,-2*.5) circle (2pt) node[below] {$-\gamma_1 = - \sqrt{\alpha_1}$};
                  \filldraw (2*.87,-2*.5) circle (2pt) node[below] {$\gamma_2 = \sqrt{\alpha_2}$};
                  \end{tikzpicture}
        \caption{Eigenvalues of $B_\infty(\lambda)$} \label{fig:timing1}
    \end{subfigure}
    \caption{Eigenvalues of $B_\infty(\lambda)$ understood through its block structure.}\label{fig: B_inf eigenvalues}
\end{figure}

Then the eigenvalues of $B_\infty(\lambda)$ are given by $\pm \gamma_k = \pm \sqrt{\alpha_k}$, $k = 1,2$. Relabeling as in Figure \ref{fig: B_inf eigenvalues}, define the vectors
\begin{equation}\label{eq: complex eigenvectors Binf block form}
V_k^s  = \begin{pmatrix} v_k \\ -\gamma_k P^{-1} v_k \end{pmatrix} \text{ and } V_k^u  = \begin{pmatrix} v_k \\ \gamma_k P^{-1} v_k \end{pmatrix}.
\end{equation}
A direct computation shows that $(V^{s}_k, -\gamma_k)$ and $(V^{u}_k, \gamma_k)$ are eigenpairs for $B_\infty(\lambda)$ for $k = 1,2$.

Finally, we can use the above characterizations of $\gamma_k$ and $v_k$ to write the eigenvectors of $B_\infty(\lambda)$. Taking real and imaginary parts yields the desired result.

\end{proof}

%-----------------------------------------------
\subsubsection{Representing Lagrangian Planes as Graphs: The technical details}
%-----------------------------------------------
Here we characterize the structure of the matrix whose graph is a given Lagrangian plane as is discussed at the beginning of Section \ref{section: grassmann}, which was used in the proof of Theorem \ref{thm: higher order quad form RS extension}. 

Let $\ell(t)$ be a Lagrangian path defined for $t \in [a,b]$ with frame matrix $L(t)$. Fix $t_0 \in \R$ and let $\mathcal W$ be a Lagrangian plane such that $\ell(t_0) \cap \mathcal W = \{0\}$. Then, there exists an $\epsilon > 0$ and a $2n \times 2n$ matrix $A(t): \ell(t_0) \to \mathcal W$ such that for all $v \in \ell(t_0)$, $v \in \ker A(t_0)$ and
\begin{equation}
v + A(t)v \in \ell(t), \qquad  \ t \in (t_0 - \epsilon, t_0 + \epsilon).
\end{equation}

\begin{figure}[H]{}
  \centering
        \begin{tikzpicture}[scale=1]
        \draw[thick, -] (-3, 0) -- (3, 0) node[right] {$\ell(t_0)$};
        \draw[thick, -] (0, -3) -- (0, 3) node[above] {$\mathcal W$};
        \draw[thick, -] (-2.4, -2.4) -- (2.4, 2.4) node[above] {$\ell(t)$};

        \draw[dashed] (1.5, 1.5) -- (1.5, 0)  node[below] {$v$};
        \draw[dashed] (1.5, 1.5) -- (0, 1.5) node[left] {$A(t)v$};

        \filldraw [black] (1.5,1.5) circle (1.5pt) node[right] {$u = v + A(t)v$};
        \filldraw [black] (1.5,0) circle (1.5pt) node[right] {};    
        \filldraw [black] (0,1.5) circle (1.5pt) node[right] {};    
        \end{tikzpicture}
    \caption{Representing $\ell(t)$ as a graph for $t$ near $t_0$.} \label{fig:timing1}
\end{figure}
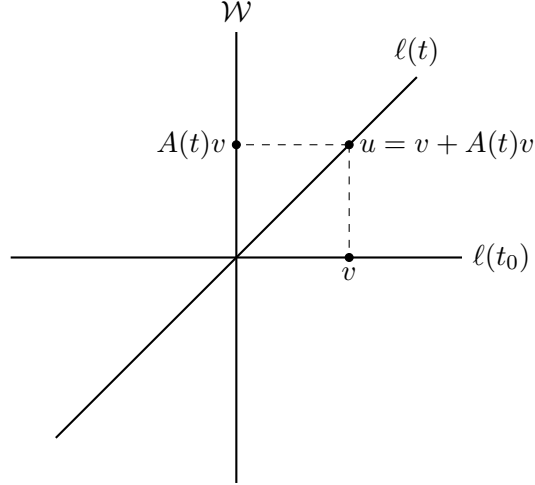

More generally, we have the following result: 
\begin{lemma}[\cite{Lee, piccione}]\label{lemma: represent subspaces as graphs} Let $\mathcal V,  \mathcal W, \mathcal P \in \Lambda(n)$ such that $\mathcal P \cap \mathcal W= \{0\}$ and $\mathcal V \oplus \mathcal W = \R^{2n}$. Then, $\mathcal P$ can be written as the graph of a unique linear map $A: \mathcal V \to \mathcal W$.
\end{lemma}

We characterize the structure of the matrix $A$ in the following proposition. 
\begin{proposition}\label{prop: Structure of A}
Suppose we are in the setting of Lemma \ref{lemma: represent subspaces as graphs} and let $\Pi_{\mathcal V}$ denote the projection matrix onto $\mathcal V$. Then, 
\begin{enumerate}[a)]
    \item $\Pi_{\mathcal V}^T JA \Pi_{\mathcal V}$ is symmetric; 
    \item if in addition $JA = (JA)^T$, then $A$ has the structure 
    \begin{equation}\label{eq: block structure of A}
A = \begin{pmatrix} B & C \\ D & -B^T
\end{pmatrix}.
\end{equation}
with $B, C, D \in \R^{n \times n}$ and $C,D$ symmetric.
\end{enumerate}
\end{proposition}
\begin{proof} ~\\
\begin{enumerate}[a)]
\item For $v_1, v_2 \in \mathcal V$, we have that $v_1 + Av_1$ and $v_2 + Av_2$ are both in $\mathcal W$. Therefore, 
\begin{align*} 0 & = \omega(v_1 + Av_1, v_2 + Av_2) \\
& = \omega(v_1, v_2) + \omega(v_1, Av_2) + \omega(Av_1, v_2) + \omega(Av_1, Av_2) \\
& = \omega(v_1, Av_2) + \omega(Av_1, v_2) \\
& = \langle v_1, JAv_2\rangle + \langle v_1, A^TJ v_2 \rangle \\
& = v_1^T \big(JA + A^TJ \big)v_2.
\end{align*}
If $\Pi_{\mathcal V}$ denotes the projection operator onto $\mathcal V$, then for all $x,y \in \R^{2n}$, we have 
$$x^T \Pi_{\mathcal V}^T\big(JA + A^TJ\big)\Pi_{\mathcal V} y = 0.$$
Therefore, the Lagrangian property tells us that $\Pi^T_{\mathcal V} JA \Pi_{\mathcal V} = -\Pi_{\mathcal V}^TA^TJ\Pi_{\mathcal V}$.
On the other hand, we can compute 
$$\big( \Pi^T_{\mathcal V} JA \Pi_{\mathcal V} \big)^T = \Pi^T_{\mathcal V} A^TJ^T \Pi_{\mathcal V}  = -\Pi^T_{\mathcal V} A^TJ \Pi_{\mathcal V}.$$
By putting these pieces together, we see that 
\begin{equation}\label{eq: entire mat sym}
\big( \Pi_{\mathcal V}^T JA \Pi_{\mathcal V} \big)^T =  \Pi^T_{\mathcal V} JA \Pi_{\mathcal V} ,
\end{equation}
so this matrix is symmetric. 
\item Now suppose that $JA = (JA)^T$. If we represent $A$ with the block structure $B,C,D,E$, we see that
\begin{align*}
JA & = \begin{pmatrix} 0 & -I_n \\ I_n & 0 
\end{pmatrix} \begin{pmatrix} B & C \\ D & E
\end{pmatrix} = \begin{pmatrix} -D & -E \\ B & C
\end{pmatrix} \\
(JA)^T & = \begin{pmatrix} -D^T & B^T \\ -E^T & C^T
\end{pmatrix}.
\end{align*}
Note that if we assume $A$ has such a structure and derive a matrix $A$ whose graph is $\mathcal P$, then we have found the unique graph representation of $\mathcal P$. 
\end{enumerate}
\end{proof}

%-------------------------------------------------
\subsection{An Explicit Example with the Maslov Index}\label{subsec: explicit example}
%-------------------------------------------------

Here we present an explicit example using a path of Lagrangian subspaces whose basis vectors are constructed from solutions to an ODE. We compute the Maslov index with respect to the sandwich plane using the crossing form. This example serves two purposes. The first is to illustrate the subtlety in expressing a path of Lagrangian subspaces as a graph that is discussed in Remark \ref{rmk: graphing error}. The second is to concretely show the relationship between the higher order crossing forms and the derivatives of the eigenvalues of the graph matrix (see Section \ref{S:eval-motion}).

Consider the autonomous differential equation 
\begin{equation}\label{eq: ODE}
\dot q = \underbrace{\begin{pmatrix} 0 & 0 & 1 & 0 \\ 0 & 0 & 0 & 0 \\ 0 & - 1 & 0 & 0 \\ -1 & 0 & 0 & 0 \end{pmatrix}}_{:= B} q = \underbrace{\begin{pmatrix} 0 & 0 & 1 & 0 \\ 0 & 0 & 0 & 1 \\ -1 & 0 & 0 & 0 \\0 & -1 & 0 & 0  \end{pmatrix}}_{:=J} \underbrace{\begin{pmatrix} 0 & 1 & 0 & 0 \\ 1 & 0 & 0 & 0 \\ 0 & 0 & 1 & 0 \\ 0 & 0 & 0 & 0   \end{pmatrix}}_{:=C}q.
\end{equation}
This system has a general solution given by 
\begin{equation}\label{eq: general solution to ODE}
\begin{split} q_{1}(s) &= -\frac{C_{2}}{2}s^{2}+ C_{3}s + C_{1} \\
q_{2}(s) &= C_{2}\\
q_{3}(s) &= -C_{2}s + C_{3}\\
q_{4}(s) &= \frac{C_{2}}{6}s^{3}- \frac{C_{3}}{2}s^{2}- C_{1}s + C_{4}. 
\end{split}
\end{equation}

%---------------------------------------------
\subsubsection{Regular Crossings}
%-----------------------------------------------

Consider the Lagrangian plane 
$$\ell_{1}(s) = \text{colspan} \begin{pmatrix} | & | \\
v_1(s) & v_2(s) \\ | & | 
\end{pmatrix} :=\text{colspan} \begin{pmatrix} -\frac{1}{2}s^{2}+ 2s & -3s^{2}+ 1 \\ 1 & 6  \\ 2 - s & -6s  \\  \frac{1}{6}s^{3}- s^{2} & s^{3}- s + 2 
\end{pmatrix}.$$
Note that this pair of solutions to the above ODE satisfies the Lagrangian condition. The first spanning vector intersects $\ell_*^{sand}$ at $s = 0$. We will use the crossing form to understand how this crossing contributes to the Maslov index. To do so, we first want to write $\ell_1(s)$ as a graph. Suppose that $A_1(s): \ell_1(0) \to (\ell_*^{sand})^\perp$ where 
$$ \big(\ell_*^{sand}\big)^\perp = \text{colspan} \begin{pmatrix} 1 & 0 \\ 0  & 0 \\ 0 & 0 \\ 0 & 1 \end{pmatrix}.$$
%Note that $\ell_1(0) \cap (\ell_*^{sand})^\perp = \{0\}$. 

\begin{remark}
It is clear from this example that there is no such matrix $A_1(s): \ell_1(0) \to (\ell_*)^\perp$ that leaves $v_1(0)$ invariant, meaning $v_1(0) + A_1(s)v_1(0) = v_1(s)$. The problem arises in the third component. To see this explicitly, we can write 
$$ \begin{pmatrix} 0\\ 1 \\ 2 \\ 0 \end{pmatrix} + A_1(s)\begin{pmatrix} 0\\ 1 \\ 2 \\ 0 \end{pmatrix} = \begin{pmatrix} f_1(s) \\ 1 \\ 2 \\ f_2(s)
\end{pmatrix}  \neq \begin{pmatrix} -\frac{1}{2}s^{2}+ 2s  \\ 1  \\ 2 - s  \\  \frac{1}{6}s^{3}- s^{2}
\end{pmatrix}.$$
There are no such $f_1(s), f_2(s)$ that could satisfy this relationship. 
\end{remark}

We seek a matrix $A_1(s): \ell_1(0) \to \mathcal (\ell_*)^\perp$ having the structure given in Proposition \ref{prop: Structure of A} such that for $k_1, k_2 \in \R$,
\begin{equation}\label{eq: simple regular crossing matrix equation}
k_{1}\begin{pmatrix} 0\\ 1 \\ 2 \\ 0 \end{pmatrix} + k_{2}\begin{pmatrix} 1 \\ 6 \\ 0 \\ 2
\end{pmatrix} + A_1(s) \begin{pmatrix} k_{2}\\ k_{1}+ 6k_{2}\\ 2k_1 \\ 2k_2
\end{pmatrix} = c_{1} \begin{pmatrix} -\frac{1}{2}s^{2} + 2s \\ 1 \\ 2-s \\ \frac{s^{3}}{6} - s^2 \end{pmatrix}+ c_{2} \begin{pmatrix} -3s^{2}+ 1 \\ 6 \\ -6s \\ s^{3}- s + 2
\end{pmatrix}.
\end{equation}
By focusing our attention on the middle two components of $A_1(s)\big(k_1 v_1(0) + k_2v_2(0)\big)$, which are $0$ by assumption, we derive
%With this set up, can use that the middle two components of  in \eqref{eq: simple regular crossing matrix equation} map into $0$ to see that 
% \begin{align*} k1 + 6k_2 & = c_1 + 6c_2\\
% 2k_1 & = c_1(2-s) -6sc_2.
% \end{align*}
% Solving this system of equations, we find that 
\begin{align*}
c_1 & = \left(\frac{s}{2} + 1\right)k_1 + (3s)k_2 \\
c_2 & = \left(-\frac{1}{12}s\right)k_1 + \left(1 - \frac{1}{2}s \right)k_2.
\end{align*}
There is one degree of freedom in this system and we find that a solution is given by
\begin{equation}
A_1(s) = \begin{pmatrix} 3s^2 - \frac{1}{2} & 0 & \frac{1}{4}s^2 + \frac{23}{24}s & 0 \\ 
a_{21} & 0 & 0 & -\frac{1}{2}a_{21} \\
-36s^2 + 6s + 2a_{21} & 6s^2 - s & -\frac{1}{2}(6s^2 - s) & -a_{21} \\ 
6s^2 - s & \frac{1}{12}(4s^3 - 11s^2 - 2s) & 0  & 0
\end{pmatrix}.
\end{equation}
%Note that $a_{21}$ is free and that there was one degree of freedom in \eqref{eq: ex1 small system}, meaning there are multiple matrices satisfying \eqref{eq: simple regular crossing matrix equation}. 
Let $\Pi_1$ be the projection onto $\ell_1(0) \cap \ell_* = v_1(0)$. This matrix can be calculated explicitly as 
$$\Pi_1 = v_1(0)\left(v_1(0)^Tv_1(0)\right)^{-1}v_1(0)^T = \frac{1}{5}\begin{pmatrix} 0 & 0 & 0 & 0 \\ 0 & 1 & 2 & 0 \\ 0 & 2 & 4 & 0 \\ 0 & 0 & 0 & 0
\end{pmatrix}. $$

%It is not the case that $A_1(s)$ is the unique operator that determines the graph of $\ell_1(s)$ but rather $A_1(s)\Pi_{\ell_1(0)}$ is the unique operator. See Lemma \ref{lemma: represent subspaces as graphs}. 
Using $\Pi_1$ given above, we can explicitly compute the matrix that determines the crossing form as 
$$ \Pi_1 JA_1(s)\Pi_1 =  -\frac{s}{150} \left(4s^2 + 23s + 48 \right) \begin{pmatrix} 0 & 0 & 0 & 0 \\ 0 & \frac{1}{2}  & 1 & 0 \\ 0 & 1  & \frac{1}{2} & 0 \\ 0 & 0 & 0 & 0
\end{pmatrix}.$$

The eigenvalues of $\Pi_1JA_1(s)\Pi_1$ determine whether this crossing contributes positively or negatively to the Maslov index. The first order crossing form (Definition \ref{def: first order quadratic form}) captures the sign of the first derivative of this eigenvalue. We can calculate that $\Pi_1JA_1(s)\Pi_1$ has three uniformly $0$ eigenvalues and one $s$-dependent eigenvalue given by 
\begin{equation}
\lambda(s) =  - \frac{1}{15}s^3 - \frac{23}{60}s^2 - \frac{4}{5}s.
\end{equation}

The crossing form calculates the sign of this derivative. We can calculate that 
$$Q^{(1)}(v_1(0))  = \langle v_1(0), J A_1'(0)v_1(0) \rangle = -4.$$
% \begin{align*} Q^{(1)}(v_1(0)) & = \langle v_1(0), J A_1'(0)v_1(0) \rangle \\ & = \begin{pmatrix} 0 & 1 & 2 & 0 
% \end{pmatrix} J \begin{pmatrix} \frac{23}{12} \\ 0 \\ 0 \\ -\frac{1}{6}
% \end{pmatrix} \\
% & = -4.
% \end{align*}
On the other hand, we can differentiate directly to see that 
$$ \lambda'(0) = -\frac{4}{5}.$$
Note that $\lambda'(0)$ and $Q^{(1)}(v_1(0))$ are off by a factor of $5$. If we want to directly apply Theorem \ref{thm: lancaster eigenvalue derivatives} to conclude $\lambda'(0) = Q^{(1)}(v_1(0))$, we first need to normalize $v_1(0)$.

We can also use \eqref{eq: simple regular crossing matrix equation} to calculate the crossing form without explicitly knowing the structure of $A_1(s)$. With $k_1 = 1$ and $k_2 = 0$, we have that 
\begin{align*} c_1(s) & = \left(\frac{s}{2} + 1 \right) \\
c_2(s) & = -\frac{1}{12}(s).
\end{align*}
Using this, we can calculate  
\begin{align*}
Q^{(1)}(v_1(0)) 
& = \big\langle v_1(0), J A_1'(0)v_1(0) \big\rangle \\
& = \left\langle v_1(0), J\big(c_1'(0)v_1(0) + c_1(0)v_1'(0) + c_2'(0) v_2(0) + c_2(0)v_2'(0)\big)v_1(0)\right\rangle \\
%& = c_1'(0)\underbrace{\langle v_1(0), Jv_1(0) \rangle}_{=0} + c_1(0)\langle v_1(0), Jv_1'(0) \rangle \\
%& \qquad \qquad + c_2'(0)\underbrace{\langle v_1(0),J  v_2(0) \rangle}_{=0} + \underbrace{c_2(0)}_{=0}\langle v_1(0),J v_2'(0)\rangle \\
& = \langle v_1(0), Jv'_1(0)\rangle \\ 
& = \langle v_1(0), JBv_1(0) \rangle \\
%& = \begin{pmatrix} 0  1 \\ 2 \\ 0 \end{pmatrix}
& = -4. 
\end{align*}

%-----------------------------------------------
\subsubsection{Non-regular Crossings}
%-----------------------------------------------
We now consider a path of Lagrangian planes with a nonregular crossing. Consider the Lagrangian plane 
$$\ell_{2}(s) =\text{colspan} \begin{pmatrix} | & | \\ v_1(0) & v_2(0) \\ | & |  \end{pmatrix} := \text{colspan} \begin{pmatrix} -\frac{s^{2}}{2} & -3s^{2}+ 1 \\ 1 & 6 \\ -s & -6s \\ \frac{s^{3}}{6} & s^{3}- s
\end{pmatrix}$$
satisfying \eqref{eq: ODE}. The first spanning vector intersects $\ell_*^{sand}$ at $s = 0$. 
%We have that $\ell_{2}(0) \cap \ell_{*}\neq \{0\}$ since 
%$$\ell_{2}(0) = \begin{pmatrix} 0 & 1 \\ 1 & 6 \\ 0 & 0 \\ 0 & 0 
%\end{pmatrix}.$$
%This is a simple crossing since the dimension of intersection is one dimensional. 

As in the previous example, we want to calculate the contribution of this intersection to the Maslov index with respect to the sandwich plane by using the crossing form. For $s$ sufficiently small, we can write $\ell_2(s)$ as the graph of a matrix $A_2(s) \in \R^{4\times 4}$ having the structure given in Proposition \ref{prop: Structure of A}. Furthermore, we want $A_2(s): \ell_{2}(0) \to \ell_2(0)^\perp$. Note that we do not seek $A_2(s): \ell_2(0) \to (\ell_*)^\perp$ because $(\ell_*)^\perp \cap \ell_2(0) \neq \{0\}$. With these assumptions $A_2(s)$ must satisfy the equation
\begin{equation}\label{eq: simple nonregular crossing matrix equation}
\begin{split}
k_{1}\begin{pmatrix} 0\\ 1 \\ 0 \\ 0 \end{pmatrix} + k_{2}\begin{pmatrix} 1 \\ 6 \\ 0 \\ 0
\end{pmatrix} + A_2(s) & \begin{pmatrix} k_{2}\\ k_{1}+ 6k_{2}\\ 0 \\ 0
\end{pmatrix} = c_{1} \begin{pmatrix} -\frac{s^{2}}{2} \\ 1 \\ -s \\ \frac{s^{3}}{6} \end{pmatrix}+ c_{2} \begin{pmatrix} -3s^{2}+ 1 \\ 6 \\ -6s \\ s^{3}- s
\end{pmatrix}\\
% A_2(s) & = \begin{pmatrix} a_{11} & a_{12} & a_{13} & a_{14} \\ a_{21} & a_{22} & a_{14} & a_{24} \\ a_{31} & a_{32} & -a_{11} & -a_{21} \\ 
% a_{32} & a_{42} & -a_{12} & -a_{22}
% \end{pmatrix}
\end{split}
\end{equation}

%Since the image of $A_2(s)$ lies in $\ell(0)^\perp$, we see that for any $k_1, k_2$ 
%\begin{align*} a_{11}k_2 + a_{12}(k_1 + 6k_2) & = 0 \\ 
%a_{21}k_2 + a_{22}(k_1 + 6k_2) & = 0 \\
%\implies a_{11} = a_{12} = a_{21} = a_{22} & = 0.
%\end{align*} 
% Using the first and second components of \eqref{eq: simple nonregular crossing matrix equation} 
% \begin{align*} k_2 & = k_2 + a_{11}k_2 + a_{12}(k_1 + 6k_2)_2   \\
% & = -c_1\frac{s^2}{2} + c_2(-3s^2 + 1) \\
% k_1 + 6k_2 & = k_1 + 6k_2 + a_{11}k_2 + a_{12}(k_1 + 6k_2)  \\
% & = c_1 + 6c_2.
% \end{align*}
we find that $c_1$ and $c_2$ satisfy  
\begin{align*}
c_1 & = k_1 - 3s^2 (k_1 + 6k_2) \\
c_2 & = k_2 + \frac{1}{2}s^2(k_1 + 6k_2).
\end{align*}

% Finally, looking at the third and fourth components of \eqref{eq: simple nonregular crossing matrix equation}, we have the system of equations 
% \begin{align*} a_{31} + 6a_{32} & = -6s \\ 
% a_{32} & = -s \\ 
% a_{32} + 6a_{42} & = -s + 2s^2 \\ 
% a_{42} & = - \frac{s^3}{3}.
% \end{align*} 

By substituting these equations in for $c_1$ and $c_2$ and solving this system of equations, we find a matrix satisfying \eqref{eq: simple nonregular crossing matrix equation} to be
$$A_2(s) = \begin{pmatrix} 0 & 0 & a_{13} & a_{14} \\ 0 & 0 & a_{14} & a_{24} \\
0 & -s & 0 & 0  \\
-s & -\frac{s^{3}}{3} & 0 & 0 \end{pmatrix}.$$
Note that $a_{13}, a_{14}$ and $a_{24}$ are free. 

Define $\Pi_2$ to be the projection matrix onto $\ell_* \cap \ell_2(0)$. We can compute that 
$$\Pi_2 = \begin{pmatrix}
0 & 0 & 0 & 0 \\ 0 & 1 & 0 & 0 \\ 0 & 0 & 0 & 0 \\ 0 & 0 & 0 & 0
\end{pmatrix}.$$

Then, the sign of the crossing form tracks how the sign of the nonzero eigenvalue of $\Pi_2 J A_2(s) \Pi_2$ changes as $s$ increases through $0$. Because $A_2(s)$ is explicitly known, we can calcualate this as 
$$ \Pi_2 J \ADD{A_2(s)} \Pi_2 = \begin{pmatrix} 0 & 0 & 0 & 0 \\ 0 & -\frac{1}{3}s^3 & 0 & 0 \\ 0 & 0 & 0 & 0 \\ 0 & 0 & 0 & 0 
\end{pmatrix}.$$
This matrix has $3$ zero eigenvalues and one $s$-dependent eigenvalue given by 
$$ \lambda_1(s) = -\frac{1}{3}s^3.$$
Since this eigenvalue is decreasing through $0$, we expect the intersection of $\ell_2(s)$ and $\ell_*^{sand}$ at $s = 0$ to contribute negatively to the Maslov index. 
%This eigenvalue changes from positive to negative as $s$ increases though $0$, so this crossing contributes negatively to the Maslov index. 

The vector $v = \begin{pmatrix} 0 & 1 & 0 & 0 
\end{pmatrix}$ is in $\ell_* \cap \ell_2(0)$. If we compute the first order crossing form, we find 
\begin{align*} Q^{(1)}(v) & = \big\langle v, JA_2'(0)v \big\rangle \\
& = \begin{pmatrix} 0  & 0 &  0 & 1 \end{pmatrix} 
 \begin{pmatrix} 0 \\ 1 \\ -1 \\ -s^2
\end{pmatrix} \bigg|_{s = 0} \\ 
& = -s^2 \bigg|_{s = 0} \\
& = 0.
\end{align*}
Proceeding similarly, we find that 
\begin{align*} Q^{(2)}(0) & = \big\langle v_1(0), J A_2''(0)v_1(0) \big\rangle \\ 
& = 0 \\ \\ 
Q^{(3)}(0) & = \big \langle v_1(0), JA_2^{(3)}(0)v_1(0) \big \rangle \\ 
& = -2.
\end{align*}
If we did not know $\lambda(s)$ explicitly, we could use Theorem \ref{theorem: maslov gen tang}
 to conclude that this crossing contributes negatively to the Maslov index. Theorem \ref{thm: lancaster eigenvalue derivatives} tells us that this means $\lambda(s)$ must be decreasing at $s = 0$.

If we didn't explicitly know $A_2(s)$ (as is the case with many applications to differential equations), we can use the general structure in \eqref{eq: simple nonregular crossing matrix equation} to calculate the sign of the third crossing form. If we set $k_1 = 1$ and $k_2 = 0$, then $c_1(s) = 1-3s^2$ and $c_2(s) = \frac{1}{2}s^2$. In this case, we have 
\begin{align*} Q^{(3)}(v_1(0)) & = \left\langle v_1(0), JA^{(3)}(0)v_1(0) \right\rangle \\
& = \underbrace{\langle v_1(0), J \left( c_1'''(0)v_1(0) + c_2'''(0)v_2(0) \right) \rangle}_{=0 \text{ via Lag. prop.}} + 3 \langle v_1(0), J \left(c_1''(0)v_1'(0) + c_2''(0)v_2'(0) \right) \rangle \\
& \qquad + \underbrace{3\langle v_1(0), J(c_1'(0)v_1''(0) + c_2'(0)v_2''(0)) \rangle}_{0 \text{ since }c_1'(0) = c_2'(0) = 0} + \langle v_1(0), J (c_1(0)v_1'''(0) + \underbrace{c_2(0)}_{=0}v_2'''(0)) \rangle \\
& = 3 c_1''(0)\underbrace{\langle v_1(0), J v_1'(0) \rangle}_{=Q^{(1)}(v_1(0)) = 0} + 3c_2''(0) \langle v_1(0), J v_2'(0)  \rangle + \langle v_1(0), J (c_1(0)v_1'''(0) \rangle \\
& = 3 \langle v_1(0), JBv_2(0) \rangle + \langle v_1(0), JBBBv_1(0) \rangle \\
& = -3 + 1 \\
& = -2.
\end{align*}

\bibliographystyle{alpha}
\bibliography{entire-paper}

\end{document}